\documentclass[a4paper,11pt]{amsart}

\usepackage[latin1]{inputenc} 
\usepackage[dvips]{graphics}
\usepackage[matrix,arrow,tips,curve]{xy}
\usepackage[english]{babel}
\usepackage{amsmath}
\usepackage{amssymb}

\usepackage{mathrsfs}

\usepackage{enumerate}

\newtheorem{thm}{Theorem}[section]
\newtheorem{lemma}[thm]{Lemma}
\newtheorem{corollary}[thm]{Corollary}
\newtheorem{proposition}[thm]{Proposition}
\newtheorem{thmdefi}[thm]{Theorem - Definition}

\newtheorem*{thm*}{Theorem}

\theoremstyle{definition}
\newtheorem{definition}[thm]{Definition}

\newtheorem{example}[thm]{Example}
\newtheorem{remark}[thm]{Remark}
\newtheorem{prg}[thm]{}

\newcommand{\ph}{\varphi}
\newcommand{\w}{\widetilde}
\newcommand{\ma}{\mathcal}
\newcommand{\la}{\longrightarrow}
\newcommand{\ol}{\mathcal{O}}
\newcommand{\wi}{\widehat}
\newcommand{\pr}{\mathbb{P}}
\newcommand{\Q}{\mathbb{Q}}

\newcommand{\R}{\mathbb{R}}
\newcommand{\Z}{\mathbb{Z}}
\newcommand{\N}{\mathcal{N}_1}
\newcommand{\Nu}{\mathcal{N}^1}
\newcommand{\Gr}{\operatorname{Gr}}
\newcommand{\dom}{\operatorname{dom}}

\newcommand{\im}{\operatorname{Im}}
\newcommand{\NE}{\operatorname{NE}}
\newcommand{\Exc}{\operatorname{Exc}}
\newcommand{\Supp}{\operatorname{Supp}}
\newcommand{\Lo}{\operatorname{Locus}}
\newcommand{\codim}{\operatorname{codim}}
\newcommand{\Eff}{\operatorname{Eff}}
\newcommand{\Nef}{\operatorname{Nef}}
\newcommand{\Chow}{\operatorname{Chow}}
\newcommand{\Mov}{\operatorname{Mov}}

\newcommand{\mov}{\operatorname{mov}}
\newcommand{\Hom}{\operatorname{Hom}}

\newcommand{\Hilb}{\operatorname{Hilb}}

\usepackage{etoolbox}
\patchcmd{\section}{\normalfont}{\normalfont\large}{}{}
\patchcmd{\subsection}{\bfseries}{\scshape\centering}{}{}
\patchcmd{\subsection}{-.5em}{.5em}{}{}

\makeatletter
\@namedef{subjclassname@2020}{\textup{2020} Mathematics Subject Classification}
\makeatother

\title{Fano $4$-folds with a small contraction}
\author{C.~Casagrande}
\address{Universit\`a di Torino,
Dipartimento di Matematica,
via Carlo Alberto 10,
10123 Torino - Italy}
\email{cinzia.casagrande@unito.it}
\date{May 18, 2022}
\calclayout 

\begin{document}
\maketitle
\begin{abstract}
Let $X$ be a smooth complex Fano 4-fold. We show that if $X$ has a small
elementary contraction, then $\rho_X\leq 12$, where $\rho_X$ is the Picard number of $X$. This
result is based on a careful study of the geometry of $X$, on which we give a lot
of information. We also show that in the boundary case $\rho_X=12$ an open subset
of $X$ has a smooth fibration with fiber $\pr^1$. 
Together with
previous results, this implies that if $X$ is a Fano 4-fold with $\rho_X\geq 13$, then every
elementary contraction of $X$ is divisorial and sends a divisor to a surface. 
The
proof is based on birational geometry and the study of families of rational
curves. More precisely the main tools are: the study of families of lines in
Fano 4-folds and the construction of divisors covered by lines, a detailed
study of fixed prime divisors, the properties of the faces of the effective
cone, and a detailed study of rational contractions of fiber type.
\end{abstract}
\renewcommand{\theequation}{\thethm}
\section{Introduction}
\noindent The classification of smooth, complex Fano varieties has been achieved up to
dimension $3$ and attracts a lot of attention also in higher dimensions.
Let us focus on dimension $4$, 
 the first open case: the context of this paper is the  study of Fano $4$-folds with ``large'' Picard number (e.g.\ $\rho\geq 6$) by means of birational geometry and families of rational curves, with the aim of gaining a good understanding of the geometry and behaviour of
these $4$-folds. 
The main result of this paper is the following.
\begin{thm}\label{main}
Let $X$ be a smooth Fano $4$-fold and $\rho_X$ its Picard number. If $X$ has a small elementary contraction, then $\rho_X\leq 12$.
\end{thm}
The proof of this result is based on a careful study
of the geometry of $X$, on which we give a lot of information.
 We also show an additional property in the boundary case $\rho_X= 12$, see Th.~\ref{main2}.

Apart from products of del Pezzo surfaces, to the author's knowledge all the known examples of Fano $4$-folds have $\rho\leq 9$; there is just one known family 
  with $\rho=9$, and it has small elementary contractions, see \cite{vb}. Thus we do not know whether the bound of Th.~\ref{main} is sharp.

Let us notice that if $X$ has an elementary contraction of fiber type, then $\rho_X\leq 11$ by \cite[Cor.~1.2(ii)-(iii)]{fanos}, and if $X$ has an elementary divisorial contraction sending a divisor to a point or to a curve, then $\rho_X\leq 5$ \cite[Rem.~2.17(1)]{blowup}. Therefore we have the following.
\begin{corollary}
Let $X$ be a smooth Fano $4$-fold with $\rho_X\geq 13$. Then every elementary contraction of $X$ is divisorial and sends the exceptional divisor to a surface.
\end{corollary}
This is the behaviour of  products of del Pezzo surfaces, indeed we expect that for $\rho_X$ large enough these should be the only Fano $4$-folds.

\bigskip

Let us now explain the content of the paper and strategy of the proof of Th.~\ref{main}. In the sequel $X$ is a smooth, complex Fano $4$-fold.

\medskip

\noindent {\bf Previous work.}
Our starting point is the results on the geometry of Fano $4$-folds developed in \cite{eff,blowup,fibrations}, in particular: the Lefschetz defect, the classification of fixed prime divisors,  and the structure of rational contractions of fiber type. 

\smallskip

\noindent {\sc The Lefschetz defect.}
As usual we denote by $\N(X)$  the real vector space of numerical equivalence classes of  one-cycles in $X$ with real coefficients, and for every closed subset $Z\subset X$, we denote by $\N(Z,X)$ the linear subspace of $\N(X)$ spanned by classes of curves contained in $Z$.
If $X$ is Fano, the Lefschetz defect ${\delta_X}$ of $X$, introduced in \cite{codim}, is defined as follows:
\stepcounter{thm}
\begin{equation}\label{Lefschetz}
  \delta_X:=\max\bigl\{\codim\N(D,X)\,|\,D\text{ is a prime divisor in }X\bigr\}.
  \end{equation}
If $X$ is not a product of surfaces and $\delta_X\geq 2$, then $\rho_X\leq 12$ (see Th.~\ref{codim} and \ref{delta2}), so that we can reduce to the case $\delta_X\leq 1$.

\smallskip

\noindent {\sc Fixed prime divisors.}  A fixed prime divisor is a prime divisor $D$ which is the stable base locus of the linear system
$|D|$. Thanks to the bounds on the Lefschetz defect, in \cite{eff} it is shown that when $\rho_X\geq 7$ there are four possible types of fixed prime divisors of $X$, called $(3,2)$, $(3,1)^{sm}$, $(3,0)^{sm}$, and $(3,0)^Q$, in relation to the associated elementary divisorial contractions; see \S \ref{secfixed} for more details. If $X$ has a fixed prime divisor of type $(3,0)^{sm}$, then   $\rho_X\leq 12$ by \cite{blowup} (see Th.~\ref{30}), so that we can exclude this case and focus on the remaining ones.

\smallskip

\noindent {\sc Rational contractions of fiber type.} A rational contraction of fiber type is a rational map $f\colon X\dasharrow Y$ that factors as sequence of flips $X{\dasharrow} X'$ followed by a contraction of fiber type $f'\colon X'\to Y$ (i.e.\ a surjective map with connected fibers, with $Y$ normal and projective, and $\dim Y\leq 3$, see Section \ref{notation}).
Rational contractions of fiber type of Fano $4$-folds are studied in detail in \cite{fibrations}, where in particular it is shown that if $X$ has a rational contraction onto a $3$-fold, and $X$ is not a product of surfaces, then $\rho_X\leq 12$ (see Th.~\ref{3fold}). 

\medskip

\noindent {\bf New results.} The new tools and results used for the proof of Th.~\ref{main} include: the construction of families of lines and of divisors covered by lines, the properties of the faces of the effective cone, a detailed study of fixed prime divisors of type $(3,1)^{sm}$ and $(3,0)^Q$, and a detailed study of rational contractions of fiber type onto surfaces and onto $\pr^1$.

\smallskip

\noindent {\sc Families of lines.} We define a line in $X$ as a rational curve $C$ such that $-K_X\cdot C=1$, and a family of lines  as a ``maximal'' irreducible subvariety $V$ of $\Chow(X)$ whose general member is a line, see Section \ref{lines};
we study such families.
By standard deformation theory one has $\dim V\geq 2$, and we show that if $\rho_X\geq 6$, then $\dim V=2$
(Th.~\ref{uno}), because a family of lines of larger dimension yields a prime divisor $D$ with small $\dim\N(D,X)$, contradicting the results on the Lefschetz defect. Moreover the curves of the family $V$ cover a surface or a prime divisor. 
Divisors covered by lines play a crucial role in the paper; when $\rho_X\geq 7$ such a divisor is either nef, or fixed of type $(3,2)$; moreover the covering family of lines is unique (Lemmas \ref{excplane} and \ref{chitarre}).

\smallskip

\noindent {\sc Movable and fixed faces of the effective cone.}
As usual we denote by  $\Nu(X)$ the real vector space of numerical equivalence classes of divisors  with real coefficients in $X$;
the cones
$\Eff(X)$ and $\Mov(X)$, of classes of effective and movable divisors respectively, are convex polyhedral cones in $\Nu(X)$.

We say that a 
(proper) face $\tau$ of $\Eff(X)$ is a {\em movable face} if the relative interior of $\tau$ intersects the movable cone $\Mov(X)$; there exists a movable face if and only if $X$ has some non-zero, non-big movable divisor. Given a movable face $\tau$ of $\Eff(X)$, we construct a rational contraction of fiber type $f\colon X\dasharrow Y$ such that $\rho_Y= \dim(\tau\cap\Mov(X))$ and $\rho_X\leq\dim\tau+\rho_F$, where $F$ is a general fiber of $f'\colon X'\to Y$ (where as above $f$ factors as a sequence of flips followed by the regular contraction $f'$) and it is smooth, Fano, with $\dim F\leq 3$, so that $\rho_F\leq 10$ by classification (see \S \ref{movable}).

On the other hand, we say that a face $\tau$ of $\Eff(X)$ is a {\em fixed face} if $\tau\cap\Mov(X)=\{0\}$. A fixed face is always simplicial, and is generated by classes of fixed prime divisors; these divisors are in very special relative positions (see \S \ref{fixedfaces} and \ref{secfixed}).

Thus when $\Eff(X)$ has a movable face of small dimension,  we get a good bound on $\rho_X$; in particular if $\Eff(X)$ has a one-dimensional movable face, then $\rho_X\leq 11$. Therefore
 we can reduce to the case where
every one-dimensional face of $\Eff(X)$ is fixed, and this yields a lot of fixed prime divisors in $X$.

\smallskip

\noindent {\sc Strategy of the proof.}
Let $f\colon X\to Y$ be a small elementary contraction, and assume that $\rho_X\geq 7$ and $\delta_X\leq 1$. We can also assume that $\Eff(X)$ is generated by classes of fixed prime divisors, and there exists one such divisor $D$ such that $D\cdot C<0$ for a curve $C\subset X$ contracted by $f$, so that $D$ contains the exceptional locus of $f$.

After the classification of fixed prime divisors, we have four possible types for $D$;
we can exclude that $D$ is of type $(3,0)^{sm}$. If $D$ were of type $(3,2)$, 
$\Exc(f)\subset D$ would yield $\dim\N(D,X)\leq 2$, which contradicts our assumptions on $\rho_X$ and $\delta_X$. Therefore 
 $D$ is of type $(3,1)^{sm}$ or $(3,0)^Q$; for this reason we focus on these two types of fixed divisors, which are the main characters of the paper (see Rem.~\ref{stress}).

 Given a fixed prime divisor $E\subset X$ of type $(3,1)^{sm}$ or $(3,0)^Q$, we construct from $E$
several families of lines, each  covering a prime divisor different from $E$; this is a key result that allows to produce many prime divisors covered by lines with several good properties.

 Using all these preliminary results, first we show that $\rho_X\leq 12$ when $X$ has a fixed prime divisor of type $(3,1)^{sm}$ (Th.~\ref{finalmente}). The proof is quite articulated and is contained in Section \ref{case31}; we refer the reader to the beginning of that section for an overview.

Then we are left to consider the case where $X$ has only fixed prime divisors of type $(3,2)$ and $(3,0)^Q$; this is treated in Section
 \ref{Q}, and again we refer the reader to the beginning of that section for an overview.

 A frequent 
strategy in the paper is to look for movable, non-big divisors on $X$, in order to construct a rational contraction of fiber type, and then use it to bound $\rho_X$. Such movable, non-big divisors are usually obtained using fixed prime divisors and/or prime divisors covered by lines.

\smallskip

\noindent {\bf Summary.}
Section \ref{notation} contains the notation and recalls the preliminary results on the birational geometry of Fano $4$-folds and on the Lefschetz defect.
Section \ref{lines} contains the results on families of lines and divisors covered by lines in Fano $4$-folds.
Section \ref{fixed} is devoted to fixed prime divisors; in particular in \S \ref{add1} and \ref{add2} we show several properties of  fixed prime divisors of type $(3,1)^{sm}$ and $(3,0)^Q$ that are needed in the sequel.
Section \ref{fiber} is devoted to rational contractions of fiber type; in particular in \S \ref{fumo} we show several results on rational contractions of fiber type onto surfaces and onto $\pr^1$, that are needed in the sequel. 

Section \ref{divlines} contains a key construction which, given a fixed prime divisor of type $(3,1)^{sm}$ or $(3,0)^Q$, produces many prime divisors covered by lines. Then we prove some properties of these divisors depending on the different settings.
Section \ref{secfiber} contains two results where we manage to construct a rational contraction of fiber type on $X$, and use it to bound $\rho_X$.
Finally Sections \ref{case31} and \ref{Q} contain the actual proof of Th.~\ref{main}, first considering the case where $X$ has a fixed prime divisor of type $(3,1)^{sm}$, and then the case where $X$ has only fixed prime divisors of type $(3,2)$ and $(3,0)^Q$.

\smallskip

\noindent {\bf Acknowledgements.} I thank the referee for useful comments.

{\footnotesize \tableofcontents}

\section{Notation and preliminaries}\label{notation}
\noindent If $\mathcal{N}$ is a finite-dimensional real vector space and $a_1,\dotsc,a_r\in \mathcal{N}$, $\langle a_1,\dotsc,a_r\rangle$ denotes the convex cone in $\mathcal{N}$ generated by $a_1,\dotsc,a_r$.
Moreover, for every $a\neq 0$, $a^{\perp}$ is the hyperplane orthogonal to $a$ in the dual vector space $\ma{N}^*$. A facet of a cone is a face of codimension one. If $\ma{C}\subset\ma{N}$ is a convex polyhedral cone, we denote by $\ma{C}^{\vee}\subset\ma{N}^*$ its dual cone.

\medskip

We refer the reader to \cite{hukeel} for the notion and basic properties of a {\bf Mori dream space}; we recall that smooth Fano varieties are Mori dream spaces by \cite[Cor.~1.3.2]{BCHM}.
We also refer
 to \cite{debarreUT,kollarmori} for the standard notions in birational geometry.

 Let $X$ be a projective, normal, and $\Q$-factorial Mori dream space.
 If $D$ is a divisor and $C$ is a curve in $X$, we denote by $[D]\in\Nu(X)$ and $[C]\in\N(X)$ their  classes, and we  set $D^{\perp}:=[D]^{\perp}\subset\N(X)$ and $C^{\perp}:=[C]^{\perp}\subset\Nu(X)$.
 
For every closed subset $Z\subset X$, we denote by $\N(Z,X)$ the linear subspace of $\N(X)$ spanned by classes of curves contained in $Z$.
 We will use the following simple property.
\begin{remark}\label{easy}
Let $X$ be a smooth projective variety, $Z\subset X$ a closed subset, and $D\subset X$ a prime divisor disjoint from $Z$. Then $\N(Z,X)\subseteq D^{\perp}$, because $C\cdot D=0$ for every curve $C\subset Z$.
\end{remark}

A movable divisor is an effective divisor $D$ such that the stable
 base locus of the linear system $|D|$  has codimension $\geq 2$. 
A {\bf fixed prime divisor} is a prime divisor $D$ which is the stable base locus of 
$|D|$, namely such that $h^0(mD)=1$ for every $m\in\Z_{>0}$.
We will consider the usual cones of divisors and of curves: $$\Nef(X)\subseteq\Mov(X)\subseteq\Eff(X)\subset\Nu(X),\qquad
\mov(X)\subseteq \NE(X)\subset\N(X),$$ where 
all the notations are standard except $\mov(X)$, which is the convex cone generated by classes of curves moving in a family covering $X$.
Since $X$ is a Mori dream space, all these cones are closed, rational and polyhedral. An extremal ray of $\NE(X)$ is a one-dimensional face of this cone.

A contraction $f\colon X\to Y$ is a surjective map, with connected fibers, where $Y$ is normal and projective. Given a divisor $D$ in $X$,
$f$ is $D$-negative  if $D\cdot C<0$ for every curve $C\subset X$ such that $f(C)=\{pt\}$. An extremal ray $R$ of $\NE(X)$ is $D$-negative if $D\cdot\gamma<0$ for $\gamma\in R$, $\gamma\neq 0$.
We do not assume that contractions or flips are $K$-negative, unless specified.  

A small $\Q$-factorial modification ({\bf SQM}) is a birational map $\ph\colon X \dasharrow X'$ which is an isomorphism in codimension one, where $X'$ is a normal and $\Q$-factorial projective variety; then $X'$ is  a Mori dream space too, and $\ph$ can be factored as a finite sequence of flips.

 A {\bf rational contraction} (also called a contracting rational map) is a rational map $f\colon X\dasharrow Y$ that can be factored as $X\stackrel{\ph}{\dasharrow} X'\stackrel{f'}{\to} Y$, where $\ph$ is a SQM and $f'$ is a contraction. As in the regular case, a rational contraction is: of fiber type if $\dim Y<\dim X$, elementary if $\rho_X-\rho_Y=1$, and elementary divisorial if $f'$ is an elementary divisorial contraction.
\begin{remark}\label{referee}
 Let $X$ be a projective, normal, and $\Q$-factorial Mori dream space, and $D$ a divisor  such that $[D]\in\Mov(X)$, $[D]\neq 0$. Then there is a prime divisor with class in $\R_{\geq 0}[D]$.

Indeed there is a SQM $X \dasharrow X'$ such that the transform $D'$ of $D$ in $X'$ is nef, and hence semiample. Therefore $|mD'|$ is base-point-free for $m\in\mathbb{N}$ large and divisible enough;  by Bertini the general member of $|mD'|$ is irreducible, unless the linear system yields a contraction $f\colon X'\to C$ onto a curve. In this last case we have  $\R_{\geq 0}[D']=f^*\Nu(C)$, thus the general fiber of $f$ is an irreducible divisor with class in $\R_{\geq 0}[D']$.
\end{remark}

\medskip

Assume now that $X$ is smooth of dimension $4$, and
let $f\colon X\to Y$ be an elementary divisorial contraction. We say that $f$ is:
\begin{enumerate}[--]
\item of type $(3,2)$ if $\dim(f(\Exc(f)))=2$;
\item of type $(3,1)^{sm}$ if $Y$ is smooth and $f$ is the blow-up of a smooth curve;
\item of type $(3,0)^{Q}$ if $\Exc(f)$ is isomorphic to an irreducible quadric $Q$, $f(\Exc(f))$ is a point, and $\mathcal{N}_{\Exc(f)/X}\cong \ol_Q(-1)$;
\item of type $(3,0)^{sm}$ if $Y$ is smooth and $f$ is the blow-up of a point.
\end{enumerate}
An {\bf exceptional plane} is a closed subset $L\subset X$ such that $L\cong\pr^2$ and $\ma{N}_{L/X}\cong\ol_{\pr^2}(-1)^{\oplus 2}$; we will denote by $C_L\subset L$ a curve corresponding to a line in $\pr^2$.
An {\bf exceptional curve} is a closed subset $\ell\subset X$ such that $\ell\cong\pr^1$ and $\ma{N}_{\ell/X}\cong\ol_{\pr^1}(-1)^{\oplus 3}$. We have $-K_X\cdot C_L=1$ and $-K_X\cdot\ell=-1$.
\begin{thm}[\cite{kawsmall}]\label{kawamata}
  Let $X$ be a smooth Fano $4$-fold and $f\colon X\to Y$ a small elementary contraction. Then $\Exc(f)$ is a disjoint union of exceptional planes.
  \end{thm}
\begin{lemma}[\cite{eff}, Rem.~3.6]\label{basic1}
Let $X$ be a smooth Fano 
 $4$-fold and $\ph\colon X\dasharrow \w{X}$ a SQM. Then
 $\w{X}$ is smooth and there are pairwise disjoint exceptional planes $L_1,\dotsc,L_r\subset X$ and exceptional curves $\ell_1,\dotsc,\ell_r\subset \w{X}$ such that $\ph$ factors as:
$$\xymatrix{&{\wi{X}}\ar[dl]_f\ar[dr]^g&\\
X\ar@{-->}[rr]^{\ph}&&{\w{X}}}$$
where $f$ is the blow-up of $L_1,\dotsc,L_r$, $g$ is the blow-up of $\ell_1,\dotsc,\ell_r$, and $E_i:=f^{-1}(L_i)=g^{-1}(\ell_i)\cong\pr^2\times\pr^1$ for every $i=1,\dotsc,r$.
Moreover:
\begin{enumerate}[$(a)$]
\item if $\w{\Gamma}\subset \w{X}$ is an irreducible curve different from $\ell_1,\dotsc,\ell_r$, with transforms  $\Gamma\subset X$ and $\wi{\Gamma}\subset\wi{X}$, then 
$-K_{\w{X}}\cdot\w{\Gamma}=-K_X\cdot \Gamma+\sum_{i=1}^{r}E_i\cdot\wi{\Gamma}>0$;
\item if $\w{\Gamma}$ intersects some exceptional curve, then $-K_X\cdot\Gamma<-K_{\w{X}}\cdot\w{\Gamma}$, so that  $-K_{\w{X}}\cdot\w{\Gamma}\geq 2$.
\end{enumerate}
\end{lemma}
\begin{lemma}[\cite{eff}, Rem.~3.7]\label{basic2}
Let $X$ be a smooth Fano 
 $4$-fold and $f\colon X\dasharrow Y$ a rational contraction. Then 
one can factor $f$ as $X\stackrel{\ph}{\dasharrow}X'\stackrel{f'}{\to} Y$, where $\ph$ is a SQM, $X'$ is smooth, and $f'$ is a $K$-negative contraction. Moreover $Y$ has rational singularities.
\end{lemma}
Finally let us recall the following results on the Lefschetz defect (see \eqref{Lefschetz}).
\begin{thm}[\cite{codim}, Th.~3.3 and Cor.~1.3]\label{codim}
Let $X$ be a smooth Fano $4$-fold which is not a product of surfaces. Then $\delta_X\leq 3$, and if $\delta_X=3$, then $\rho_X\leq 6$.
\end{thm}
\begin{thm}[\cite{cdue}, Th.~1.2]\label{delta2}
Let $X$ be a smooth Fano 4-fold with  $\delta_X=2$. Then  $\rho_X\leq 12$, and if 
$\rho_X\geq 7$, then $X$ has a rational contraction onto a $3$-fold.
\end{thm}
\begin{remark}[\cite{codim}, Ex.~3.1]\label{delta}
If $X\cong S_1\times S_2$ where $S_i$ are del Pezzo surfaces with $\rho_{S_1}\geq\rho_{S_2}$, then $\delta_X=\rho_{S_1}-1$, and $\rho_X\leq 2\delta_X+2$.
\end{remark}  
 \section{Lines in Fano $4$-folds}\label{lines}
 \noindent
 Let $X$ be a normal projective variety. We denote by $\Hom_{bir}(\pr^1,X)$ the Hilbert scheme of morphisms $\pr^1\to X$ which are birational onto their image, and by $\Hom_{bir}(\pr^1,X)^n$ its normalization; there is a natural morphism  $\Hom_{bir}(\pr^1,X)^n\to\Chow(X)$ \cite[Cor.~I.6.9]{kollar}.
By a family of rational curves in $X$ we mean an irreducible subvariety $V$ of $\Chow(X)$ which is the closure of the image of an irreducible component of 
$\Hom_{bir}(\pr^1,X)^n$, see \cite[II.2.11]{kollar}. For
 $v\in V$ general, the corresponding cycle is an irreducible and reduced rational curve, and every member of the family is an effective, connected one-cycle with rational components. 
We denote by $[V]\in\N(X)$ the numerical equivalence class of the general curve of the family. 
There is a universal family:
\stepcounter{thm}
\begin{equation}\label{universal}
  \xymatrix{\ma{C}\ar[d]_{\pi}\ar[r]^{e\quad} & {X}\\ V& }
  \end{equation}
 Moreover $\Lo V:=e(\ma{C})\subseteq X$ is the union of the curves of the family; it is an irreducible closed subset, and we say that $V$ is covering if $\Lo V=X$.
\begin{remark}\label{covering}
If $X$ is smooth and $V$ is covering, 
we have $-K_X\cdot [V]\geq 2$, indeed
the general member of $V$ corresponds to 
 a free curve \cite[4.10 and Example 4.7(1)]{debarreUT}.
\end{remark}
\begin{definition}\label{l}
Let $X$ be a Gorenstein normal projective variety. 
A {\bf family of lines} in $X$ is a family of rational curves such that $-K_X\cdot[V]=1$.

We say that a prime divisor $D\subset X$ is covered by a family of lines if there exists a family of lines $V$ with $\Lo V=D$.
\end{definition}
We note that, for a smooth Fano variety $X$, lines are usually defined in terms of the index $r_X$ of $X$: if $-K_X=r_XH$, then $V$ would be a family of lines when  $H\cdot [V]=1$. In this paper we are interested in  Fano $4$-folds $X$ of index $1$, indeed we will tipically assume that $\rho_X\geq 7$, while Fano $4$-folds with $r_X>1$  are classified and have $\rho_X\leq 4$, see \cite[Theorems 3.3.1 and 7.2.15]{fanoEMS}. Thus for simplicity we define lines directly as in Def.~\ref{l}. 

\medskip

If $V$ is a family of lines and $X$ is Fano, then every member of the family is irreducible and reduced, so that $V$ is an 
\emph{unsplit} family in the terminology of \cite[Def.~IV.2.1]{kollar}. 
\begin{remark}
Suppose that $X$ is smooth and Fano, and let $V$ be a family of lines in $X$.
If $Z\subset X$ is a closed subset, we set $\Lo V_Z$ to be the union of the curves of the family meeting $Z$,  namely: $$\Lo V_Z:=e\bigl(\pi^{-1}\bigl(\pi(e^{-1}(Z))\bigr)\bigr).$$ Thus $\Lo V_Z$ is a closed subset, and if it is non-empty, then
\stepcounter{thm}
\begin{equation}\label{lori}
\N(\Lo V_Z,X)=\N(Z,X)+\R[V],
\end{equation}
see \cite[Lemma 4.1]{occhettaGM}. If $Z=\{x\}$ is a point, we just write $\Lo V_x$, and set $V_x:=\pi(e^{-1}(x))\subseteq V$ for the subvariety of $V$ parametrizing curves containing $x$.
\end{remark}
\begin{example}[lines in products of surfaces]\label{linesproduct} Let $X=S\times T$ with $S$ and $T$ del Pezzo surfaces, and let $V$ be a family of lines in $X$.
It is not difficult to see that up to switching $S$ and $T$, the curves of the family have the form $C\times\{p\}$ where $C\subset S$ is a given rational curve with $-K_S\cdot C=1$, and $p$ varies in $T$, so that $\dim V=2$, $V\cong T$ and $\Lo V\cong C\times T$.

If $C$ is smooth, then it is a $(-1)$-curve in $S$, and $[V]$ generates an extremal ray of type $(3,2)$ of $\NE(X)$.

If $C$ is singular, then $\rho_S=9$, $C\in |-K_S|$, and $D=\Lo V$ is a nef prime divisor. 
\end{example}
\begin{thm}\label{uno}
Let $X$ be a smooth Fano $4$-fold and  $V$ a family of lines in $X$. Then $2\leq\dim V\leq 4$, and moreover:
\begin{enumerate}[$(a)$]
\item if $\dim V= 4$, then $\dim\Lo V=3$ and $\rho_X\leq 3$;
\item if $\dim V= 3$, then $\dim\Lo V=3$ and $\rho_X\leq 5$;
\item if $\dim V= 2$, then one of the following holds:
\begin{enumerate}[$(i)$]
\item $\dim\Lo V=3$ and $V_x$ is finite for $x\in\Lo V$ general;
\item $\dim\Lo V=2$ and $\Lo V=\Lo V_x$ for every $x\in\Lo V$.
\end{enumerate}
If there are two curves of the family which are disjoint, then we are in case $(i)$.
\end{enumerate}
\end{thm}
The bounds on $\rho_X$ in $(a)$ and $(b)$ are sharp, see Examples \ref{a} and \ref{b}.
\begin{proof}
The inequality $\dim V\geq 2$ follows from the standard lower bound on $\dim\Hom(\pr^1,X)$, see \cite[Th.~II.1.2 and Prop.~II.2.11]{kollar}. We also have
\stepcounter{thm} 
\begin{equation}\label{rel}
\dim\Lo V+\dim\Lo V_x =\dim V+2
\end{equation} 
for general $x\in \Lo V$, because the family is unsplit, see \cite[Prop.~IV.2.5]{kollar}, which yields $\dim\Lo V\geq 2$.

 On the other hand $\Lo V\subsetneq X$ by Rem.~\ref{covering}, thus $\dim\Lo V\leq 3$, and \eqref{rel} yields $\dim V\leq 4$.

\smallskip

Suppose that $\dim V=4$. By \eqref{rel} for general  $x\in \Lo V$ we have $\dim\Lo V+\dim\Lo V_x =6$, and $\dim\Lo V\leq 3$, thus $D=\Lo V$ is a prime divisor and $D=\Lo V_x$ for general $x\in D$. This implies that $\dim\N(D,X)=1$ by \eqref{lori}, hence $\rho_X\leq 3$ by \cite[Prop.~3.16]{fanos}.

\smallskip

Suppose that $\dim V=3$. By \eqref{rel} for general  $x\in \Lo V$ we have $\dim\Lo V+\dim\Lo V_x =5$, 
thus $D=\Lo V$ is a prime divisor
and $\Lo V_x$ is a surface for general $x\in D$, while $\dim \Lo V_x\geq 2$ for every $x\in D$.

Choose $x_0$ such that $\dim \Lo V_{x_0}= 2$. Since $x_0\in\Lo V_{x_0}\subset D$ and $\dim D=3$, we can choose 
 an irreducible curve $C\subset D$ containing $x_0$ and not contained in $\Lo V_{x_0}$. Then $$D=\bigcup_{x\in C}\Lo V_x=\Lo V_C,$$ thus $\N(D,X)=\R[C]+\R[V]$ by \eqref{lori} and $\dim\N(D,X)\leq 2$.
This implies that $\delta_X\geq\rho_X-2$.
Let us notice that $X$ cannot be a product of surfaces (see Ex.~\ref{linesproduct}), so that
 $\delta_X\leq 3$ by Th.~\ref{codim}, and hence $\rho_X\leq 5$.

\smallskip

Finally suppose that $\dim V=2$. By \eqref{rel} for general  $x\in \Lo V$ we have $\dim\Lo V+\dim\Lo V_x =4$.

If $\dim\Lo V=3$, then for $x\in\Lo V$ general we have $\dim\Lo V_x=1$ and $\dim V_x=0$.

If $\dim\Lo V=2$, we get $\dim\Lo V_x=2$ for $x\in\Lo V$ general, and in fact for every $x\in\Lo V$ by upper semicontinuity. 
Since $\Lo V$ is irreducible, we conclude that $\Lo V=\Lo V_x$ for every $x\in\Lo V$.

We show that in this last case two curves of the family always meet.
Let $\Gamma\subset\Lo V$ be a curve of the family and consider the universal family  as in \eqref{universal}. Then $e$ has fibers of dimension at least $1$, so that $e^{-1}(\Gamma)$ has dimension $2$. Since $e(e^{-1}(\Gamma))=\Gamma$,
we have
  $\pi(e^{-1}(\Gamma))=V$. This means that $e^{-1}(\Gamma)$ meets every fiber of $\pi$, hence $\Gamma$ meets every curve of the family $V$.
\end{proof}
\begin{example}[a case with $\dim V=4$ and $\rho_X=3$]\label{a}
Let $Y=\pr_{\pr^3}(\ol\oplus\ol(3))$, $G\subset Y$ a section of the $\pr^1$-bundle $Y\to\pr^3$ with normal bundle $\ma{N}_{G/Y}\cong\ol(3)$, and $X$ the blow-up of $Y$ along a plane contained in $G\cong\pr^3$. The $4$-fold $X$ is toric, Fano, with $\rho_X=3$, and contains a divisor $D\cong\pr^3$ with normal bundle $\ma{N}_{D/X}\cong\ol(-3)$; this is $E_1$ in \cite{bat2} and $X^7_{3,1}$ in \cite{saverio}.
Let $V$ be the family of lines in $D$, then $-K_X\cdot[V]=1$ and $\dim V=4$; in fact $V\cong\Gr(1,\pr^3)$. Every curve of the family is smooth with normal bundle $\ol(1)^{\oplus 2}\oplus\ol(-3)$, and $D\cdot [V]=-3$.

More generally,  Fano $4$-folds with $\rho_X=3$ and containing a prime divisor $D$ with  $\dim\N(D,X)=1$ are classified in \cite{saverio}, there are 28 families; these are the possible Fano $4$-folds with $\rho_X=3$ and with a family of lines $V$
 such that $\dim V=4$.
\end{example}
\begin{example}[a case with $\dim V=3$ and $\rho_X=5$]\label{b}
  Let $X$ be the toric Fano $4$-fold $K_1$ in Batyrev's list \cite{bat2}; we have $\rho_X=5$ and $X$ has a smooth fibration onto $\pr^2$, with fiber the surface obtained by blowing-up $\pr^2$ at $3$ non-collinear points. The $4$-fold $X$ contains a prime divisor $D\cong\pr^1\times\pr^2$ with normal bundle $\ma{N}_{D/X}\cong\ol(-1,-2)$. Let $V$ be the family of lines in $\{pt\}\times\pr^2\subset D$, then $-K_X\cdot[V]=1$ and $\dim V=3$; in fact $V\cong\pr^1\times\pr^2$. Every curve of the family is smooth with normal bundle $\ol(1)\oplus\ol\oplus\ol(-2)$, and $D\cdot [V]=-2$.
  Let us note that $D$ is also the locus of another family of lines $W$, given by the curves $\pr^1\times\{pt\}\subset D$.

 More generally, Fano $4$-folds with $\rho_X=5$ and containing a prime divisor $D$ with  $\dim\N(D,X)=2$ are classified in \cite{delta3_4folds}, there are 6 families; these are the possible Fano $4$-folds with $\rho_X=5$ and with a family of lines $V$
 such that $\dim V=3$.
\end{example}
\begin{lemma}\label{excplane}
  Let $X$ be a smooth Fano $4$-fold with $\rho_X\geq 7$, and $D\subset X$ a prime divisor covered by a family of lines.
  \begin{enumerate}[$(a)$]
    \item
      The family of lines $V$ such that $D=\Lo V$ is unique, so that $D$ and $V$ determine each other.
    \item  If $D$ contains an exceptional plane $L$, then $[V]=[C_L]$.
      \end{enumerate}
\end{lemma}
\begin{proof}
 We show $(a)$. Suppose by contradiction that there are two distinct families $V$, $W$ of lines with $D=\Lo V=\Lo W$. We prove that this yields $\dim\N(D,X)\leq 3$.

 Let $x\in D$ and consider
  $$S_x:=\Lo W_{\,\Lo V_x}.$$
  We have $\dim S_x\geq 2$ and by \eqref{lori} $\N(S_x,X)=\R[V]+\R[W]$. If $S_x=D$ for some $x$, we are done. Otherwise
    take $x_0\in D$ a general point and let $\Gamma\subset D$ be an irreducible curve through $x_0$ such that $\Gamma\not\subset S_{x_0}$.
   Consider
   $$T:=\bigcup_{x\in\Gamma}S_x=\Lo W_{\Lo V_{\Gamma}}.$$
Then  $T=D$, so that
by \eqref{lori} we get
$$\N(D,X)=\N(\Lo V_{\Gamma},X)+\R[W]=\R[\Gamma]+\R[V]+\R[W]$$ and $\dim\N(D,X)\leq 3$.
Since $\rho_X\geq 7$, this yields $\delta_X\geq 4$ and hence $X$ is a product of surfaces by Th.~\ref{codim}, a contradiction (see Ex.~\ref{linesproduct}).
    
The proof of $(b)$ is similar and we leave it to the reader.
  \end{proof}  
\begin{lemma}[\cite{blowup}, Lemma 2.18]\label{chitarre}
Let $X$ be a smooth Fano $4$-fold with $\rho_X\geq 7$, and $D\subset X$ a prime divisor covered by a family $V$ of lines. Then $D\cdot [V]\geq -1$, and
one of the following holds:
\begin{enumerate}[$(i)$]
\item
 $D$ is nef;
\item  $D\cdot [V]= -1$ and $[V]$ generates an extremal ray of type $(3,2)$ of $\NE(X)$.
\end{enumerate}
\end{lemma}
The following result is similar to \cite[Prop.~5.3]{morimukai2} on Fano $3$-folds.
\begin{lemma}\label{demons}
Let $X$ be a smooth Fano $4$-fold  and $V$ a family of lines in $X$ with $\dim V=2$ and $D:=\Lo V$ a prime divisor. 
Suppose that there exists a curve $C$ belonging to the family $V$ such that $C\cong\pr^1$, $\ma{N}_{C/X}\cong\ol^{\oplus 2}\oplus \ol(-1)$, and $C$ does not intersect other curves of the family $V$. Then $D\cdot [V]=-1$.
\end{lemma}
\begin{proof}
  Since $C\cong\pr^1$ and $h^1(C,\ma{N}_{C/X})=0$, $\Hilb(X)$ and $\Chow(X)$  are smooth  and locally isomorphic at the point corresponding to $C$
\cite[Th.~I.2.8 and I.6.3, Cor.~I.6.6.1]{kollar}, and 
$C$ corresponds to a smooth point $v_0$ of a unique irreducible component $V$ of $\Chow(X)$; consider the universal family \eqref{universal} and set $C_0:=\pi^{-1}(v_0)\subset\ma{C}$. Notice that $\ma{C}$ is smooth around $C_0$, $\pi$ is a smooth fibration in $\pr^1$ around $C_0$, and $e_{|C_0}\colon C_0\to C$ is an isomorphism.

  We show that the differential of $e\colon\ma{C}\to X$ is injective at every point of $C_0$; this is a standard argument. Restricting to the two curves,  $de\colon T_{\ma{C}\,|C_0}\to T_{X|C}$ restricts to an isomorphism between $T_{C_0}$ and $T_C$, and induces $\alpha\colon \ma{N}_{C_0/\ma{C}}\to\ma{N}_{C/X}$. Moreover $\ma{N}_{C_0/\ma{C}}\cong T_{v_0}V\otimes\ol_{C_0}\cong H^0(C,\ma{N}_{C/X})\otimes\ma{O}_{C_0}$, and $\alpha$ is naturally identified with the evaluation of global sections (see e.g.\ \cite[\S 2.2]{KPS}), which is injective because $\ma{N}_{C/X}\cong\ol^{\oplus 2}\oplus \ol(-1)$.

  Therefore 
  $e$ is smooth at every point of $C_0$. If there exist $z_1\in C_0$ and $z_2\in\ma{C}$ such that $e(z_1)=e(z_2)$, then   $e(\pi^{-1}(\pi(z_2)))$ is a curve of the family which intersects $C$, so by assumption $e(\pi^{-1}(\pi(z_2)))=C$. This implies that $\pi(z_2)=v_0$, $z_2\in C_0$ and hence that $z_1=z_2$.

  We conclude that $e\colon \ma{C}\to D$ is birational and that
  $C_0$ is contained in the open subset where $e$ is an isomorphism. Therefore $D$ is smooth around $C$, $-K_D\cdot C=2$, and finally $D\cdot C=-1$.
\end{proof}
\section{Fixed prime divisors and associated contractions}\label{fixed}
\noindent This section is devoted to fixed prime divisors. In \S \ref{fixedfaces} we introduce the notions of ``fixed face'' of the effective cone of a Mori dream space, and of ``adjacent'' fixed prime divisors; both will be very relevant in the rest of the paper.  In \S \ref{secfixed} we recall the classification of fixed prime divisors in Fano $4$-folds with $\rho\geq 7$, and report some results on them. Finally in \S \ref{contraction} we focus on fixed prime divisors of type $(3,1)^{sm}$ and $(3,0)^Q$, which are the ones that are relevant for the proof of Th.~\ref{main} (see Rem.~\ref{stress}), and prove many properties that we will need in the sequel.
\subsection{Fixed faces of the effective cone}\label{fixedfaces}
\begin{definition}
Let $X$ be a projective, normal and $\Q$-factorial Mori dream space.
  A face 
  $\tau$  of $\Eff(X)$ is called a {\bf fixed face} if $\tau\cap\Mov(X)=\{0\}$.\end{definition}
A fixed face is generated by classes of fixed prime divisors.
There is a bijection between fixed prime divisors of $X$ and 
one-dimensional fixed faces of $\Eff(X)$, via $D\mapsto
\R_{\geq 0}[D]$ (see \cite[Rem.~2.19]{eff}).
\begin{lemma}\label{simplicial}
 Let $X$ be a projective, normal, and $\Q$-factorial Mori dream space. Every fixed face  of $\Eff(X)$ is simplicial.
\end{lemma}
\begin{proof}
  Let $f\colon X\dasharrow Y$ be a birational (rational) contraction such that $Y$ is $\Q$-factorial, and  $\varepsilon_f\subset\Nu(X)$ 
   the convex cone generated by the  classes of all exceptional prime divisors of $f$.
  Then $\varepsilon_f$ is a face of $\Eff(X)$, because if $D_1,D_2$ are effective $\Q$-divisors on $X$ such that $[D_1+D_2]\in\varepsilon_f$, then $0=f_*(D_1)+f_*(D_2)$ and $f_*(D_1)$, $f_*(D_2)$ are effective $\Q$-divisors on $Y$; we conclude that 
  $f_*(D_1)=f_*(D_2)=0$ so that $D_1$ and $D_2$ are exceptional and $[D_1],[D_2]\in \varepsilon_f$.
  Thus $\varepsilon_f$ is a fixed face, and it
  is a simplicial cone by \cite[Lemma 2.7]{okawa_MCD}.
 
Consider the cone $\sigma_f:=\varepsilon_f+f^*\Nef(Y)\subset\Nu(X)$; then every face of  $\sigma_f$ is generated by a face of $\varepsilon_f$ and a face of $f^*\Nef(Y)$. 
  Recall that $\Eff(X)$ is the union of the cones $\sigma_f$ when $f$ varies
  \cite[Prop.~1.11(2)]{hukeel},
  and such cones intersect each other along common faces 
  \cite[Prop.~2.9]{okawa_MCD}.

 Now if $\tau$ is a fixed face of $\Eff(X)$,
  there exists some $f$ such that $\sigma_f$ has a face $\eta$ with $\eta\subseteq\tau$ and $\dim\eta=\dim\tau$.
  Since $\tau$ is fixed, 
we have $\eta\cap f^*\Nef(Y)=\{0\}$ and $\eta\subseteq\varepsilon_f$, therefore $\eta$ is a simplicial face of 
 $\Eff(X)$ and $\eta=\tau$.
  \end{proof}
\begin{definition} We say that two fixed prime divisors are {\bf adjacent} if their classes in $\Nu(X)$ generate a fixed face of $\Eff(X)$.
\end{definition}
 \begin{lemma}\label{taberna}
    Let $X$ be a projective, normal, and $\Q$-factorial Mori dream space,
 $f\colon X\dasharrow Y$ an elementary divisorial rational contraction, and 
$D:=\Exc(f)$, so that $D$ is a fixed prime divisor.
    \begin{enumerate}[$(a)$]
      \item
    For every $r\in \Z_{>0}$ there is a bijection between $(r+1)$-dimensional faces of $\Eff(X)$ containing $[D]$, and $r$-dimensional faces of $\Eff(Y)$, given by $\tau\mapsto f_*\tau$;
\item
 $\tau$ is fixed if and only if $f_*\tau$ is fixed;
\item
  there is a bijection between fixed prime divisors $E_Y\subset Y$ and fixed prime divisors $E_X\subset X$ adjacent to $D$ (and different from $D$); here $E_X$ is the transform of $E_Y$ in $X$.
  \end{enumerate}
    \end{lemma}
\begin{proof}
\cite[Lemma 2.21]{blowup} shows $(a)$ for $r=1$, as well as $(c)$. The same proof yields $(a)$ for any $r$, because under the push-forward of divisors $f_*\colon\Nu(X)\to\Nu(Y)$, the cone $\Eff(Y)$ can be seen as the ``quotient cone'' of $\Eff(X)$ modulo the one-dimensional face $\langle [D]\rangle$, see for instance \cite[Def.~V.2.8 and Th.~V.2.9]{ewald2}.

For $(b)$, if there exists a non-zero movable divisor $M$ with $[M]\in\tau$, then $f_*M$ is a non-zero movable divisor with class in $f_*\tau$. Conversely, if there exists a non-zero movable divisor $M_Y$ with $[M_Y]\in f_*\tau$, then
there exists $\mu\in\Q_{\geq 0}$ such that $M:=f^*M_Y-\mu D$ is non-zero and movable, and since $f_*([M])=[M_Y]\in f_*\tau$, there exists  $\lambda\in\Q$ such that $[M+ \lambda D]\in\tau$. 
Since $\tau$ is a face of $\Eff(X)$, there exists a class $\gamma\in\Eff(X)^{\vee}$ such that $\tau=\Eff(X)\cap\gamma^{\perp}$.
Then $M\cdot\gamma=(M+\lambda D)\cdot\gamma-\lambda D\cdot\gamma=0$ and $[M]\in\Eff(X)$, so that $[M]\in\tau$.
\end{proof}
\subsection{Fixed prime divisors of Fano $4$-folds}\label{secfixed}
\noindent Let $X$ be a smooth Fano 
 $4$-fold with $\rho_X\geq 7$. After \cite{eff,blowup} there are four possible types of fixed prime 
divisors in $X$.
 In this subsection we recall this classification.
\begin{thmdefi}[\cite{blowup}, Th.~5.1, Def.~5.3, Cor.~5.26, Def.~5.27]\label{long}
Let $X$ be a smooth Fano $4$-fold with $\rho_X\geq 7$, and $D$ a fixed prime divisor in $X$. 
\begin{enumerate}[$(a)$]
\item There exists a unique diagram:
$$X\stackrel{\ph}{\dasharrow}\w{X}\stackrel{\sigma}{\la}Y$$
where $\ph$ is a SQM, $\sigma$ is an elementary divisorial contraction with exceptional divisor the transform $\w{D}$ of $D$, and $Y$ is Fano (possibly singular);
\item  $\sigma$ 
is of type $(3,0)^{sm}$, $(3,0)^Q$, $(3,1)^{sm}$, or $(3,2)$, and we define $D$ to be of type $(3,0)^{sm}$, $(3,0)^Q$, $(3,1)^{sm}$, or $(3,2)$, respectively.
\item If $D$ is of type $(3,2)$, then $X=\w{X}$.
\item We define $C_D\subset D\subset X$ to be the transform of a general irreducible curve $C_{\w{D}}\subset \w{X}$ contracted by $\sigma$, of minimal anticanonical degree. Then  $C_D\cong\pr^1$, $D\cdot C_D=-1$, and $C_D$ is contained in the open subset where $\ph$ is an isomorphism.
\item Given a SQM $X\dasharrow X'$ and an elementary divisorial contraction $k\colon X'\to Y$ with $\Exc(k)$ the transform of $D$, then $k$ 
  has the same type as $\sigma$.
\end{enumerate}
 \end{thmdefi}
 \noindent We will frequently use the notations $C_D\subset D$ and
 $C_{\w{D}}\subset \w{D}$ introduced  above.
 \begin{example}\label{2022}
Let $X$ be a smooth Fano $4$-fold with $\rho_X\geq 7$. If $X$ is a product of surfaces, then every elementary birational contraction of $X$ 
has the form $S\times T\to S'\times T$, where $S\to S'$ is the blow-up of a point. In particular $X$ has no small contraction, and
 every fixed prime divisor of $X$ is of type $(3,2)$.
 \end{example}  

Fano $4$-folds with a fixed prime divisor of type $(3,0)^{sm}$ have been treated in \cite{blowup}:
 \begin{thm}[\cite{blowup}, Th.~5.40]\label{30}
Let $X$ be a smooth Fano $4$-fold with $\rho_X\geq 7$, having a fixed prime divisor of type $(3,0)^{sm}$. Then $\rho_X\leq 12$, and if $\rho_X=12$, then 
$X$ has a rational contraction onto a $3$-fold.
\end{thm}
Concerning fixed prime divisors of type $(3,2)$, we recall the following results, that will be relevant in the sequel.
\begin{thm}[\cite{blowup}, Prop.~5.32]
\label{natale}
Let $X$ be a smooth Fano $4$-fold with $\rho_X\geq 7$ and $\delta_X\leq 1$, having
 a fixed prime divisor
 $D$ of type $(3,2)$ such that  $\N(D,{X})\subsetneq\N({X})$.
Then $\rho_X\leq 12$, and if $\rho_X=12$, then 
$X$ has a rational contraction onto a $3$-fold.
\end{thm}
\begin{remark}[\cite{blowup}, Rem.~2.17(2)]\label{ikea}
Let $X$ be a Fano $4$-fold with $\rho_X\geq 7$ and $D\subset X$ a fixed prime divisor of type $(3,2)$. Then $D$ does not contain exceptional planes.
\end{remark}
\begin{lemma}\label{cena}
Let $X$ be a Fano $4$-fold with $\rho_X\geq 7$, $D\subset X$ a fixed prime divisor of type $(3,2)$, $X\dasharrow\w{X}$ a SQM, and $\w{D}\subset \w{X}$ the transform of $X$. Then $\dim\N(D,X)=\dim\N(\w{D},\w{X})$.
\end{lemma}
\begin{proof}
This follows from Rem.~\ref{ikea}  and  \cite[Cor.~3.14]{eff}.
\end{proof}
 Let us note that Lemma \ref{cena} extends the applicability of Th.~\ref{natale} also to SQM's of $X$. This is a special property of fixed prime divisors of type $(3,2)$, as for a general prime divisor $D$ in a Fano $4$-fold $X$, if $X\dasharrow\w{X}$ is a SQM and $\w{D}\subset\w{X}$ is the transform of $D$, then $\dim\N(\w{D},\w{X})$ may be smaller than $\dim\N(D,X)$ (see for instance \S \ref{contraction}).
\begin{remark}[fixed prime divisors and small contractions]\label{stress}
  Let $X$ be a smooth Fano $4$-fold with $\rho_X\geq 7$. If $X$ has a fixed prime divisor $D$ not of type $(3,2)$, then $X$ has a small elementary contraction, indeed by \cite[Th.~5.1]{blowup} the map $\ph$ in Th.-Def.~\ref{long} factors as a sequence of at least $\rho_X-4$ $K$-negative flips.

  Conversely,
  it is not difficult to show (see Rem.~\ref{stress2}) that if $X$ has a small elementary contraction $f\colon X\to Y$, then either $\rho_X\leq 11$, or $X$ has a fixed prime divisor $D$ such that $D\cdot\NE(f)<0$, and $D$ cannot be of type $(3,2)$ by Th.~\ref{kawamata} and Rem.~\ref{ikea}. Moreover, if $D$ is of type $(3,0)^{sm}$, we can apply Th.~\ref{30}. 

  Therefore Th.~\ref{main} can be seen as a statement on Fano $4$-folds with $\rho_X\geq 7$ having a fixed prime divisor of type $(3,1)^{sm}$ or $(3,0)^Q$, and we will focus on these two types.
\end{remark}
We will need some further properties of fixed prime divisors.
\begin{lemma}\label{2faces}
Let $X$ be a smooth Fano $4$-fold with $\rho_X\geq 7$ and $D_1,D_2\subset X$ two distinct fixed prime divisors.
\begin{enumerate}[$(a)$]
\item If $D_1\cdot C_{D_2}=0$, then $D_1$ and $D_2$ are adjacent.
  \item If $D_1\cdot C_{D_2}>0$ and $D_2\cdot C_{D_1}>0$, then $D_1+D_2$ is movable and $[C_{D_1}+C_{D_2}]\in\mov(X)$.
  \item
    $D_1\cdot C_{D_2}=D_2\cdot C_{D_1}=1$ if and only if
    $\dim(\langle [D_1],[D_2]\rangle\cap\Mov(X))=1$; in this case
    $\langle [D_1],[D_2]\rangle\cap\Mov(X)=\langle[D_1+D_2]\rangle$,
  and   $D_1+D_2$ is movable and  non-big.
\end{enumerate}
\end{lemma}
\begin{proof}
This is \cite[Lemma 4.6]{fibrations}, except $(b)$ which follows  from the same proof.
\end{proof}  
\begin{lemma}[\cite{fibrations}, Lemma 4.9]\label{nigra}
Let $X$ be a smooth Fano $4$-fold with $\rho_X\geq 7$ and $D,E$ two adjacent fixed prime divisors. 
\begin{enumerate}[$(a)$]
\item If $E$ is of type $(3,2)$, then $D\cdot C_E=0$;
\item if $D$ and $E$ are not of type $(3,2)$, then 
$D\cdot C_E=E\cdot C_D=0$, and $D\cap E$ is either empty or a disjoint union of exceptional planes.
\end{enumerate}
\end{lemma}
\begin{corollary}\label{solfeggio}
Let $X$ be a smooth Fano $4$-fold with $\rho_X\geq 7$  and $D_1,D_2\subset X$ two fixed prime divisors, both of type $(3,2)$ or  neither. If  $D_1\cdot C_{D_2}>0$, then  $D_2\cdot C_{D_1}>0$, $D_1$ and $D_2$ are not adjacent, $D_1+D_2$ is movable, and  $[C_{D_1}+C_{D_2}]\in\mov(X)$.
\end{corollary}
\begin{proof}
By Lemma \ref{nigra}  $D_1$ and $D_2$ are not adjacent, and the rest follows from Lemma \ref{2faces}.
\end{proof}
\begin{lemma}\label{preistorico}
Let $X$ be a smooth Fano $4$-fold with $\rho_X\geq 7$ and $E_1,E_2$ two fixed prime divisors of type $(3,2)$ such that $E_1\cdot C_{E_2}=0$ and $E_1\cap E_2\neq\emptyset$. Then $E_2\cdot C_{E_1}=0$ and every connected component of $E_1\cap E_2$ is isomorphic to $\pr^1\times\pr^1$ with normal bundle $\ol(-1,0)\oplus\ol(0,-1)$.
\end{lemma}
\begin{proof}
By Cor.~\ref{solfeggio}  we have
  $E_2\cdot C_{E_1}=0$ and $E_1$ and $E_2$ are adjacent. Moreover 
   $\langle[C_{E_i}]\rangle$ is the unique $E_i$-negative extremal ray of $\NE(X)$, for $i=1,2$, by  \cite[Rem.~2.17]{blowup}. Then $-K_X+E_1+E_2$ is nef and $(-K_X+E_1+E_2)^{\perp}\cap\NE(X)=\langle [C_{E_1}],[C_{E_2}]\rangle$ is a face of $\NE(X)$. The associated contraction $f\colon X\to Z$ is birational, has exceptional locus $E_1\cup E_2$, and the general fiber of $f_{|E_i}$ is one-dimensional.
  The possible two-dimensional fibers of $f$ are classified in \cite[Th.~4.7]{AWaview} and are $\pr^2$ or a (possibly singular/reducible) quadric surface.

  If $F$ is an irreducible component of $E_1\cap E_2$, then $\N(F,X)=\R[C_{E_1}]\oplus\R[C_{E_2}]$ and $f(F)=\{pt\}$. Thus $F$ cannot be isomorphic to $\pr^2$ nor to a quadric cone, and we conclude that  $F\cong\pr^1\times\pr^1$ and $F$ is a fiber of $f$. The normal bundle is given in \cite{AWaview}, and $F$ must be a connected component of $E_1\cap E_2$.
\end{proof}
\subsection{Contraction of a fixed prime divisor of type $(3,1)^{sm}$ or $(3,0)^Q$}
\label{contraction}
\noindent Let $X$ be a smooth Fano $4$-fold with $\rho_X\geq 7$ and $D\subset X$ a fixed prime divisor of type $(3,1)^{sm}$ or $(3,0)^Q$. Let us consider the diagram:
$$X\stackrel{\ph}{\dasharrow}\w{X}\stackrel{\sigma}{\la} Y$$
given by Th.-Def.~\ref{long}$(a)$,
where $\ph$ is a SQM, $\sigma$ is an elementary divisorial contraction with exceptional divisor the transform $\w{D}\subset\w{X}$ of $D$, and $Y$ is Fano.
We will refer to this diagram as \emph{the contraction associated to $D$}. Let us sum up here its main properties and fix the related notation; see \cite[\S 5.1]{blowup} for more details.

Recall from Lemma \ref{basic1} that the indeterminacy locus of $\ph$ is a disjoint union of exceptional planes.

For every exceptional plane $L\subset X$ in the indeterminacy locus of $\ph$ we have $L\subset D$, $D\cdot C_L<0$, and if $\ell\subset\w{X}$ is the corresponding exceptional curve, then $\w{D}\cdot\ell=-D\cdot C_L>0$. 
Conversely, every exceptional plane contained in $D$ is in the indeterminacy locus of $\ph$, so that the exceptional planes in $D$ are pairwise disjoint.
No exceptional line of $\w{X}$ is contained in $\w{D}$ \cite[Rem.~5.6]{blowup}.
We have $-K_X\cdot C_D=-K_{\w{X}}\cdot C_{\w{D}}=2$.

If $D$ is of type $(3,1)^{sm}$, then $Y$ is smooth, $\sigma$ is the blow-up of a smooth irreducible curve $C\subset Y$, $\w{D}$ is a $\pr^2$-bundle over $C$, and $C_{\w{D}}\subset\w{D}$ is a line in a fiber of $\sigma$. Every fiber of the $\pr^2$-bundle $\sigma_{|\w{D}}$ meets the union of the exceptional lines of $\w{X}$ in at most one point
\cite[Rem.~5.5]{blowup}.

If $D$ is of type $(3,0)^Q$, then $Y$ has an isolated terminal and locally factorial singularity at $p:=\sigma(\w{D})$, and $\w{D}$ is either a smooth quadric, or the cone over a smooth $2$-dimensional quadric \cite[Lemma 2.19]{blowup}; moreover $C_{\w{D}}\subset\w{D}$ is a line.
\begin{remark}[\cite{blowup}, Lemmas 5.10 and 5.20]\label{vecchissimo}
Let $\Gamma\subset Y$ be an irreducible curve such that  $\Gamma\cap\sigma(\w{D})\neq\emptyset$ and $\Gamma\neq\sigma(\w{D})$. Then
$-K_Y\cdot\Gamma\neq 2$, and if $-K_Y\cdot\Gamma=1$ we have 
 $\Gamma=\sigma(\ell)$, where $\ell\subset\w{X}$ is an exceptional curve such that $\w{D}\cdot\ell=1$. 
If $D$ is $(3,1)^{sm}$, then $C$ cannot meet any exceptional plane.
\end{remark}
\begin{remark}\label{generators}
  In the above setting set
  $$m:=\dim\N(D,X)-\dim\N(\w{D},\w{X})=\begin{cases}
    \dim\N(D,X)-1\,\text{ if $D$ is of type $(3,0)^Q$}\\
    \dim\N(D,X)-2\,\text{ if $D$ is of type $(3,1)^{sm}$.}
  \end{cases}$$
  It follows from \cite[Rem.~3.13 and Cor.~3.14]{eff} that $D$ contains at least $m$ exceptional planes $L_1,\dotsc,L_m$ such that the classes $[C_{L_1}],\dotsc,[C_{L_m}]$ are linearly independent in $\N(X)$.
Let us note that $X$ cannot be a product of surfaces,  thus $\delta_X\leq 2$ by Th.~\ref{codim}, $\dim\N(D,X)\geq\rho_X-2\geq 5$, and $m\geq 3$.
  \end{remark}
\subsection{Additional properties for the case $(3,1)^{sm}$}\label{add1}
\begin{prg}\label{31}
In this subsection $X$ is a smooth Fano $4$-fold with $\rho_X\geq 7$ and $D\subset X$ is a fixed prime divisor of type $(3,1)^{sm}$. We keep the same notation as in \S\ref{contraction}.\end{prg}
\begin{lemma}\label{vecchio}
In Setting \ref{31}, let $\ell\subset\w{X}$ an exceptional curve and $\Gamma:=\sigma(\ell)\subset Y$. Then $\Gamma$ cannot meet any curve of anticanonical degree one, except possibly $C$ (and $\Gamma$ itself).
\end{lemma}
\begin{proof}
  Let $C_0\subset Y$ be an irreducible curve with $-K_Y\cdot C_0=1$ and $C_0\neq C$, $C_0\neq \Gamma$, and let $\w{C}_0\subset \w{X}$ be its transform. Note that $\w{C}_0\neq\ell$, as $C_0\neq \Gamma$.

If $C\cap C_0\neq\emptyset$, then $\w{C}_0$ is an exceptional curve by Rem.~\ref{vecchissimo}, and $\w{C}_0\neq\ell$, so that $\w{C}_0$ and $\ell$ are disjoint and meet $\w{D}$ in different fibers of $\sigma$. Therefore $\Gamma\cap C_0=\emptyset$.

If instead  $C\cap C_0=\emptyset$, then $-K_{\w{X}}\cdot\w{C}_0=1$ and $\w{D}\cap 
\w{C}_0=\emptyset$, therefore $\ell\cap \w{C}_0=\emptyset$ by Lemma \ref{basic1}$(b)$, and their images stay disjoint in $Y$.
\end{proof}
\begin{lemma}\label{fabrizio}
In Setting \ref{31}, let $\Gamma\subset Y$ be an integral curve such that $-K_Y\cdot \Gamma=1$, $\Gamma\cap C\neq\emptyset$, and $\Gamma\neq C$. Then  $\Gamma\cong\pr^1$ and $\ma{N}_{\Gamma/Y}\cong \ol(-1)\oplus\ol^{\oplus 2}$.
\end{lemma}
\begin{proof}
By Rem.~\ref{vecchissimo} we have $\Gamma=\sigma(\ell)$ where $\ell\subset\w{X}$ is an exceptional curve with $\w{D}\cdot\ell=1$; therefore $\sigma_{|\ell}\colon\ell\to\Gamma$ is an isomorphism, and  $\Gamma\cong\pr^1$.

  We have $\ma{N}_{\ell/\w{X}}\cong\ol(-1)^{\oplus 3}$ and $\ma{N}_{\Gamma/Y}\cong\ol(a)\oplus\ol(b)\oplus\ol(-1-a-b)$ with $a,b\in\Z$.
The differential of $\sigma$ induces a morphism of sheaves
$\zeta\colon\ma{N}_{\ell/\w{X}}\to\ma{N}_{\Gamma/Y}$, which is generically an isomorphism. Then $\zeta$ is given by a $3\times  3$ matrix $(s_{ij})_{i,j=1,2,3}$, where for $j=1,2,3$
\begin{gather*}
  s_{1j}\in\Hom(\ol(-1),\ol(a))=H^0(\pr^1,\ol(a+1)),\\
  s_{2j}\in H^0(\pr^1,\ol(b+1)),\quad
s_{3j}\in H^0(\pr^1,\ol(-a-b)),
\end{gather*}  and the matrix is generically invertible. Thus for every $i=1,2,3$ there exists $j\in\{1,2,3\}$ such that $s_{ij}\neq 0$, and this implies that $a\geq -1$, $b\geq -1$, and $a+b\leq 0$, giving two possibilities for the normal bundle of $\Gamma$: either $\ol(-1)\oplus\ol^{\oplus 2}$ or $\ol(-1)^{\oplus 2}\oplus\ol(1)$.

Suppose that $\ma{N}_{\Gamma/Y}\cong\ol(-1)^{\oplus 2}\oplus\ol(1)$.  Up to reordering we have $a=b=-1$ and hence $s_{1j}$ and $s_{2j}$ are constant for every $j$, namely the first and second column of the matrix are constant. Since the matrix is generically invertible, these columns must be linearly independent, so that the rank of $\zeta$ is at least $2$ at every point of $\ell$. To rule out this case, we show that for $x\in \w{D}\cap\ell$, the rank of $\zeta$ at $x$ is $1$.

Set $y:=\sigma(x)\in C\cap\Gamma$.
With a local computation one checks that $d_x\sigma\colon T_x\w X\to T_yY$ has rank 2. On the other hand $x\in\ell$ and $d_x\sigma(T_x\ell)=T_y\Gamma$, thus 
$\im d_x\sigma$ contains $T_y\Gamma$.
 Then the image of $\im d_x\sigma$ in $(\ma{N}_{\Gamma/Y})_y=T_yY/T_y\Gamma$ is one dimensional, and this is the image of $\zeta$.
\end{proof}
\begin{lemma}\label{test2}
In Setting \ref{31}:
\begin{enumerate}[$(a)$]
\item
if $Y$ contains a nef prime divisor
$H$  covered by a family $V$ of lines, such that $H\cap C\neq\emptyset$, then $C$ is a curve of the family $V$;
\item if $Y$ is covered by a family $W$ of rational curves with $-K_Y\cdot [W]=2$, then $C$ is a component of a curve of the family $W$.
\end{enumerate}
\end{lemma}
\begin{proof}
$(a)\,$ Let $q\in C\cap H$, and let $\Gamma$ be a curve of the family $V$ containing $q$. If $\Gamma=C$, we are done. Otherwise,
since $-K_Y\cdot \Gamma=1$ and $\Gamma\cap C\neq\emptyset$, we have that  $\Gamma\cong\pr^1$, $\ma{N}_{\Gamma/Y}\cong\ol(-1)\oplus\ol^{\oplus 2}$,
and $\Gamma$ is the image of an exceptional curve,
 by Lemma \ref{fabrizio} and Rem.~\ref{vecchissimo}. Since $H$ is nef, by Lemma \ref{demons} $\Gamma$ must intersect some other curve of the family $V$. On the other hand by Lemma \ref{vecchio} $\Gamma$ cannot intersect any curve of anticanonical degree one, except $C$; this implies that $C$ belongs to the family $V$.

$(b)\,$ Let $\Gamma_2$ be a one-cycle of the family $W$ intersecting $C$. 
If $\Gamma_2$ is irreducible and reduced, then $\Gamma_2=C$ by Rem.~\ref{vecchissimo}. 

If $\Gamma_2$ is reducible, let $\Gamma_2'$, $\Gamma_2''$ be its irreducible components, both of anticanonical degree one. We can assume that $\Gamma_2'\neq C$ and  $\Gamma_2'\cap C\neq\emptyset$. Then $\Gamma_2''=C$ by Lemma \ref{vecchio}.

Finally if $\Gamma_2$ is non-reduced, either it is supported on $C$, or
 $\Gamma_2=2\Gamma_2'$ where $-K_Y\cdot \Gamma_2'=1$, $\Gamma_2'\neq C$, and  $\Gamma_2'\cap C\neq\emptyset$. Then again by Rem.~\ref{vecchissimo} $\Gamma_2'$
 is the image of an exceptional curve of $\w{X}$. In this last case, 
 $\Gamma_2$ meets 
$C$ in the finite set $T$ given by the intersection with the images of the finitely many  exceptional curves of $\w{X}$. Thus if $\Gamma_3$ is a one-cycle of the family $W$ meeting $C$ outside $T$, it must have
$C$ as a component.
\end{proof}
\begin{lemma}\label{pfizer}
In Setting \ref{31}, suppose moreover that $\rho_X\geq 8$. Let $E_Y\subset Y$ be a fixed prime divisor, and $E\subset X$ its transform, so that $E$ is a fixed prime divisor adjacent to $D$ (see Lemma \ref{taberna}$(c)$).
Then $E$ and $E_Y$ have the same type.
\end{lemma}
\begin{proof}
  Suppose first that $E$ is not of type $(3,2)$; then, by Lemma \ref{nigra}$(b)$, $D\cap E$ is either empty of a disjoint union of exceptional planes, and for every such $L$, we have $D\cdot C_L<0$ and $E\cdot C_L<0$ (see \S\ref{contraction}).

There is a SQM $\psi\colon X\dasharrow\wi{X}$ whose indeterminacy locus is the union of all exceptional planes contained in $D\cup E$ (this can be obtained by flipping consecutively all small $K$-negative extremal rays of $\NE(X)$ having negative intersection with $D+E$). The transforms $\wi{D},\wi{E}\subset\wi{X}$ are disjoint, and the SQM $\psi\circ\ph^{-1}\colon \w{X}\dasharrow\wi{X}$ induces an isomorphism between $\w{D}$ and $\wi{D}$. Similarly, if $X\stackrel{\ph_E}{\dasharrow} \w{X}_E\to Y_E$ is the contraction associated to $E$, the SQM  $\psi\circ\ph_E^{-1}\colon \w{X}_E\dasharrow\wi{X}$ induces an isomorphism between $\w{E}\subset\w{X}_E$ and $\wi{E}$.

It is not difficult to see that 
$\wi{D},\wi{E}$ are the loci of two divisorial extremal rays $R_D,R_E$ of $\NE(\wi{X})$, such that $R_D+R_E$ is a face of $\NE(\wi{X})$. 
In fact, if $m:=-K_X\cdot C_E$ ($m\in\{2,3\}$ depending on the type of $E$), the divisor $-K_{\wi{X}}+2\wi{D}+m\wi{E}$ is nef and $(-K_{\wi{X}}+2\wi{D}+m\wi{E})^{\perp}\cap\NE(\wi{X})=R_D+R_E$.
Thus we have a diagram:
$$\xymatrix{{\wi{X}}\ar[r]^{\sigma_D}\ar[d]_{\sigma_E}&{\wi{Y}}\ar[d]^k\\
  {Z}\ar[r]&W
}$$
where $\sigma_D$ and $\sigma_E$ are elementary divisorial contractions with exceptional divisors  $\wi{D}$ and $\wi{E}$ respectively, $\wi{Y}$ is a SQM of $Y$, and $k$ is  an elementary divisorial contraction with exceptional divisor the transform of $E_Y$. Since   $\wi{D}\cap\wi{E}=\emptyset$, $\sigma_E$ and $k$ coincide in a neighborhood of their exceptional divisors. Moreover, by Th.-Def.~\ref{long}$(e)$, $\sigma_E$ and $E$ have the same type, and similarly for $k$ and $E_Y$. Therefore  $E$ and $E_Y$ have the same type.

Suppose now that $E$ is of type $(3,2)$; then by Lemma \ref{nigra}$(a)$ we have $D\cdot C_{E}=0$, so that the general curve $C_{E}$ is disjoint from $D$, hence it is contained in the open subset where the birational map $X\dasharrow Y$ is an isomorphism.
Let $\Gamma\subset E_Y\subset Y$ be the  image  of  the general $C_{E}$. Then $-K_Y\cdot\Gamma=1$ and $E_Y\cdot \Gamma=-1$, so that $E_Y$ is covered by a family of lines and it is not nef; then $E_Y$ is of type $(3,2)$ by Lemma \ref{chitarre}.
\end{proof}  
\begin{lemma}\label{aikido}
Let $X$ be a smooth Fano $4$-fold with $\rho_X\geq 7$, and 
$D$ and $E$ two adjacent fixed prime divisors, of type $(3,1)^{sm}$ and $(3,2)$ respectively, such that  $E\cdot C_D>0$. Then  
$E\cdot C_D=1$ and $E\cap L=\emptyset$ for every exceptional plane $L\subset D$.
\end{lemma}  
\begin{proof}
  Let $X\dasharrow\w{X}\to Y$ be the contraction of $D$ as in \S\ref{contraction}, and $E_{\w{X}}\subset\w{X}$, $E_Y\subset Y$ the transforms of $E$, so that $E_Y$ is a fixed prime divisor by Lemma \ref{taberna}$(c)$. Since $E_{\w{X}}\cdot C_{\w{D}}=E\cdot C_D>0$, $E_{\w{X}}$ intersects every nontrivial fiber of $\sigma$, thus $E_Y$ contains $C$.
  By \cite[Lemma 5.11]{blowup} we deduce that 
$[C]$ generates an extremal ray of type $(3,2)$ of $\NE(Y)$,
$C\cong\pr^1$ is a fiber of the associate contraction, $E_Y$ is a smooth $\pr^1$-bundle around $C$, $\ma{N}_{C/Y}\cong\ol^{\oplus 2}\oplus\ol(-1)$, and $\w{D}\cong\pr_{\pr^1}(\ol^{\oplus 2}\oplus\ol(1))$ is isomorphic to the blow-up of $\pr^3$ along a line. Moreover $E_{\w{X}}$ and $\w{D}$ intersect transversally along a smooth irreducible surface and $\sigma^*E_Y=E_{\w{X}}+\w{D}$, so that $E\cdot C_D=E_{\w{X}}\cdot C_{\w{D}}=(\sigma^*E_Y-\w{D}) \cdot C_{\w{D}}=1$.

If $\Gamma\subset\w{D}$ is a non-trivial fiber of the blow-up $\w{D}\to\pr^3$, it is not difficult to see that $\w{D}\cdot\Gamma=0$, using that $\ol_{\w{X}}(-\w{D})_{|\w{D}}$ is the tautological line bundle. Then $-K_{\w{X}}\cdot\Gamma=-K_{\w{D}}\cdot\Gamma=1$, and $E_{\w{X}}\cdot\Gamma=\sigma^*E_Y\cdot\Gamma=E_Y\cdot C=-1$, so that $\Gamma\subset E_{\w{X}}$. This implies that $\w{D}\cap E_{\w{X}}\cong\pr^1\times\pr^1$ is the exceptional divisor of the blow-up $\w{D}\to\pr^3$, and it is covered by lines. On the other hand also $E_{\w{X}}\smallsetminus\w{D}\cong E_Y\smallsetminus C$ is covered by lines.  
Thus $E_{\w{X}}$ cannot meet any exceptional curve of $\w{X}$ by Lemma \ref{basic1}$(b)$, and $E$ cannot meet any exceptional plane contained in $D$. 
\end{proof}
\subsection{Additional properties for the case $(3,0)^Q$}\label{add2}
\begin{prg}\label{30Q}
In  this subsection $X$ is a smooth Fano $4$-fold with $\rho_X\geq 7$ and $D\subset X$ is a fixed prime divisor of type $(3,0)^{Q}$. We keep the same notation as in \S\ref{contraction}.\end{prg}
\begin{remark}\label{vertex}
In Setting \ref{30Q}, suppose that $\w{D}\subset\w{X}$ is a quadric cone. Then no exceptional curve can intersect $\w{D}$ at the vertex of the cone. 

Indeed, suppose otherwise: then for every line $\Gamma\subset\w{D}$ through the vertex of the cone, we have $-K_{\w{X}}\cdot\Gamma=2$, so if $\Gamma_X\subset X$ is the transform of $\Gamma$, we have $-K_X\cdot\Gamma_X=1$ by Lemma \ref{basic1}$(b)$. This yields that $D$ is covered by a family of lines, it is not nef and of type $(3,0)^Q$, contradicting Lemma \ref{chitarre}.
\end{remark}
\begin{remark}\label{intersection}
  Let $B\subset X$ be a prime divisor, different from $D$, and such that $B\cap D\neq\emptyset$. Then either $B\cap D$ is a disjoint union of exceptional planes, or $B\cdot C_D>0$.

 Indeed let $\w{B}\subset\w{X}$ be the transform of $B$. If $\w{B}\cap\w{D}=\emptyset$, then  $B\cap D$ is contained in the indeterminacy locus of $\ph\colon X\dasharrow\w{X}$, thus it is a disjoint union of exceptional planes. If $\w{B}\cap\w{D}\neq\emptyset$, then $B\cdot C_D=\w{B}\cdot C_{\w{D}}>0$.
\end{remark}  
The proof of the next lemma is analogous to that of Lemma \ref{vecchio}, thus we omit it.
\begin{lemma}\label{vecchio2}
In Setting \ref{30Q}, let $\ell\subset\w{X}$ be an exceptional curve and $\Gamma:=\sigma(\ell)\subset Y$. Then $\Gamma$ cannot meet any curve of anticanonical degree one outside $p$ (except possibly $\Gamma$ itself).
\end{lemma}
\begin{lemma}\label{pizza}
In Setting \ref{30Q}, let $E\subset X$ be a fixed prime divisor adjacent to $D$, and $E_Y\subset Y$ its transform. If
 $\N(E_Y,Y)\subsetneq\N(Y)$, then one of the following holds:
\begin{enumerate}[$(i)$]
\item $X$ has a rational contraction onto a $3$-fold;
\item $X$ has a
 fixed prime divisor $G$ of type $(3,2)$, adjacent to $D$, and such that $E\cdot C_G>0$.
\end{enumerate}
\end{lemma}
\begin{proof}
We recall that $Y$ is Fano and has one isolated terminal and locally factorial singularity. Suppose that $\N(E_Y,Y)\subsetneq\N(Y)$; we apply \cite[Th.~3.1 and Lemma 3.3]{gloria}  to $Y$
  and $E_Y$. This yields a diagram:
$$Y=Y_0\stackrel{\sigma_0}{\dasharrow}Y_1\dasharrow\cdots\dasharrow Y_{k-1}\stackrel{\sigma_{k-1}}{\dasharrow} Y_k\stackrel{\psi}{\la} Z$$ 
where each $\sigma_i$ is either an elementary divisorial contraction or a flip, and $\psi$ is an elementary contraction of fiber type \cite[Th.~3.1(2)]{gloria}.
The divisor $E_Y$ is not exceptional for any $\sigma_i$; let 
 $E_Y^i\subset Y_i$ be its transform and set $c_i:=\codim\N(E_Y^i,Y_i)$. We have
$c_{i+1}\in\{c_i,c_i-1\}$ for $i=0,\dotsc,k-1$ and $c_k\in\{0,1\}$  \cite[Th.~3.1(3)]{gloria}.

If $c_k=1$, then $\NE(\psi)\not\subset\N(E_Y^k,Y_k)$ and $E_Y^k\cdot\NE(\psi)>0$ \cite[Th.~3.1(2) and (3)]{gloria}. This implies that if $F\subset Y_k$ is a fiber of $\psi$, then $\dim(F\cap E_Y^k)\leq 0$ and $F\cap E_Y^k\neq\emptyset$, thus  $\dim F=1$ 
and $\dim Z=3$, and we get $(i)$.

If $c_k=0$, since $c_0=\codim\N(E_Y,Y)>0$, there exists some $i\in\{0,\dotsc,k-1\}$ such that $c_{i+1}=c_i-1$. Then $\sigma_i$ is a $K$-negative elementary divisorial contraction of type $(3,2)$ \cite[Lemma 3.3(2)]{gloria}, and $\Exc(\sigma_i)$ is contained in the open subset where the birational map $Y\dasharrow Y_i$ is an isomorphism \cite[Lemma 3.3(3)]{gloria}; let $G_Y\subset Y$ and $G\subset X$ be the transforms of $\Exc(\sigma_i)\subset Y_i$. 

Then $G_Y$ is covered by irreducible curves $\Gamma$ such that $E_Y\cdot \Gamma>0$, $-K_Y\cdot \Gamma=1$, and $G_Y\cdot \Gamma=-1$ \cite[Lemma 3.3(5)]{gloria}; the general $\Gamma$ does not contain $p=\sigma(\w{D})$, so by Lemma \ref{vecchio2} it is contained 
in the open subset where the birational map $X\dasharrow Y$ 
is an isomorphism. Therefore $G\subset X$ is covered by a family of lines $V$ such that $G\cdot [V]=-1$ and $E\cdot [V]>0$; $G$ is a fixed prime divisor of type $(3,2)$ with $[V]=[C_G]$ by Lemma \ref{chitarre}, and it is adjacent to $D$ by Lemma \ref{taberna}$(c)$, so we have $(ii)$.
  \end{proof}
\section{Rational contractions of fiber type}\label{fiber}
\subsection{Rational contractions of fiber type in Mori dream spaces}
\noindent Let $X$ be a projective, normal, and $\Q$-factorial Mori dream space. We need to introduce two special notions of rational contractions of fiber type, namely ``quasi-elementary'' and ``special'' contractions; we refer the reader to \cite[\S 2.2]{eff} and \cite[\S 2]{fibrations} respectively for more details.
\begin{definition}
Let $f\colon X\to Z$ be a contraction of fiber type.
We say that $f$ is {\bf quasi-elementary} if $Z$ is $\Q$-factorial and for every prime divisor $B\subset Z$ the pull-back $f^*B$ is irreducible (but possibly non-reduced).

We say that $f$ is {\bf special}  if $Z$ is $\Q$-factorial and $\codim f(D)\leq 1$ for every prime divisor $D\subset X$. A quasi-elementary contraction is always special.

We will be interested mainly in the cases where $\dim Z\leq 2$: if $Z$ is a curve, then every contraction of fiber type is special. If $Z$ is a surface, then $f$ is special if and only if it is {\bf equidimensional}: indeed if $f$ is special and $\dim Z=2$, then $f$ must be equidimensional; the converse follows from \cite[Lemma 2.7]{fibrations}. 

Consider now a rational contraction of fiber type $f\colon X\dasharrow Z$. We say that $f$ is quasi-elementary, respectively special, if given a factorization of $f$ as a SQM $X\dasharrow X'$ followed by a regular contraction $f'\colon X'\to Z$, then $f'$ is quasi-elementary, respectively special; this does not depend on the choice of the factorization.
\end{definition}
\begin{lemma}\label{Qfactorial}
Let $X$ be a projective, normal, and $\Q$-factorial Mori dream space, and $f\colon X\dasharrow S$ a rational contraction onto a surface. Then 
  $S$ is $\Q$-factorial.
\end{lemma}
\begin{proof}
Suppose first that $X$ is a surface, so that $f$ is birational and regular.
By dimensional reasons, $\Exc(f)$ is a divisor. If $f$ is elementary, the statement is well known. In general, we can factor $f$ as $X\stackrel{\sigma}{\to}X'\stackrel{f'}{\to} S$, where $\sigma$ is elementary; then $X'$ is a $\Q$-factorial Mori dream space, and we conclude by induction on $\rho_X-\rho_S$.

 If $\dim X\geq 3$, then $f$ is of fiber type and by \cite[Prop.~2.13]{fibrations} it can be factored as
  $$X\stackrel{g}{\dasharrow}T\stackrel{h}{\la}S,$$
  where $g$ is a special rational contraction of fiber type and $h$ is birational. Then $T$ is $\Q$-factorial, so that
  $S$ is $\Q$-factorial  by the first part of the proof.
\end{proof}
\begin{lemma}\label{special}
Let $X$ be a projective, normal, and $\Q$-factorial Mori dream space. Let $f\colon X\dasharrow Z$ be a rational contraction of fiber type, and  $\tau$ the smallest face of $\Eff(X)$ containing $f^*\Eff(Z)$. The following are equivalent:
\begin{enumerate}[$(i)$]
\item $f$ is special;
\item $\tau\cap\Mov(X)=f^*\Mov(Z)$;
\item $\dim(\tau\cap\Mov(X))=\rho_Z$.
\end{enumerate}
\end{lemma}
\begin{proof}
We show $(i)\Rightarrow(ii)$. Up to composing with a SQM of $X$, we can assume that $f$ is regular. Let $F\subset X$ be a general fiber of $f$; by \cite[Lemma 2.21]{eff} we have $\tau=\Eff(X)\cap \N(F,X)^{\perp}$, thus 
$\tau\cap\Mov(X)=\Mov(X)\cap \N(F,X)^{\perp}$.

If $B$ is a movable divisor on $Z$, then $f^{-1}(B)\cap F=\emptyset$, thus $f^*B\cdot C=0$ for every curve $C\subset F$, and $[f^*B]\in \N(F,X)^{\perp}$. Moreover if $B_0\subset Z$ is the stable base locus of the linear system $|B|$, then $\codim B_0\geq 2$, and since $f$ is special we have $\codim f^{-1}(B_0)\geq 2$. Since the stable base locus of 
$|f^*B|$ is contained in $f^{-1}(B_0)$, we have  $[f^*B]\in \Mov(X)\cap \N(F,X)^{\perp}$. Hence $f^*\Mov(Z)\subseteq \Mov(X)\cap \N(F,X)^{\perp}$.

Conversely, let $D$ be a divisor on $X$ with $[D]\in \Mov(X)\cap \N(F,X)^{\perp}$. By Rem.~\ref{referee}, up to replacing $D$ with a divisor with class in 
$\R_{\geq 0}[D]$, we can assume that $D$ is a prime divisor. Since $D\cdot C=0$ for every curve $C\subset F$, we must have $D\cap F=\emptyset$ and $f(D)\subsetneq Z$; since $f$ is special, $f(D)$ is a prime divisor, and it is movable. If $f^{-1}(f(D))$ is reducible, then by \cite[Cor.~2.18]{fibrations} every irreducible component of $f^{-1}(f(D))$ is a fixed prime divisor; since $D$ is movable and is a component of $f^{-1}(f(D))$, we conclude that $f^{-1}(f(D))=D$ and $D=\lambda f^*(f(D))$ for some $\lambda\in\Q_{>0}$, so that $[D]\in f^*\Mov(Z)$, and we get $(ii)$.

\smallskip

The implication $(ii)\Rightarrow(iii)$ is clear. We show $(iii)\Rightarrow (i)$. 
By \cite[Prop.~2.13]{fibrations} $f$ can be factored as
$X\stackrel{g}{\dasharrow}T\stackrel{h}{\to} Z$, where $g$ is a special rational contraction of fiber type and $h$ is birational. Up to composing with a SQM of $X$, we can assume that $g$ is regular; then $g$ and $f$ have the same general fiber $F\subset X$, and again $\tau=\Eff(X)\cap \N(F,X)^{\perp}$.
By the first part of the proof 
we have $\Mov(X)\cap \N(F,X)^{\perp}=g^*\Mov(T)$, thus
$$\rho_Z=\dim\bigl(\tau\cap\Mov(X)\bigr)=\dim\bigl(\Mov(X)\cap \N(F,X)^{\perp}\bigr)=\rho_T,$$ therefore  $h$ is an isomorphism and $f$ is special. 
\end{proof} 
\subsection{Movable faces of the effective cone}\label{movable}
\noindent Let $X$ be a projective, normal, and $\Q$-factorial Mori dream space.
\begin{definition}
  Let $\tau$ be a proper face of $\Eff(X)$. We say that $\tau$ is a {\bf movable face} if the relative interior of $\tau$ intersects $\Mov(X)$.
\end{definition}
There exists a movable face of $\Eff(X)$ if and only if $\partial\Eff(X)\cap\Mov(X)\neq\{0\}$, if and only if $X$ has some non-zero, non-big movable  divisor.

If a (non-zero) face of $\Eff(X)$ is not fixed, it always contains a movable face.

We say that a movable face $\tau$ is {\bf minimal} if every proper face of $\tau$ is not movable, equivalently if every proper face of $\tau$ is fixed.

We recall that  the Mori chamber decomposition of $X$ is a fan in $\Nu(X)$, supported on the cone $\Mov(X)$, given by the cones $f^*\Nef(Z)$ for all rational contractions $f\colon X\dasharrow Z$; the cones of the Mori chamber decomposition are in
bijection with the rational contractions of $X$.
\begin{prg}\label{setup}
Let $\tau$ be a movable face of $\Eff(X)$. We associate to $\tau$ a special rational contraction of fiber type of $X$, as follows.

The cone $\tau\cap\Mov(X)$ is a non-zero face of the movable cone, so it is a union of cones of the Mori chamber decomposition of $X$. Let us choose a cone $\eta$ of the Mori chamber decomposition such that $\eta\subseteq\tau\cap\Mov(X)$ and $\dim\eta=\dim(\tau\cap\Mov(X))$; note that $\eta$ must intersect the relative interior of $\tau$, because $\tau$ is movable.

Then $\eta$ yields a rational contraction $f\colon X\dasharrow Z$, 
  of fiber type because $\eta$ is contained in the boundary of the effective cone.
We  have: 
$$\eta=f^*\Nef(Z),\quad
\rho_Z=\dim\eta=\dim\bigl(\tau
\cap\Mov(X)\bigr),\quad \text{and }\ \tau
\cap\Mov(X)\subset f^*\Nu(Z).$$

We show that $f^*\Eff(Z)\subseteq \tau$. Indeed since $\tau$ is a face of $\Eff(X)$, there exists $\gamma\in\Eff(X)^{\vee}$ such that $\tau=\Eff(X)\cap\gamma^{\perp}$. Then $f^*\Nef(Z)=\eta\subseteq\tau\subseteq\gamma^{\perp}$, thus $f^*\Nu(Z)\subseteq\gamma^{\perp}$ and $f^*\Eff(Z)\subseteq \gamma^{\perp}\cap\Eff(X)=\tau$.

On the other hand
 $\eta\subseteq f^*\Eff(Z)$, therefore $f^*\Eff(Z)$ intersects the relative interior of $\tau$, and $\tau$ is the minimal face of $\Eff(X)$ containing $f^*\Eff(Z)$. In particular $f$ is special by Lemma \ref{special}.

Let $X\dasharrow\w{X}\stackrel{\tilde{f}}{\to}Z$ be a factorization of $f$ as a SQM followed by a regular  contraction, and $F\subset \w{X}$ a general fiber of $\tilde{f}$.
By  \cite[Lemma 2.21]{eff} we deduce that, under the natural identification $\Eff(X)\cong\Eff(\w{X})$, we have: 
\stepcounter{thm}
\begin{equation}\label{estate}
\tau=\Eff(\w{X})\cap\N(F,\w{X})^{\perp}\quad\text{and}\quad
\rho_X=\dim\tau+\dim\N(F,\w{X})\leq \dim\tau+\rho_F.
\end{equation}
Thus given a prime divisor $D\subset X$, we have $[D]\in\tau$ if and only if 
the transform $\w{D}\subset\w{X}$  does not dominate $Z$ under $\tilde{f}$. Indeed $[\w{D}]\in\N(F,\w{X})^{\perp}$ is equivalent to $\w{D}\cap F=\emptyset$.
\end{prg}
\begin{remark}\label{stress2}
  Let $X$ be a smooth Fano $4$-fold with
  a small elementary contraction $g\colon X\to Y$. Then $\NE(g)\not\subset\mov(X)=\Eff(X)^{\vee}$, hence there exists a one-dimensional face $\tau$ of $\Eff(X)$ such that $\tau\cdot\NE(g)<0$.

  If $\tau$ is movable, we apply the construction above; by Lemma \ref{basic2} we can take $\w{X}$ smooth and $\tilde{f}$ $K$-negative. Then $F$ is smooth, Fano, with $\dim F\leq 3$, therefore $\rho_F\leq 10$ (see \cite[\S 12.6]{fanoEMS}) and
  $\rho_X\leq 11$ by \eqref{estate}.

  If instead $\tau$ is fixed, then there exists a fixed prime divisor $D\subset X$ such that $D\cdot\NE(g)<0$.
\end{remark}
\subsection{Rational contractions of fiber type of Fano $4$-folds}\label{fumo}
\noindent Fano $4$-folds with a rational contraction of fiber type have been studied in \cite{fibrations}; in particular we will need the following.
\begin{thm}[\cite{fibrations}, Th.~6.1]\label{3fold}
Let $X$ be a smooth Fano $4$-fold which is not a product of surfaces. 
If $X$ has a rational contraction onto a $3$-fold, then 
 $\rho_X\leq 12$.
\end{thm}
In the rest of this subsection we focus on Fano $4$-folds having a rational contraction onto a surface or $\pr^1$; in this case we do not have an analog of Th.~\ref{3fold}, and we will prove several results
which we will use in the rest of the article
to tackle this situation.
\begin{lemma}\label{ufficio}
  Let $X$ be a smooth Fano $4$-fold with $\rho_X\geq 7$ and $f\colon X\dasharrow S$ a quasi-elementary rational contraction onto a surface. Then $S$ is a smooth del Pezzo surface.
\end{lemma}
\begin{proof} 
By \cite[Prop.~4.1 and its proof]{eff} we know that
 $S$ is smooth, rational, and a Mori dream space.
We show that $-K_S\cdot\NE(g)>0$ for every 
elementary contraction $g\colon S\to S_1$; this implies that   $-K_S$ is ample.
If $g$ is of fiber type, then either $S\cong\pr^2$, or $S$ is  isomorphic to a Hirzebruch surface and $g$ is a $\pr^1$-bundle, hence $-K_S\cdot\NE(g)>0$.

Let assume that $g$ is birational,
so that $C:=\Exc(g)$ is an irreducible curve, and  $D:=f^*C$ is a fixed prime divisor, see \cite[Rem.~3.10]{eff}. Moreover  $S_1$ has (at most) a rational singularity at the point $g(C)$ (see Lemma \ref{basic2}).
Consider a factorization of $f$ as in Lemma \ref{basic2}:
$$X\stackrel{\psi}{\dasharrow}X'\stackrel{f'}{\la}S$$
where $\psi$ is a SQM.
We know that $X\smallsetminus\dom\psi$ is a finite union of exceptional planes (see Lemma \ref{basic1}).
Then in $X$ there exists a curve $\Gamma\equiv C_D$, $\Gamma\subset D\cap \dom\psi$; this follows from Rem.~\ref{ikea} if $D$ is of type $(3,2)$, and from the description in Th.-Def.~\ref{long} in the other cases.

Let $\Gamma'\subset D'\subset X'$ be the transforms of $\Gamma\subset D\subset X$. Then
$$  -1=D\cdot\Gamma=D'\cdot\Gamma'=(f')^*C\cdot\Gamma'=C\cdot (f')_*\Gamma'.$$
This implies that $f'(\Gamma')$ is not a point, thus $f'(\Gamma')=C$ and
$(f')_*\Gamma'=mC$ with $m\in\Z_{\geq 1}$. We get $m C^2=-1$, hence $m=1$ and
$C^2=-1$.
We conclude that $g$ is  the blow-up of a smooth point and
$-K_S\cdot \NE(g)>0$.
\end{proof}
\begin{lemma}[\cite{eff}, proof of Cor.~3.9]\label{regular}
Let $X$ be a smooth Fano $4$-fold, $\ph\colon X\dasharrow \w{X}$ a SQM, and $f\colon\w{X}\to Z$ a $K$-negative contraction of fiber type. If the general fiber $F$ of $f$ is either a del Pezzo surface with $\rho_F=9$, or a product $S\times \pr^1$ with $S$ a del Pezzo surface with $\rho_S=9$, then $\ph$ is an isomorphism.
\end{lemma}
\begin{lemma}\label{sushi}
  Let $X$ be a smooth Fano $4$-fold with $\rho_X\geq 7$,
$f\colon X\dasharrow Z$  a rational contraction of fiber type, and $H$ an effective divisor in $X$ such that $[H]\in f^*\Nu(Z)$ and 
 $\Supp H$ has a component $E$ which is a fixed prime divisor of type $(3,2)$.  Then one of the following holds:
  \begin{enumerate}[$(i)$]
  \item $\N(E,X)\subsetneq\N(X)$;
  \item $\dim Z=3$;
    \item $Z\cong\pr^2$ and $f$ is equidimensional.
    \end{enumerate}
\end{lemma}
\begin{proof}
 Let 
  $\w{X}\dasharrow X$ be a SQM such that the composition $\tilde{f}\colon \w{X}\to Z$ is regular and $K$-negative (see Lemma \ref{basic2}), and $\w{E}\subset\w{X}$ the transform of $E$.

  We have $\N(\tilde{f}(\w{E}),Z)=\tilde{f}_*(\N(\w{E},\w{X}))$, thus
  $\codim \N(\w{E},\w{X})\geq\codim\N(\tilde{f}(\w{E}),Z)$.
  Since $H$ is the pullback of an effective $\Q$-divisor of $Z$, $\tilde{f}(\w{E})\subseteq\tilde{f}(\Supp H)\subsetneq Z$.
  If $\dim Z=1$, or $\dim Z=2$ and $\rho_Z>1$, this easily implies that
$\codim\N(\tilde{f}(\w{E}),Z)>0$, therefore
  $\codim \N(\w{E},\w{X})>0$ and we get $(i)$ by Lemma \ref{cena}.
 
If $\dim Z=2$ and $f$ is not equidimensional, by \cite[Prop.~2.13]{fibrations} we can factor $f$ as $X\stackrel{g}{\dasharrow}T\to Z$ where $T$ is a surface with $\rho_T>\rho_Z$, and again $E$ cannot dominate $T$, so we get $(i)$ as before.

Finally if  $\dim Z=2$, $\rho_Z=1$, and $f$ is equidimensional, then $Z$ is  smooth by \cite[Lemma 4.3]{fibrations}, and being a rational surface, we have  $(iii)$. 
\end{proof}
\begin{lemma}\label{matrimonio}
Let $X$ be a smooth Fano $4$-fold with $\rho_X\geq 7$ and $f\colon X\dasharrow S$ an equidimensional rational contraction onto a surface. If $f$ is not quasi-elementary and $\rho_S>1$, then $X$ has a fixed prime divisor $E$ of type $(3,2)$ such that $\N(E,X)\subsetneq\N(X)$.
\end{lemma}
\begin{proof}
  Since $f$ is not quasi-elementary, there exists an irreducible curve $\Gamma\subset S$ such that $f^*\Gamma$ is reducible. By \cite[Lemma 5.2]{fibrations} there is a fixed prime divisor $E\subset X$ of type $(3,2)$ which is a component of $f^*\Gamma$. Then the statement follows from Lemma \ref{sushi}.
\end{proof}
\begin{lemma}\label{27}
  Let $X$ be a smooth Fano $4$-fold with $\rho_X\geq 7$ and $f\colon X\dasharrow S$ a rational contraction onto a surface with $\rho_S=1$. Suppose that there is a unique prime divisor $D\subset X$ contracted to a point by $f$, and that $D$ is fixed not of type $(3,2)$.
  Then one of the following holds:
  \begin{enumerate}[$(i)$]
  \item   $X$ has a fixed prime divisor $E$ of type $(3,2)$ such that $\N(E,X)\subsetneq\N(X)$;
    \item $\rho_X\leq 10$.
    \end{enumerate}
\end{lemma}
\begin{proof}
 By \cite[Prop.~2.13]{fibrations}  we can factor $f$ as $X\stackrel{g}{\dasharrow} T\stackrel{h}{\to} S$ where $g$ is equidimensional and $h$ is birational.
 The surfaces $T$ and $S$ are $\Q$-factorial by Lemma \ref{Qfactorial}, so that $\rho_T-\rho_S$ is the number of irreducible components of $\Exc(h)$. 

Let $X'\dasharrow X$ be a SQM such that the composition
$g'\colon X'\to T$ is regular and $K$-negative (see Lemma \ref{basic2}), $F\subset X'$ a general fiber of $g'$, and $D'\subset X'$ the transform of $D$. 
Since $D'$ is the unique prime divisor contracted to a point by $h\circ g'\colon X'\to S$, we deduce that $D'=(g')^{-1}(\Exc(h))$, $\Exc(h)$ is irreducible, and $\rho_T=2$.

  Let $R$ be an extremal ray of $\NE(X')$ such that $(g')_*R=\NE(h)$ (see \cite[\S 2.5]{fanos}). Then for every curve $C\subset X'$ with $[C]\in R$ we have $g'(C)=\Exc(h)$, therefore 
$\Lo(R)\subseteq D'$. Moreover if $F\subset X'$ is a non-trivial fiber of the contraction of $R$, $g'$ is finite on $F$, and $g'(F)=\Exc(h)$, so that $\dim F=1$.
If $-K_{X'}\cdot R>0$, then
$R$ is of type $(3,2)$ by \cite[Th.~1.2]{wisn},
and $D'=\Lo(R)$, a contradiction by Th.-Def.~\ref{long}$(e)$. Therefore $-K_{X'}\cdot R\leq 0$ and $X'$ is not Fano. This implies that $\rho_F\leq 8$ by Lemma \ref{regular}.

Now if $g$ is quasi-elementary,  by \cite[Cor.~2.16]{fibrations} we get $\rho_X\leq\rho_F+\rho_T\leq 10$, namely $(ii)$.
  Instead if $g$ is not quasi-elementary, then we get $(i)$ by Lemma \ref{matrimonio}.
\end{proof}
\begin{lemma}\label{detail2}
  Let $X$ be a smooth Fano $4$-fold  with  $\delta_X\leq 1$ and $f\colon X\to S$ a contraction onto a surface. Then 
  either $\rho_S=2$ and $f$ is equidimensional, or $\rho_S=1$ and there is at most one prime divisor contracted to a point.
\end{lemma}
\begin{proof}
For any prime divisor $D\subset X$ such that $f(D)\subsetneq S$ we have
$$1\geq\codim\N(D,X)\geq\codim\N(f(D),S)\geq\rho_S-1,$$
hence $\rho_S\leq 2$. 
If $\rho_S=2$, then $f(D)$ must always be a curve, and $f$ is equidimensional. 

Suppose that
$\rho_S=1$ and $f$ is not equidimensional, so there exists a prime divisor $D$ contracted to a point. 

We note that in our setting $f(D)=\{pt\}$ is equivalent to the pushforward of $D$ (as a cycle) being zero, and this is invariant under linear equivalence. Since there cannot be a positive dimensional linear system of divisors contracted to points, $D$ must be a fixed divisor.

Let us take a general very ample curve $A\subset S$ and consider the prime divisor $f^{-1}(A)$. We have $f^{-1}(A)\cap D=\emptyset$, hence $\N(f^{-1}(A),X)\subseteq D^{\perp}$ (see Rem.~\ref{easy}) and $\codim\N(f^{-1}(A),X)\leq\delta_X\leq 1$, so that $\N(f^{-1}(A),X)= D^{\perp}$.

We show that $D$ is unique. Indeed, if $B\subset X$ is another prime divisor contracted to a point, as above we get 
 $\N(f^{-1}(A),X)= B^{\perp}$, hence $D^{\perp}=B^{\perp}$. This means that the classes $[D]$ and $[B]$ are multiples in $\Nu(X)$,
 and being $D$ a fixed divisor, it implies that $B=D$.
\end{proof}
\begin{lemma}\label{DAD}
Let $X$ be a smooth Fano $4$-fold, $X\dasharrow\w{X}$ a SQM, and $f\colon\w{X}\to\pr^1$ a $K$-negative contraction with general fiber $F$.
Suppose that $F\cong\pr^1\times S$ and that $\dim\N(F,\w{X})=\rho_F$.
Then $X$ has a rational contraction onto a $3$-fold.
\end{lemma}
\begin{proof}
Set $C:=\pr^1\times\{pt\}\subset F$. Then $\ma{N}_{C/F}\cong\ol_{\pr^1}^{\oplus 2}$, $(\ma{N}_{F/\w{X}})_{|C}\cong\ol_{\pr^1}$, and $\ma{N}_{C/\w{X}}\cong\ol_{\pr^1}^{\oplus 3}$. This implies that $\Hilb(\w{X})$ is smooth of dimension $3$ at the point $[C]$; let $W$ be the irreducible component of $\Hilb(\w{X})$ containing $[C]$.

Any point in $W$ yields 
 an effective and connected one-cycle $\Gamma$ in $\w{X}$ such that $\Gamma\equiv C$.
We show that
 if $\Gamma\cap C\neq \emptyset$, then $\Gamma=C$. 

 We have $f(\Gamma)=\{pt\}$, and $\Gamma$ intersects the fiber $F$, so that $\Gamma\subset F$.
Let us consider the natural linear map $\N(F)\twoheadrightarrow\N(F,\w{X})\subseteq\N(\w{X})$: since $\dim\N(F,\w{X})=\rho_F$, this map is injective. Therefore  $\Gamma\equiv_{F} C$, thus $\Gamma$ is supported on a fiber of the projection $F\to S$. Since $\Gamma$ intersects $C$, we conclude that $\Gamma=C$.

Let $\w{X}_0\subset \w{X}$ be the open subset where the fibers of $f$ satisfy the same assumptions as $F$, and $W_0\subset W$ the open subset parametrizing the curves $\pr^1\times\{pt\}$ in the fibers of $f_{|\w{X}_0}$.
Let $\ma{C}_0\to W_0$ be the universal family and $e\colon \ma{C}_0\to \w{X}$  the natural map. Then $W_0$ and $\ma{C}_0$ are smooth and
 $e\colon \ma{C}_0\to \w{X}_0$ is bijective, and since $\w{X}$ is smooth, we conclude that $e$ is an isomorphism and there is a projective morphism $g_0\colon \w{X}_0\to W_0$. Then $X$ has a rational contraction onto a $3$-fold, see \cite[proof of Th.~1.2]{fibrations}.
\end{proof}
\begin{proposition}\label{usseaux}
Let $X$ be a smooth Fano 4-fold with $\rho_X\geq 7$ and  $\delta_X\leq 1$, $\tau$ a movable face of $\Eff(X)$, and $f\colon X\dasharrow Z$ an associated  rational contraction as in \ref{setup}. 
Then one of the following holds:
\begin{enumerate}[$(i)$]
\item
 $X$ has a rational contraction onto a $3$-fold;
\item $X$ has a fixed prime divisor $D$
 of type $(3,2)$ such that $\N(D,X)\subsetneq\N(X)$;
\item $\rho_X\leq 8+\dim\tau$, $Z\cong\pr^1$,  and $\tau$ does not contain classes of fixed prime divisors of type $(3,2)$;
\item $\rho_X=10$ and
 $f\colon X\to\pr^1$ is regular with general fiber $S\times\pr^1$, where $S$ is a del Pezzo surface with $\rho_S=9$;
\item $\rho_X\leq  
9+\dim\tau$, $Z\cong\pr^2$, and every fixed prime divisor with class in $\tau$ is 
 of type 
$(3,2)$ or $(3,1)^{sm}$;
\item $\tau$ is not minimal, $f$ is quasi-elementary, and $\dim Z=2$. 
\end{enumerate}
  Moreover in cases $(i)$ and $(ii)$ we have $\rho_X\leq 12$, and if $\rho_X=12$, then $X$ has a rational contraction onto a $3$-fold.
\end{proposition}
\begin{proof}
We apply the set-up as in \ref{setup}, and keep the same notation. We factor $f$ as $X\dasharrow\w{X}\stackrel{\tilde{f}}{\to}Z$ where $\tilde{f}$ is $K$-negative (see Lemma \ref{basic2}), so that the general fiber $F\subset\w{X}$ is a smooth Fano variety.
If $\dim Z=3$ we get $(i)$.

\smallskip

Suppose that $Z=\pr^1$. If there exists a fixed prime divisor $D$ of type $(3,2)$ such that $[D]\in\tau$,  then its transform $\w{D}\subset\w{X}$ is contained in a fiber of $\tilde{f}$. Thus $\N(\w{D},\w{X})\subseteq\ker\tilde{f}_*\subsetneq\N(\w{X})$, and we get $(ii)$ by Lemma \ref{cena}. 

Assume that $\tau$ does not contain classes of fixed prime divisors of type $(3,2)$.
If $\dim\N(F,\w{X})\leq 8$, we get $(iii)$ by \eqref{estate}.
If $\dim\N(F,\w{X})\geq 9$ and $\dim\N(F,\w{X})=\rho_F$, then $F\cong\pr^1\times S$ (see \cite[\S 12.6]{fanoEMS}), and we get $(i)$ by Lemma \ref{DAD}. 

The last possibility is $9=\dim\N(F,\w{X})<\rho_F=10$, and 
 $F\cong\pr^1\times S$ where $\rho_S=9$. In this case  $X\cong\w{X}$ by Lemma \ref{regular}, 
and $f\colon X\to\pr^1$ is regular. Thus $0<\codim\N(F,X)\leq\delta_X=1$,
 and we conclude that $\codim\N(F,X)=1$ and $\rho_X=10$, so we get $(iv)$.

\smallskip

Finally suppose that $\dim Z=2$, so that $f$ is an equidimensional rational contraction, and $Z$ is a smooth rational surface by \cite[Lemma 4.3]{fibrations}.

Suppose first that 
$\rho_Z>1$. If $f$ is not quasi-elementary,  we get $(ii)$ by Lemma \ref{matrimonio}.

If $f$ is quasi-elementary, then $f^*\Eff(Z)$ is a face of $\Eff(X)$ by \cite[Prop.~2.22]{eff}, thus $f^*\Eff(Z)=\tau$. On the other hand there is a contraction $Z\to\pr^1$, so that the boundary of $\Eff(Z)$ contains some non-zero movable divisor, therefore the boundary of $\tau$ contains some non-zero movable divisor. Hence $\tau$ is not minimal, 
 and we get $(vi)$.

Finally if $\rho_Z=1$, then $Z\cong \pr^2$. If $D\subset X$ is a fixed prime divisor with $[D]\in\tau$, then the image of $D$ in $\pr^2$ is an irreducible curve $\Gamma$, and $D$ is a component of $f^*\Gamma$. Since $f^*\Gamma$ is movable, while $D$ is fixed, $f^*\Gamma$ is reducible.  By \cite[Lemma 4.10]{fibrations} $D$ must be of type $(3,2)$ or $(3,1)^{sm}$.
Moreover $F$ is a del Pezzo surface, so that $\rho_F\leq 9$ and $\rho_X\leq 9+\dim\tau$ by \eqref{estate},
 so we get $(v)$.

The last statement follows from Theorems \ref{3fold} and \ref{natale} (note that $X$ is not a product of surfaces, see Rem.~\ref{delta}).
\end{proof}
\section{Constructing divisors covered by lines}\label{divlines}
\noindent Let $X$ be a smooth Fano $4$-fold with $\rho_X\geq 7$. In this section we show that, given a fixed prime divisor $D\subset X$ of type $(3,1)^{sm}$ or $(3,0)^Q$,\footnote{Note in particular that $X$ is not a product of surfaces, see Ex.~\ref{2022}.}  and $L\subset D$
an exceptional plane, then there exists a prime divisor $B$ covered by a family of lines $V$ such that $[V]+C_L\equiv C_D$ (Prop.~\ref{importante}). The construction of this family of lines if based on the explicit geometry of the divisor $D$, and we also use the results of Section \ref{lines}. 
By Rem.~\ref{generators} $D$ contains at least $\rho_X-4$ exceptional planes with distinct classes $[C_L]$, so that we obtain many distinct prime divisors covered by lines. This construction will be important in the rest of the paper. Then we give some properties of $D$ and $B$ depending on the different settings.
\begin{proposition}\label{importante}
Let $X$ be a smooth Fano $4$-fold with $\rho_X\geq 7$, $D\subset X$ a fixed prime divisor of type $(3,1)^{sm}$ or $(3,0)^Q$, and $L\subset D$
an exceptional plane. Then there exists a family of lines $V$ 
 such that 
$B:=\Lo V$ is a divisor different from $D$,  $[V]+C_L\equiv C_D$, and $B\cdot C_L>0$. Moreover $L'\not\subset B$ for every exceptional plane $L'\subset D$.
\end{proposition}
\begin{proof}
  Let $X\stackrel{\ph}{\dasharrow} \w{X}\stackrel{\sigma}{\to} Y$ be the contraction associated to $D$ as in \S\ref{contraction}, and $\ell\subset \w{X}$ the exceptional curve corresponding to $L$ (see Lemma \ref{basic1}), so that $\w{D}\cdot\ell>0$ and $\ell\not\subset\w{D}$. Let
 $q\in\ell\cap \w{D}$.

If  $D$ is 
$(3,1)^{sm}$, then $\sigma$ is the blow-up of a smooth curve, and $q$ belongs to a fiber $S_0$ of $\sigma$. Let $\Gamma_0$ be a line in $S_0\cong\pr^2$ through $q$.

If $D$ is $(3,0)^{Q}$, then $\w{D}$ is an irreducible quadric, either smooth or with one singular point, which cannot be $q$ by  Rem.~\ref{vertex}.
Let  $\Gamma_0$ be a line in $\w{D}$ through $q$, and let
$S_0\subset\w{D}$ be the union of the lines through $q$, which is
  a quadric cone (if $\w{D}$ is smooth) or a reducible quadric surface, singular at $q$ (if $\w{D}$ is singular).

In both cases $-K_{\w{X}}\cdot \Gamma_0=2$, so that in the factorization of $\ph$ given in Lemma \ref{basic1}:
$$\xymatrix{X\ar@{-->}@/_1pc/[rr]_{\ph}&{\wi{X}}\ar[r]^g\ar[l]_f&{\w{X},}
 }$$
if $E\subset\wi{X}$ is the exceptional divisor over $L$ and $\ell$, by Lemma \ref{basic1}$(a)$ we have $E\cdot\wi{\Gamma}=1$ for the transform $\wi{\Gamma}\subset\wi{X}$ of $\Gamma_0$. This implies that $\Gamma_0$ intersects $\ell$ only at $q$, transversally, and does not intersect other exceptional curves. Moreover the transform $\Gamma\subset X$ of $\Gamma_0$
is a smooth rational curve with $-K_X\cdot\Gamma=1$, so that
there is a family of lines $V$ in $X$ containing the general $\Gamma$. 
The transform $S\subset D\subset X$ of $S_0$ yields a surface contained in $\Lo V$.

The curve $C_D\subset D$ is the transform of a general line in $S_0$ (if $D$ is $(3,1)^{sm}$), or of a general line in $\w{D}$ (if $D$ is $(3,0)^{Q}$), so we must 
have $C_D\equiv [V]+mC_L$ with $m\in\Z$. Intersecting with $-K_X$ we get $m=1$ and 
 $C_D\equiv [V]+C_L$.

We claim that, varying $\Gamma_0$ in $S_0$, we can find two disjoint $\Gamma$'s in $S$. Indeed if $D$ is 
$(3,1)^{sm}$, then the plane $T_qS_0\subset T_q\w{X}$ is the union of $T_q\Gamma_0$ for all $\Gamma_0$'s, and since $\Gamma_0$ and $\ell$ are transverse at $q$, $T_qS_0$ does not
contain $T_q\ell$, namely $S_0$ and $\ell$ meet transversally at $q$. Moreover $S_0\cap \ell=\{q\}$ and $S_0$ does not meet other exceptional curves (see \S\ref{contraction}).
Locally 
$g\colon\wi{X}\to\w{X}$ is the blow-up of $\ell$, therefore the strict transform of $S_0$ is isomorphic to the  blow-up of $S_0$ at $q$, and the lines $\Gamma_0$ get separated.

If $D$ is $(3,0)^{Q}$ and $\w{D}$ is singular, we can just repeat the same argument with each irreducible component of $S_0$. 
 Finally if  $D$ is $(3,0)^{Q}$ and $\w{D}$ is smooth, then $T_q\w{D}\subset T_q\w{X}$ is a hyperplane, and when $\Gamma_0$ varies, $T_q\Gamma_0$ describes a quadric cone surface in $T_q\w{D}$, not containing $T_q\ell$. 
If we choose $\Gamma_0$ and $\Gamma_0'$ such that the plane generated by 
 $T_q\Gamma_0$ and $T_q\Gamma_0'$ in $T_q\w{X}$ does not contain $T_q\ell$, then
the curves $\wi{\Gamma}$ and $\wi{\Gamma}'$ in $\wi{X}$ must meet $g^{-1}(q)\cong\pr( (T_q\w{X}/T_q\ell)^{\vee})$ at different points. Since $f_{|g^{-1}(q)}\colon
 g^{-1}(q)\to L$ is an isomorphism, the curves $\Gamma$ and $\Gamma'$ are disjoint in $X$.

 By Th.~\ref{uno} we deduce that $\dim V=2$ and $\dim\Lo V=3$.

 Since $D\cdot [V]+D\cdot C_L=D\cdot C_D=-1$, we have $[V]\neq[C_L]$, and $L\not\subset B:=\Lo V$ by Lemma \ref{excplane}$(b)$. On the other hand $B\cap L\neq \emptyset$ because $S\subset B$, so that $B\cdot C_L>0$. This also implies that $B\neq D$, because $D\cdot C_L<0$, hence $D\cdot [V]\geq 0$.
 If $L'\subset D$ is any exceptional plane, we have $D\cdot C_{L'}<0$, thus $[V]\neq [C_{L'}]$, and $L'\not\subset B$ again by Lemma \ref{excplane}$(b)$.
\end{proof}
\begin{lemma}\label{atletica}
Notation as in Prop.~\ref{importante}. Let $L'\subset D$ be an exceptional plane with $C_{L'}\not\equiv C_L$ , and $V'$ the corresponding family of lines given by Prop.~\ref{importante}, with locus $B'$. Suppose moreover that  $D$ is of type $(3,0)^{Q}$ and that both $B$ and $B'$ are fixed of type $(3,2)$. 
Then $B\cdot [V']>0$ and $B'\cdot [V]>0$.
\end{lemma}
\begin{proof}
Since  $C_D\equiv [V]+C_L\equiv [V']+C_{L'}$, we have $[V]\neq[V']$ and hence $B\neq B'$ by Lemma \ref{excplane}$(a)$. Moreover by the same lemma we have $[V]=[C_B]$ and $[V']=[C_{B'}]$.

Assume by contradiction that $B\cdot [V']=0$. Then by Lemma \ref{preistorico} we have $B'\cdot [V]=0$ and, if $B\cap B'\neq\emptyset$, every connected component of $B\cap B'$ is irreducible.

We keep the same notation as in the proof of Prop.~\ref{importante}.
Let us consider the transforms $\w{B}$, $\w{B}'$ in $\w{X}$, and let $\ell'\subset\w{X}$ be the exceptional curve corresponding to $L'$. By construction $\w{B}'$ contains the surface $S_0'$, given by the union of the lines $\Gamma_0'$ in $\w{D}$ through a point $q'\in\ell'\cap\w{D}$. The surfaces $S_0$ and $S_0'$ meet along a plane conic $\Lambda$, which does contain neither $q$ nor $q'$; in fact $\Lambda$ is contained in the open subset where $\ph\colon X\dasharrow\w{X}$ is an isomorphism.

If some line $\Gamma_0$ is not contained in $\w{B}'$, then its transform $\Gamma$ in $X$ is a curve of the family $V$ such that $\Gamma\cap B'\neq\emptyset$ and $\Gamma\not\subset B'$, contradicting $B'\cdot [V]=0$. Therefore every $\Gamma_0$ must be contained in $\w{B}'$, so that $S_0\subseteq \w{B}'$ and $S\subseteq B\cap B'$. Similarly
$S'\subseteq B\cap B'$, where $S'$ is the transform of $S_0'$. Hence we get $S\cup S'\subseteq B\cap B'$ and
 $S\cap S'\neq\emptyset$; this gives a reducible connected component of  $B\cap B'$, a contradiction.
\end{proof}
\begin{lemma}\label{zoom2}
Notation as in Prop.~\ref{importante}. Suppose moreover that $D$ is of type $(3,0)^{Q}$ and that $D\cdot C_L=-1$.
 Then  $B\cap L'=\emptyset$  for every exceptional plane $L'\subset D$ with $C_{L'}\not\equiv C_L$.
\end{lemma}
\begin{proof}
 Since $D\cdot C_D=D\cdot C_L=-1$ and $C_D\equiv C_L+[V]$, we have $D\cdot [V]=0$, thus every line of the family $V$ that meets $D$ must be contained in $D$. On the other hand $D\neq B$, thus  the general line of the family $V$ is disjoint from $D$, and $V$ yields a family of lines $V_{\w{X}}$ in $\w{X}$. We have $[V_{\w{X}}]\equiv\ell+C_{\w{D}}$, $\w{D}\cdot [V_{\w{X}}]=0$, and $\Lo V_{\w{X}}=\w{B}$, the transform of $B$.
Since $-K_{\w{X}}$ is not ample, the family  $V_{\w{X}}$ can have reducible members, containing some exceptional curve.

 Let us consider all the exceptional planes $L_1:=L,L_2,\dotsc,L_d\subset D$ such that 
 $C_{L_i}\equiv C_L\equiv C_D-[V]$, and let
  $\ell_1,\dotsc,\ell_d\subset\w{X}$ be the corresponding exceptional curves.
We show that for every one-cycle $\Lambda$ of the family $V_{\w{X}}$ 
one of the following holds:
\begin{enumerate}[$(i)$]
\item  $\Lambda$ is integral and  disjoint from $\w{D}$ and from every exceptional curve of $\w{X}$;
\item $\Lambda=\Gamma_0+\ell_i$ for some $i\in\{1,\dotsc,d\}$,
 where $\Gamma_0\subset\w{D}$ is a line meeting $\ell_i$, and $\Lambda$ does not meet exceptional curves except $\ell_i$.
\end{enumerate}

Indeed, let us consider $-K_{\w{X}}+2\w{D}=\sigma^*(-K_Y)$, so that $-K_{\w{X}}+2\w{D}$ is nef and $$(-K_{\w{X}}+2\w{D})^{\perp}\cap\NE(\w{X})=\R_{\geq 0}[C_{\w{D}}].$$
 We have  $(-K_{\w{X}}+2\w{D})\cdot \Lambda=1$, so we can write $\Lambda=\Lambda_1+\Lambda_0$ where $\Lambda_1$ is an integral curve with  $(-K_{\w{X}}+2\w{D})\cdot \Lambda_1=1$, and $\Lambda_0$ is an effective one-cycle with $(-K_{\w{X}}+2\w{D})\cdot\Lambda_0 =0$. In particular $\Lambda_1\not\subset\w{D}$, while $\Lambda_0$ is supported in $\w{D}$.

If $\Lambda_0=0$, then $\Lambda=\Lambda_1$ is integral and disjoint from $\w{D}$; moreover $\Lambda$ cannot intersect any exceptional curve (see Lemma \ref{basic1}$(b)$), and we get $(i)$.

If $\Lambda_0\neq 0$, then $\w{D}\cdot\Lambda_0<0$. Moreover we have $0=\w{D}\cdot \Lambda=\w{D}\cdot \Lambda_1+\w{D}\cdot \Lambda_0$, therefore $\w{D}\cdot \Lambda_1>0$. Together with  $(-K_{\w{X}}+2\w{D})\cdot \Lambda_1=1$, this implies that $-K_{\w{X}}\cdot \Lambda_1<0$, so that $\Lambda_1$ is an exceptional curve (see Lemma \ref{basic1}$(a)$), $-K_{\w{X}}\cdot \Lambda_1=-1$, $\w{D}\cdot \Lambda_1= 1$, and finally $\w{D}\cdot \Lambda_0=-1$. Hence $\Lambda_0\equiv C_{\w{D}}$ is integral and is a line in the quadric $\w{D}$; moreover $\Lambda_1\equiv [V_{\w{X}}]-C_{\w{D}}\equiv\ell$, so that $\Lambda_1=\ell_i$ for some $i\in\{1,\dotsc,d\}$. Since the 1-cycle $\Lambda$ is connected, $\Lambda_0$ must meet $\ell_i$. Finally $\Lambda_0$ cannot meet other exceptional curves (as shown in the proof of Prop.~\ref{importante}), and neither can $\Lambda_1$ (by Lemma \ref{basic1}$(b)$), so that we get $(ii)$.

\smallskip

 Now let $L'\subset D$ be an exceptional plane such that 
$C_{L'}\not\equiv C_L$, and $\ell'\subset\w{X}$ the corresponding exceptional curve, so that $\ell'\neq \ell_i$ for every $i=1,\dotsc,d$. By what precedes, we have $\ell'\cap\w{B}=\emptyset$, thus $L'\cap B=\emptyset$. 
\end{proof}
\begin{proposition}\label{yoga}
Notation as in Prop.~\ref{importante}. Suppose moreover that 
$B$ is fixed of type $(3,2)$ and adjacent to $D$, let
 $L'\subset D$ be an exceptional plane with $C_{L'}\not\equiv C_L$, and $V'$  the corresponding family given by  Prop.~\ref{importante}.
Then $D\cdot C_L=-1$, $D\cdot [V]=0$, and:
\begin{enumerate}[$\bullet$]
\item $B\cdot C_D=0$, $B\cdot C_L=1$,  $B\cdot [V']=0$, if $D$ is of type $(3,1)^{sm}$;
\item $B\cdot C_D=1$, $B\cdot C_L=2$, $B\cdot [V']=1$, if $D$ is of type $(3,0)^Q$.
\end{enumerate}
\end{proposition}
\begin{proof}
We keep the same notation as in the proof of Prop.~\ref{importante}.
Note that $[C_B]=[V]$ by Lemma \ref{excplane}$(a)$.
\begin{prg}\label{mobili}
By Lemma \ref{nigra}$(a)$ we have $D\cdot [V]=0$; since $C_D\equiv C_L+[V]$, this also yields $D\cdot C_L=-1$. Thus $\w{D}\cdot\ell=1$, and $\w{D}$ and $\ell$ meet transversally at the unique point $q$. This implies that $D$ is smooth around $L$ and $\ph_{|D}\colon D\dasharrow \w{D}$ is regular around $L$ and is just the blow-up of $q$, with exceptional divisor $L$.

The transform $S\subset D\subset X$ of the surface ${S}_0\subset \w{D}\subset\w{X}$
is isomorphic to either $\mathbb{F}_1$ (if $D$ is $(3,1)^{sm}$), or to $\mathbb{F}_2$ (if $\w{D}$ is a smooth quadric), or to the union of two copies of $\mathbb{F}_1$ (if $\w{D}$ is a singular quadric). 
\end{prg}
\begin{prg}
Let $f\colon X\to X'$ be the contraction of $\R_{\geq 0}[C_B]$, so that $\Exc(f)=B$, $f$ can have at most finitely many $2$-dimensional fibers, and $B$ is a smooth $\pr^1$-bundle outside these fibers. 

We show that $D$ is disjoint from the possible $2$-dimensional fibers of $f$. Indeed if $F$ is such a fiber and $F\cap D\neq \emptyset$, since $D\cdot C_B=0$, we have $D\cdot\Gamma=0$ for every curve $\Gamma\subset F$, so that $F\subset D$.
For every exceptional plane $L'\subset D$ we have $D\cdot C_{L'}<0$, therefore the classes 
$[C_{L'}]$ and $[C_B]$ cannot be proportional. Since $\N(F,X)=\R[C_B]$ and 
$\N(L',X)=\R[C_{L'}]$, we deduce that
  $\dim(L'\cap F)\leq 0$. 
Therefore there exists an irreducible curve $\Gamma_2\subset F$  disjoint from every exceptional
plane $L'\subset D$. Then $\Gamma_2$ is contained in the open subset where $\ph\colon X\dasharrow \w{X}$ is an isomorphism, and we conclude that $\Gamma_2\equiv mC_D$ for some $m\in\Z_{\geq 0}$, which is impossible because $[\Gamma_2]\in\NE(f)$.

We deduce that $B\cap D$ is covered by one-dimensional fibers of $f$  and $B\cap D\subset B_{reg}$. 
The irreducible components of 
$B\cap D$ are $f^{-1}(K)$ where $K$ is an irreducible component of the curve $f(B\cap D)$, and distinct irreducible components  intersect at most along fibers of $f$.
\end{prg}
\begin{prg}
Suppose that $D$ is $(3,1)^{sm}$. We show that $S$ is a connected component of $B\cap D$. By contradiction,
let $T$ be an irreducible component of $B\cap D$ such that $T\neq S$ and $T\cap S\neq\emptyset$. Then $T$ must contain some curve $\Gamma$, and $\Gamma\cap L\neq\emptyset$, so that $T\cap L\neq\emptyset$ and  $T\cap L$ is a curve (as $T\subset D$ and $L\subset D_{reg}$). 
Since $f(T)$ is an irreducible curve and $f_{|T}$ is a $\pr^1$-bundle, we also have $\rho_T=2$ and $\N(T)=\R [T\cap L]\oplus\R[\Gamma]$.

 If $\w{T}\subset \w{D}$ is the transform of $T$, then $\w{T}$ cannot contain exceptional curves, and
$\rho_{\w{T}}=1$.
 We have $S_0\cap \w{T}\neq\emptyset$, $S_0\neq \w{T}$, $S_0$ is a fiber of the $\pr^2$-bundle $\sigma_{|\w{D}}$, and $\rho_{\w{T}}=1$, which yields a contradiction. 

Thus every irreducible component of $\w{B}\cap \w{D}$ is a fiber of $\sigma$, therefore $\w{B}$ is disjoint from the general fiber of $\sigma_{|\w{D}}$; this yields $B\cdot C_D=\w{B}\cdot C_{\w{D}}=0$. Then $C_D\equiv C_B+C_L$ and $B\cdot C_B=-1$ imply that $B\cdot C_L=1$.

Finally set $B':=\Lo V'$. We have $[V']\neq [V]$, thus $B'\neq B$ by Lemma \ref{excplane}$(a)$, and $B\cdot [V']\geq 0$. On the other hand we also have $B\cdot C_{L'}\geq 0$ because $L'\not\subset B$ by Prop.~\ref{importante}, hence
$0=B\cdot C_D=B\cdot C_{L'}+B\cdot [V']$ implies that $B\cdot C_{L'}=B\cdot [V']=0$.
\end{prg}
\begin{prg}
Suppose now that $D$ is $(3,0)^{Q}$. The proof of Lemma \ref{zoom2} shows that $\w{B}\cap\w{D}=S^1_0\cup\cdots\cup S^d_0$, where $S^i_0$ is the union of the lines in $\w{D}$ through the point $q_i:=\w{D}\cap\ell_i$ for $i=1,\dotsc,d$, and $\ell_1,\dotsc,\ell_d$ are all the exceptional curves of $\w{X}$ numerically equivalent to $\ell=\ell_1$, so that $\w{D}\cdot\ell_i=1$ for every $i$.
In $X$ we have $B\cap D=S^1\cup\cdots\cup S^d$, where $S^i$ is the transform of $S_0^i$.

We show that $d=1$. Otherwise,
 for $i\neq j$ the surfaces $S_0^i$ and $S_0^j$ meet along a plane conic contained in the open subset where $\ph\colon X\dasharrow\w{X}$ is an isomorphism,
 so that $S^i$ and $S^j$ meet along a curve with class in $\R_{\geq 0}[C_D]$.
 This is impossible, because they should
  intersect along fibers of $f\colon X\to X'$. We conclude that $d=1$ and that $D\cap B=S$.

Suppose that $\w{D}$ is a smooth quadric, so that $S$ is irreducible and contained in $D_{reg}$ (see \ref{mobili}). We have $B_{|D}=mS$ for some $m\in\Z_{\geq 1}$. If $\Gamma\subset S$ is a curve of the family $V$, we have $-1=B\cdot \Gamma=m S\cdot\Gamma$, where the last intersection is in $D$. This implies that $m=1$ and $B\cdot C_L=S\cdot C_L=2$ (again the last intersection is in $D$), because $S\cap L$ is a conic in $L\cong\pr^2$.

Suppose now that $\w{D}$ is a quadric cone, with vertex $v_0$; then $S_0$ is a Cartier divisor in $\w{D}$. In $D$ we have $(\ph_{|D})^*(S_0)=S+2L$, so that again $S$ is Cartier in $D$. Moreover $S=H_1+H_2$ where $H_i\cong\mathbb{F}_1$ and $H_1\cap H_2$ is a common fiber of the $\pr^1$-bundles, which contains the singular point $v:=\ph^{-1}(v_0)$. Note that $v\not\in L$ and that $L\cup (S\smallsetminus \{v\})\subset D_{reg}$.

Write $B_{|D}=m_1H_1+m_2H_2$, with $m_i\in\Z_{\geq 1}$, and let $\Gamma_1\subset H_1$ be a curve of the family $V$ disjoint from $H_2$, so that $\Gamma_1\subset D_{reg}$. Then $-1=B\cdot\Gamma_1=m_1 H_1\cdot\Gamma_1=-m_1$, so that $m_1=1$. Similarly we see that $m_2=1$, and finally that $B_{|D}=S$; as before this yields $B\cdot C_L=2$.

  We have $C_D\equiv C_B+C_L$ and $B\cdot C_B=-1$, therefore $B\cdot C_D=1$.
Finally we have $D\cap L'=\emptyset$ by Lemma \ref{zoom2}, thus $1=B\cdot C_D= B\cdot [V']+ B\cdot C_{L'}=B\cdot [V']$.\qedhere\end{prg}
\end{proof}
\begin{lemma}\label{zoom}
  Let $X$ be a smooth Fano $4$-fold with $\rho_X\geq 7$,  $D\subset X$ a fixed prime divisor of type $(3,1)^{sm}$ or $(3,0)^Q$, and $E\subset X$ a fixed prime divisor of type $(3,2)$  such that $D\cap E\neq \emptyset$ and $D\cdot C_E=0$.
  If $D$ is of type $(3,1)^{sm}$, assume moreover that $E\cdot C_D=0$.

  Then there exists an exceptional plane $L\subset D$ such that $D\cdot C_L=-1$  and $C_E$ is the family of lines given by $D$ and $L$ as in Prop.~\ref{importante}. 
\end{lemma}
\begin{proof}
Let $X\dasharrow\w{X}\stackrel{\sigma}{\to}Y$ be the contraction associated to $D$ as in \S\ref{contraction}, and
 $E_{\w{X}}\subset\w{X}$, $E_Y\subset Y$ the transforms of $E$. We have $D\cap E\neq\emptyset$ and $E$ does not contain exceptional planes (see Rem.~\ref{ikea}), therefore $\w{D}\cap E_{\w{X}}\neq\emptyset$ and $\sigma(\w{D})\cap E_Y\neq\emptyset$.

Since $D\cdot C_E=0$, $D$ is disjoint from the general curve $C_E$, so that $Y$ has a family of lines $V_Y$ with locus $E_Y$. If $D$ is $(3,1)^{sm}$ and $[C]=[V_Y]$, then $C\subset E_Y$ and $E\cdot C_D=E_{\w{X}}\cdot C_{\w{D}}>0$,
against our assumptions. Thus  $[C]\neq [V_Y]$.

Let $\Gamma\subset Y$ be a curve of the family $V_Y$ such that $\Gamma\cap\sigma(\w{D})\neq\emptyset$ and $\Gamma\neq\sigma(\w{D})$.
By Rem.~\ref{vecchissimo} we have $\Gamma=\sigma(\ell)$ where $\ell\subset\w{X}$ is an exceptional curve; moreover if $\w{C}_E\subset\w{X}$ is the transform of the general $C_E\subset X$, we have 
$\w{C}_E\equiv \ell+C_{\w{D}}$ in $\w{X}$. Then in $X$ we get $C_D\equiv C_L+C_E$ where $L\subset D$ is the exceptional plane corresponding to $\ell$. Moreover $-1=D\cdot C_D=D\cdot (C_L+C_E)=D\cdot C_L$, and we get the statement.
\end{proof}
\begin{corollary}\label{pasquetta}
Let $X$ be a smooth Fano $4$-fold with $\rho_X\geq 7$, and 
$D$ and $E$ two adjacent fixed prime divisors, of type $(3,0)^{Q}$ and $(3,2)$ respectively, such that  $D\cap E\neq \emptyset$. Then  
$E\cdot C_D=1$.
\end{corollary}
\begin{proof}
By Lemma \ref{zoom} there exists an exceptional plane $L\subset D$ such that $C_E$ is the family of lines given by $D$ and $L$ as in Prop.~\ref{importante}, so that $E\cdot C_D=1$ by Prop.~\ref{yoga}.
\end{proof}  
\section{Constructing rational contractions of fiber type}\label{secfiber}
\noindent In this section we consider two situations where we can  construct a rational contraction of fiber type on $X$, and then prove that $\rho_X\leq 12$ using the results of Section \ref{fiber}. 
\begin{proposition}\label{test}
Let $X$ be a smooth Fano $4$-fold with $\rho_X\geq 8$ and $\delta_X\leq 1$,
 $D$ a fixed prime divisor of type $(3,1)^{sm}$ or $(3,0)^Q$, and $X\stackrel{\ph}{\dasharrow} \w{X}\stackrel{\sigma}{\to}Y$ the associated contraction as in \S\ref{contraction}.
Suppose that $Y$ contains a nef prime divisor
  $H$ covered by a family  of lines and such that $H\cap\sigma(\Exc(\sigma))=\emptyset$.
  Then 
  one of the following holds:
  \begin{enumerate}[$(i)$]
  \item  $X$ has a contraction onto a $3$-fold;
  \item $X$ contains a fixed prime divisor $E$ of type $(3,2)$ with $\N(E,X)\subsetneq\N(X)$;
  \item $\rho_X\leq 10$.
          \end{enumerate}
Moreover $\rho_X\leq 12$, and if $\rho_X=12$, then $X$ has a rational contraction onto a $3$-fold. \end{proposition}
\begin{proof}
Let $g\colon Y\to Z$ be the contraction defined by $mH$ for $m\in\mathbb{N}$, $m\gg 0$,
so that $H=g^*A$ where $A\subset Z$ is an ample prime Cartier divisor, and $\NE(g)=\NE(Y)\cap H^{\perp}$. 
If $\ell\subset\w{X}$ is an exceptional curve, then by Lemma \ref{vecchio} $\sigma(\ell)$ cannot meet any line in $Y$ outside $\sigma(\Exc(\sigma))$, so that $\sigma(\ell)\cap H=\emptyset$ and $[\sigma(\ell)]\in\NE(g)$. Moreover $H$ is contained in the open subset of $Y$ where the birational map $X\dasharrow Y$ is an isomorphism; 
let $H_X\subset X$ and $\w{H}\subset \w{X}$ be the transforms of $H$. Then
$\w{H}=\sigma^*(g^*A)$, and $H_X$ is  still nef, so that $h:=g\circ\sigma\circ\ph\colon X\to Z$ is regular, $H_X=h^*A$, and $\NE(h)=\NE(X)\cap H_X^{\perp}$. Moreover $H_X\cap D=\emptyset$, therefore $h(D)=\{pt\}$.
$$\xymatrix{X\ar[dr]_h\ar@{-->}[r]^{\ph}&{\w{X}}\ar[r]^{\sigma}&Y\ar[dl]^g\\
  &Z&  }$$

\medskip

We show that $g$ and $h$ are of fiber type.

If $D$ is $(3,1)^{sm}$, have $H\cdot C=0$ hence $[C]\in\NE(g)$.
If $[C]\in\mov(Y)$, then $g$ is of fiber type. Otherwise, 
$[C]$ generates an extremal ray of type $(3,2)$ of $\NE(Y)$ by \cite[Lemma 5.11(3)]{blowup}, let $E_1$ be its locus.
By \cite[Prop.~5.8 and its proof]{blowup}, there is fixed prime divisor $E_2\subset Y$ of type $(3,2)$, such that $E_2\cdot C>0$, so that
 $[C]+[C_{E_2}]\in\mov(X)$
 by Cor.~\ref{solfeggio}. If $\Gamma\subset E_2$ is an irreducible curve such that $\Gamma\equiv C_{E_2}$ and $\Gamma\cap C\neq\emptyset$, then by Rem.~\ref{vecchissimo} $\Gamma=\sigma(\ell)$ where $\ell\subset\w{X}$ is an exceptional curve, thus $[\Gamma]=[C_{E_2}]\in\NE(g)$,   $[C]+[C_{E_2}]\in\NE(g)$, and again 
$g$ is of fiber type.

If $D$ is $(3,0)^Q$, 
 let $V$ be a family of lines in $X$ given by $D$ and an exceptional plane $L\subset D$ as in Prop.~\ref{importante}, so that $\Lo V$ is a divisor and $C_D\equiv [V]+C_L$. Since $H_X$ is nef and $H_X\cdot C_D=0$, we have $H_X\cdot [V]=0$ and $[V]\in \NE(h)$. If $[V]\in\mov(X)$, then $h$ is of fiber type. Otherwise there exists a prime divisor $B$ such that $B\cdot [V]<0$, hence $B=\Lo V$ and by Lemma \ref{chitarre} $[V]$ generates an extremal ray of type $(3,2)$. 
By Rem.~\ref{generators} we can choose an exceptional plane $L'\subset D$ with $C_{L'}\not\equiv C_L$ and this yields another family of lines $V'$ with $[V']\neq [V]$. As before, $[V']\in\NE(h)$, and either $[V']\in\mov(X)$ and $h$ is of fiber type,
or  $[V']$ generates an extremal ray of type $(3,2)$, with locus $B'$.
Then $B\cdot [V']>0$ and $B'\cdot [V]>0$ by Lemma \ref{atletica}, so that $[V]+[V']\in\mov(X)$ by Lemma \ref{2faces}$(b)$, and again 
$h$ is of fiber type.

\medskip

If $\dim Z=3$, we get $(i)$.

We show that $\dim Z>1$. Otherwise, let $F\subset X$ be a general fiber of $h$.
We have $\N(F,X)\subseteq\ker h_*\subsetneq\N(X)$, and 
since $\delta_X\leq 1$, we deduce that  $\N(F,X)=\ker h_*$, and similarly $\N(D,X)=\ker h_*$. On the other hand $F\cap D=\emptyset$, thus $\N(F,X)\subseteq D^{\perp}$, which is impossible because $D\cdot C_D=-1$ and $[C_D]\in\N(D,X)=\N(F,X)$.

Finally suppose that $\dim Z=2$. Since $h$ is not equidimensional, by Lemma \ref{detail2} we have that $\rho_Z=1$ and $D$ is the unique prime divisor contracted to a point. Then Lemma \ref{27} yields $(ii)$ or $(iii)$.

The last statement follows from Th.~\ref{3fold} and \ref{natale} (note that $X$ is not a product of surfaces, see Ex.~\ref{2022} or Rem.~\ref{delta}).
\end{proof}
\begin{lemma}\label{eccellenza}
Let $X$ be a smooth Fano $4$-fold with $\rho_X\geq 7$ and $\delta_X\leq 1$, and $D\subset X$ a fixed prime divisor of type $(3,0)^Q$. Suppose that there are two fixed prime divisors $E_1,E_2$, of type $(3,2)$, both adjacent to $D$.
Then $\rho_X\leq 12$, and if $\rho_X=12$, then $X$ has a rational contraction onto a $3$-fold.
\end{lemma}
\begin{proof}
Let $i\in\{1,2\}$; we have $D\cdot C_{E_i}=0$ by Lemma \ref{nigra}$(a)$.
If $E_i\cap D=\emptyset$ for some $i\in\{1,2\}$, then $\N(E_i,X)\subseteq D^{\perp}\subsetneq\N(X)$, and we conclude by Th.~\ref{natale}. Thus we can assume that $E_i\cap D\neq \emptyset$ for $i=1,2$, so that by Lemma \ref{zoom} there exists an exceptional plane $L_i\subset D$ such that $C_{E_i}$ is the family of lines given by $D$ and $L_i$ as in Prop.~\ref{importante}.
Then Prop.~\ref{yoga} yields:
$$E_i\cdot C_D=E_1\cdot C_{E_2}=E_2\cdot C_{E_1}=1.$$

Consider now $H:=E_1+E_2+2D$; we show that $H$ is movable. By \cite[Lemma 5.29(2)]{blowup}, $[H]\in\Mov(X)$  if and only if $H\cdot C_G\geq 0$ for every fixed prime divisor $G\subset X$. This is clear if $G\neq E_1,E_2,D$; moreover
$H \cdot C_{E_1}=H \cdot C_{E_2}=H\cdot C_D=0$, therefore $H$ is movable. 
We also have $H\cdot (C_{E_1}+C_{E_2})=0$ and $[C_{E_1}+C_{E_2}]\in\mov(X)=\Eff(X)^{\vee}$ by Lemma \ref{2faces}$(b)$, so that $[H]\in\partial\Eff(X)$, namely $H$ is not big.

Let $f\colon X\dasharrow Z$ be the rational contraction of fiber type defined by $mH$ for $m\in\mathbb{N}$, $m\gg 0$. Since $H\cdot C_D=0$, $f$ contracts $D$ to a point, so that if $\dim Z=2$ $f$ is not equidimensional. Therefore by Lemma \ref{sushi} we have either $\dim Z=3$, or $\N(E_1,X)\subsetneq\N(X)$, and we get the statement by Th.~\ref{3fold} and \ref{natale} respectively  (note that $X$ is not a product of surfaces, see Ex.~\ref{2022} or Rem.~\ref{delta}). 
\end{proof}
\section{The case $(3,1)^{sm}$}\label{case31}
\noindent Let $X$ be a smooth Fano $4$-fold with $\rho_X\geq 7$ and $\delta_X\leq 1$; in this section we prove the following result, which implies
Th.~\ref{main} when $X$ has a fixed prime divisor $D$ of type $(3,1)^{sm}$.
\begin{thm}\label{finalmente}
  Let $X$ be a smooth Fano $4$-fold with $\rho_X\geq 7$ and $\delta_X\leq 1$, having a fixed divisor of type $(3,1)^{sm}$. Then $\rho_X\leq 12$, and if $\rho_X=12$, then $X$ has a rational contraction onto a $3$-fold. 
\end{thm}
Let us give an outline of the proof.
We consider the contraction $X\dasharrow Y$ associated to $D$, and
work in the smooth Fano $4$-fold $Y$. 
We show that we can reduce to the following situation:
\begin{enumerate}[--]
\item
 except possibly one, the fixed prime divisors $E_1,\dotsc,E_r$ of type $(3,2)$ of $Y$ satisfy $E_i\cdot C_{E_j}=0$ for $i\neq j$;
\item $Y$ contains at most one nef prime divisor covered by lines;
\item $\Eff(Y)$ is generated by classes of fixed prime divisors, none of type $(3,0)^{sm}$, and there is at least one not of type $(3,2)$, say $B$.
  \item 
    The divisors covered by lines given by $B$ and its exceptional planes, as in Section \ref{lines}, are all fixed of type $(3,2)$, except at most one.
     From the properties of these families of lines we deduce that $B$ is $(3,1)^{sm}$, so that every fixed prime divisor of $Y$ is of type $(3,2)$ or $(3,1)^{sm}$.
\end{enumerate}
    Then we
    use these fixed prime divisors to construct some movable, non-big divisors in $Y$, which yield some rational contractions of fiber type on $Y$ and $X$; these also allow to describe some facets of $\Eff(Y)$. Finally we conclude applying the results
    in \S \ref{fumo}.
\begin{proof}[Proof of Th.~\ref{finalmente}]
  We assume that $\rho_X\geq 8$, and note that $X$ is not a product of surfaces (see Ex.~\ref{2022} or Rem.~\ref{delta}).
  We also assume the following:
  \begin{enumerate}[$(a)$]
  \item $X$ has no rational contraction onto a $3$-fold;
  \item  $\N(E_X,X)=\N(X)$ for every  fixed prime divisor
     $E_X\subset X$ of type $(3,2)$;
\item $X$ has no fixed prime divisor of type $(3,0)^{sm}$.
   \end{enumerate}
  Otherwise, the statement follows from Th.~\ref{3fold}, \ref{natale}, or \ref{30}, respectively.
  \begin{prg}\label{Umask}
    Let  $X\dasharrow\w{X}\stackrel{\sigma}{\to} Y$ be the contraction associated to $D$ as in \S\ref{contraction}, so that $Y$ is a smooth Fano $4$-fold with $\rho_Y\geq 7$, and 
$\sigma$ is the blow-up of a smooth irreducible curve $C\subset Y$.
    We have the following:
\begin{enumerate}[$(a')$]
\item $Y$ has no rational contraction onto a $3$-fold;
\item  $\N(E,Y)=\N(Y)$ for every  fixed prime divisor
     $E\subset Y$ of type $(3,2)$;
   \item $Y$ has no fixed prime divisor of type $(3,0)^{sm}$.
   \end{enumerate}
\end{prg}
   \begin{proof}
 The implication $(a)\Rightarrow (a')$ is clear.
  For $(b)\Rightarrow (b')$, let $E\subset Y$ be a fixed prime divisor  of type $(3,2)$, and let us consider its transforms 
  $E_X\subset X$ and $E_{\w{X}}\subset\w{X}$. Then $E_X$ is a fixed prime divisor of type $(3,2)$ by Lemma \ref{pfizer}, and $\dim\N(E_X,X)=\dim\N(E_{\w{X}},\w{X})$ by Lemma \ref{cena}, thus $(a)$ yields $\N(E_{\w{X}},\w{X})=\N(\w{X})$. On the other hand
  $E=\sigma(E_{\w{X}})$, hence $\N(E,Y)=\sigma_*(\N(E_{\w{X}},\w{X}))=\N(Y)$.
  
 Finally if $Y$ 
 has a fixed prime divisor of type $(3,0)^{sm}$,
its transform in $X$
is a fixed prime divisor of type $(3,0)^{sm}$ by Lemma \ref{pfizer}, therefore
$(c)\Rightarrow (c')$.
\end{proof}
\begin{prg}\label{reductions}
We note that $Y$ is not a product of surfaces, because a Fano $4$-fold 
 $S_1\times S_2$ with $\rho_{S_1}>1$ has a contraction onto $\pr^1\times S_2$, contradicting $(a')$.
Then we have $\delta_Y\leq 1$ by $(a')$ and  Theorems \ref{codim} and  \ref{delta2}.
\end{prg}
\begin{prg}\label{previous}
 Every fixed prime divisor $E$ of type $(3,2)$ in $Y$ meets $C$.
 
  Indeed if $E\cap C=\emptyset$, then $E$ is disjoint from the images of all exceptional curves in $\w{X}$, by Lemma \ref{vecchio}, thus
  $E$ is contained in the open where the birational map $X\dasharrow Y$ is an isomorphism. Its transform $E_X\subset X$ is a fixed prime divisor of type $(3,2)$ by Lemma \ref{pfizer}, and it is
  disjoint from $D$, so that $\N(E_X,X)\subseteq D^{\perp}\subsetneq\N(X)$ (see Rem.~\ref{easy}), contradicting $(b)$.
\end{prg}
\begin{prg}\label{E_0}
  We introduce the following notation: if $C$ belongs to a family of lines which cover a divisor, we denote such divisor by $E_0$; by Lemma \ref{chitarre} it could be nef, or fixed of type $(3,2)$; in this last case $C_{E_0}\equiv C$.
\end{prg}
\begin{prg}\label{pippo}
  If $E\subset Y$ is a fixed prime  divisor  of type $(3,2)$, 
then $E\cap C\neq\emptyset$ by \ref{previous}, and applying \cite[Lemma 5.11]{blowup} we know that either $C\not\subset E$ and $E\cdot C>0$, or $E=E_0$ and $C_{E_0}\equiv C$.
\end{prg}  
\begin{prg}\label{VQR}
Let $E_1,\dotsc,E_r$ be the fixed prime divisors of $Y$ of type $(3,2)$ with
$E_i\cdot C>0$, so that by \ref{pippo} the fixed prime divisors of type $(3,2)$ of  $Y$ are either $r$ (namely $E_1,\dotsc,E_r$) or $r+1$  (namely $E_0,E_1,\dotsc,E_r$).

We show that $E_i\cdot C_{E_j}=0$ for every $i,j\in\{1,\dotsc,r\}$ with $i\neq j$. 

Take for simplicity $i=1$ and $j=2$, and let $\Gamma$ be a curve such that $\Gamma\equiv C_{E_2}$ and $C\cap\Gamma\neq\emptyset$. Then by Rem.~\ref{vecchissimo} $\Gamma$ is the image of an exceptional curve of $\w{X}$, hence by Lemma \ref{vecchio} $\Gamma$ cannot intersect any curve of anticanonical degree one (different from $C$ and from $\Gamma$ itself). Therefore $\Gamma\cap E_1=\emptyset$ and $E_1\cdot C_{E_2}=0$.

This implies that the classes $[E_1],\dotsc,[E_r]$ are linearly independent in $\Nu(Y)$, hence $r\leq\rho_Y$. In fact $r<\rho_Y$, otherwise given an ample divisor $A$, we can write $A\equiv\sum_{i=1}^{\rho_Y}\lambda_i E_i$, and $A\cdot C_{E_j}=-\lambda_j>0$ for every $j=1,\dotsc,\rho_Y$, which gives a contradiction.
\end{prg}
\begin{prg}\label{ancora}
Suppose that $E_0$ is a fixed prime divisor of type $(3,2)$. Then $E_0$ 
and $E_i$ are not adjacent, $E_0\cdot C_{E_{i}}>0$, and $E_0+E_i$ is movable, for every $i=1,\dotsc,r$.
This follows from Cor.~\ref{solfeggio}, because $E_{i}\cdot C_{E_0}=E_{i}\cdot C>0$.
\end{prg}
 \begin{prg}\label{fibr}
  Suppose that there exists
  a movable face $\tau$ of $\Eff(Y)$ such that, if $f\colon Y\dasharrow Z$ is an associated rational contraction of fiber type as in \ref{setup}, we have $Z\cong\pr^2$.
  
Then either $\rho_X\leq 10$, or
  $\tau$ is a facet and the general fiber of $f$ is $\pr^2$.
\end{prg}
\begin{proof}
   Let $\zeta\colon\w{Y}\dasharrow Y$ be a SQM such that  $\tilde{f}:=f\circ\zeta\colon \w{Y}\to\pr^2$ is regular and $K$-negative (see Lemma \ref{basic2}), and let $F\subset\w{Y}$ be a general fiber of $\tilde{f}$, so that $F$ is a smooth del Pezzo surface. Note that $F$ is contained in the open subset where $\zeta$ is an isomorphism, because the indeterminacy locus of $\zeta$ has dimension at most one, see Lemma \ref{basic1}.

If $\rho_F=1$, then $F\cong\pr^2$ and $\tau$ is a facet by \eqref{estate}.
Suppose that $\rho_F>1$. Then $F$ is covered by rational curves of anticanonical degree $2$, so that $\w{Y}$ and $Y$ have a covering family of rational curves of anticanonical degree $2$. By Lemma \ref{test2}$(b)$
 $C$ must be a component of a curve of this family in $Y$; moreover 
$C$ cannot meet any exceptional plane (see Rem.~\ref{vecchissimo}) and thus it is contained in the open subset where $\zeta^{-1}$ is an isomorphism and $f$ is regular, so that $f(C)$ is a point.

 Let us consider now the composition $X\dasharrow\pr^2$. It contracts $D$ to a point, and no other prime divisors, because $f$ is equidimensional (see \ref{setup}). Then Lemma \ref{27} and $(b)$ yield $\rho_X\leq 10$.
 \end{proof}
\begin{prg}\label{GGA}
  Suppose that there exists  a minimal movable face  $\tau$ of $\Eff(Y)$ with $\dim \tau\leq 2$.

  Then either  $\rho_X\leq 10$, or $\dim\tau=2$ and  $\tau$ does not contain classes of fixed prime divisors of type $(3,2)$.
\end{prg}
\begin{proof}
We apply
  Prop.~\ref{usseaux} to $Y$ and $\tau$; case $(vi)$ of the Proposition is excluded because $\tau$ is minimal, and
cases $(i)$ and $(ii)$ are excluded by $(a')$ and $(b')$ respectively. 

In case $(iii)$ of Prop.~\ref{usseaux} we get either $\dim\tau=1$, $\rho_Y\leq 9$, and $\rho_X\leq 10$, or $\dim\tau=2$ and  $\tau$ does not contain classes of fixed prime divisors of type $(3,2)$.

In case $(v)$  we apply \ref{fibr} and get $\rho_X\leq 10$, because $\tau$ is not a facet.

Finally in case $(iv)$ of Prop.~\ref{usseaux} we have $\rho_Y=10$ and
there is a contraction $f\colon Y\to\pr^1$ with general fiber $S\times\pr^1$, $S$ a del Pezzo surface with $\rho_S=9$. We show that this last case cannot happen.

 Since $S\times\pr^1$ is covered by rational curves of anticanonical degree $2$, as in \ref{fibr} we see that $C$ must be contained in a fiber of $f$.
  Let us consider the composition $f\circ\sigma\colon\w{X}\to\pr^1$. There is a SQM $\psi\colon \wi{X}\dasharrow\w{X}$ such that  $g:=\psi\circ f\circ\sigma\colon \wi{X}\to\pr^1$ is regular and $K$-negative (see Lemma \ref{basic2}), and $\wi{X}\cong X$ by Lemma \ref{regular}; let $F\subset X$ be the general fiber of $g$.
  We have $g(D)=\{p_0\}$ and $g^{-1}(p_0)$ has at least another irreducible component $D_2$, so that $F\cap (D\cup D_2)=\emptyset$ and $\N(F,X)\subseteq D^{\perp}\cap D_2^{\perp}$ (see Rem.~\ref{easy}). On the other hand, since $D$ is fixed, the classes $[D]$ and $[D_2]$ cannot be proportional in $\Nu(X)$, thus
  $D^{\perp}\neq D_2^{\perp}$ and $\dim\N(F,X)\leq\rho_X-2$, contradicting $\delta_X\leq 1$.
\end{proof}
 \emph{Thus
  we can assume that $\Eff(Y)$ is generated by classes of fixed prime divisors, and that if $\tau$ is a 2-dimensional movable face of $\Eff(Y)$, then $\tau$ does not contain classes of fixed prime divisors of type $(3,2)$.}
\begin{prg}\label{primavera}
If $Y$ contains a nef prime divisor $H$ covered by a family $V$ of lines, then  either $\rho_X\leq 10$, or $H=E_0$ and $[C]\equiv [V]$.
  
Indeed if $H\cap C=\emptyset$, then Prop.~\ref{test} together with $(a)$ and $(b)$ yields $\rho_X\leq 10$. If instead $H\cap C\neq\emptyset$, then $C$ is a member of the family $V$ by Lemma \ref{test2}$(a)$, so that $H=E_0$ (see \ref{E_0}).
\end{prg}
\emph{Thus
  we can assume that if $Y$ contains a nef prime divisor $H$ covered by a family $V$ of lines, then $H=E_0$ and $[C]\equiv [V]$.}
\begin{prg} We show that $Y$ has some fixed prime divisor not of type $(3,2)$.
  
  By \ref{GGA} $\Eff(Y)$ is generated by classes of fixed prime divisors, so there are at least $\rho_Y$ of them. On the other hand, by \ref{VQR} $Y$ has at most $r+1\leq\rho_Y$ fixed prime divisors of type $(3,2)$. If $Y$ has
  exactly $\rho_Y$ fixed prime divisors of type $(3,2)$, then $E_0$ is fixed of type $(3,2)$ (see \ref{E_0} and \ref{VQR}). By \ref{GGA} any $2$-dimensional face of  $\Eff(Y)$ containing $[E_0]$ is fixed, and yields a fixed prime divisor $B$ adjacent to $E_0$. Then $B\neq E_i$ for every $i=1,\dotsc,r$ by \ref{ancora},
 thus $B$  is not of type $(3,2)$.
 \end{prg}
 \begin{prg}\label{B0}
   Let $B\subset Y$ be a fixed prime divisor not of type $(3,2)$; by $(c')$ $B$ is of type $(3,1)^{sm}$ or  $(3,0)^Q$. Let  $L\subset B$ be an exceptional plane, and 
   let us consider the family of lines $V$ in $Y$ given by $B$ and $L$
    as in Prop.~\ref{importante}, so that $\Lo V$ is a divisor and $[V]\equiv C_B-C_L$.
    We show that
$[V]\in\{[C],[C_{E_1}],\dotsc,[C_{E_r}]\}$.

Indeed if $\Lo V$ is not nef, then it is a fixed prime divisor of type $(3,2)$ by Lemma \ref{chitarre}, so that  by \ref{VQR} $\Lo V\in\{E_0,E_1\dotsc,E_r\}$, and
$[V]\in\{[C],[C_{E_1}],\dotsc,[C_{E_r}]\}$. If instead $\Lo V$ is nef, then 
by \ref{primavera} $\Lo V=E_0$ and $[V]=[C]$.
\end{prg}
\begin{prg}\label{dip}
Let us vary the exceptional plane $L$ in $B$ and consider all the families of lines that we obtain; 
  we denote by $\eta_1,\dotsc,\eta_s\in\N(Y)$ their distinct numerical classes. We show that  $s\geq \rho_Y-3$.

  Indeed we have $\delta_Y\leq 1$ by \ref{reductions}, so $\dim\N(B,Y)\geq \rho_Y-1$.
  By Rem.~\ref{generators} $B$ contains at least $\rho_Y-3$ exceptional planes  $L_1,\dotsc,L_{\rho_Y-3}\subset B$ such that the classes $[C_{L_1}],\dotsc,[C_{L_{\rho_Y-3}}]$ are linearly independent.  In particular the classes
$[C_B-C_{L_i}]$  of the $\rho_Y-3$ corresponding families are all distinct.
\end{prg}
\begin{prg}\label{B}
  By \ref{B0} we have $\eta_j \in\{[C],[C_{E_1}],\dotsc,[C_{E_r}]\}$ for every $j=1,\dotsc,s$, and $s\geq\rho_Y-3$ by \ref{dip}. Therefore
$r\geq s-1\geq \rho_Y-4=\rho_X-5\geq 3$; moreover up to renumbering we can assume that  $\eta_{1},\eta_{2}\in\{[C_{E_1}],\dotsc,[C_{E_r}]\}$.
  \end{prg}
 \begin{prg}
By Lemma \ref{pfizer} for every $i=1,\dotsc,r$ the transform 
$\w{E}_i$ of $E_i$ in $X$ is a fixed prime divisor of type $(3,2)$ adjacent to $D$. Moreover, since $E_i\not\supset C$ (see \ref{pippo}), in $\w{X}$ the divisor $\sigma^{-1}(E_i)$ is disjoint from the general fiber of $\sigma$, and hence $\w{E}_i\cdot C_D=\sigma^{-1}(E_i)\cdot C_{\w{D}}=0$.

We conclude that $D$ has at least $\rho_X-5$ adjacent fixed prime divisors of type $(3,2)$ and having intersection zero with $C_D$.
 \end{prg}
In this first part of the proof we have shown the following.
\begin{proposition}\label{sabrina}
  Let $X$ be a Fano $4$-fold with $\rho_X\geq 8$ and $\delta_X\leq 1$, and $D\subset X$
  a fixed prime divisor of type $(3,1)^{sm}$. Then one of the following holds:
  \begin{enumerate}[$(i)$]
  \item $X$ has a rational contraction onto a $3$-fold;
  \item $X$ has a fixed prime divisor $E_X$ of type $(3,2)$ such that $\N(E_X,X)\subsetneq\N(X)$;
\item $X$ has a fixed prime divisor of type $(3,0)^{sm}$;
  \item $\rho_X\leq 10$;
   \item there are at least $\rho_X-5$ fixed prime divisors of type $(3,2)$ adjacent to $D$ and having intersection zero with $C_D$.
     \end{enumerate}
\end{proposition}
We assume from now on that $\rho_X\geq 9$, so that $\rho_Y\geq 8$.
\begin{prg}\label{B2}
By \ref{B} we have
$\eta_{1}=[C_{E_{i_1}}]$ and $\eta_{2}=[C_{E_{i_2}}]$ for some $i_1,i_2\in\{1,\dotsc,r\}$. Then 
$E_{i_1}$ is the locus of the family of lines with class $\eta_{1}$, and $E_{i_1}\cdot \eta_{2}=E_{i_1}\cdot C_{E_{i_2}}=0$ by
\ref{VQR}.  By Lemma \ref{atletica} this implies that $B$ is of type $(3,1)^{sm}$.
\end{prg}
We conclude that \emph{every fixed prime divisor of $Y$ is of type $(3,1)^{sm}$ or $(3,2)$.}
\begin{prg}\label{lite}
  Let $i\in\{1,\dotsc,r\}$. We show that $E_i$ and $B$ are adjacent if and only if $B\cdot C_{E_i}=0$, if and only if $E_i\cdot C_B=0$.

  Indeed the first equivalence follows from Lemmas \ref{2faces}$(a)$ and \ref{nigra}$(a)$; similarly
if $E_i\cdot C_B=0$ then  $E_i$ and $B$ are adjacent. Conversely, if the two divisors are adjacent,  let $i_1,i_2$ be as in \ref{B2}; we can assume that $i_1\neq i$, so that $E_i\cdot C_{E_{i_1}}=0$ by \ref{VQR}. There exists an exceptional plane $L_1\subset B$ such that $C_B\equiv C_{L_1}+\eta_1\equiv C_{L_1}+C_{E_{i_1}}$, hence $E_i\cdot C_B=E_i\cdot C_{L_1}$. Now if 
  $E_i\cdot C_B>0$, then  $E_i\cdot C_{L_1}>0$, but Lemma \ref{aikido} yields $E_i\cap L_1=\emptyset$, a contradiction. Thus $E_i\cdot C_B=0$.
\end{prg}  
\begin{prg}\label{notturna}
  Suppose that $E_0$ is a fixed prime divisor of type $(3,2)$. Then $E_0\cdot C_B>0$.

  Indeed consider again $i_1$ and $L_1$ as in \ref{B2} and \ref{lite}. We  have $E_0\cdot C_{E_{i_1}}>0$ by \ref{ancora}, and $E_0\cdot C_{L_1}\geq 0$ (see Rem.~\ref{ikea}), so that $E_0\cdot C_B=E_0\cdot C_{L_1}+E_0\cdot C_{E_{i_1}}>0$.  
\end{prg} 
\begin{prg}\label{torino}
Either $\rho_X\leq 11$, or $B$ is adjacent to some divisor among $E_1,\dotsc,E_r$.

Indeed, let us  apply Prop.~\ref{sabrina} to $Y$ and $B$. 
By $(a')$, $(b')$, and $(c')$, either $\rho_Y\leq 10$ and $\rho_X\leq 11$, or 
$Y$ has a fixed prime divisor of type $(3,2)$ adjacent to $B$ and having intersection zero with $C_B$. By \ref{notturna} this divisor must be among $E_1,\dotsc,E_r$.
\end{prg} 
\emph{Thus we can assume that every fixed prime divisor $B\subset Y$ of type $(3,1)^{sm}$ is adjacent to some divisor among $E_1,\dotsc,E_r$.}
\begin{prg}\label{colazione}
Let $h\in\{1,\dotsc,r\}$ be such that $E_h$ and $B$ are adjacent. Then
$B\cdot C_{E_h}=E_h\cdot C_B=0$ by \ref{lite}, and
$E_h\cap B\neq\emptyset$
by $(b')$  (otherwise $\N(E_h,Y)\subseteq B^{\perp}\subsetneq\N(Y)$, see Rem.~\ref{easy}), so that we can apply Lemma \ref{zoom}. We deduce that
 there exists an exceptional plane  $L_0\subset B$  such that $C_B\equiv C_{E_h}+C_{L_0}$ and $E_h$ is the locus of the family of lines given by $B$ and $L_0$, so that
$[C_{E_h}] =\eta_{j_0}$ for some $j_0\in\{1,\dotsc,s\}$ (see \ref{dip}).
\end{prg}  
\begin{prg}\label{motorino}
We have $\eta_j\in\{[C_{E_1}],\dotsc,[C_{E_r}]\}$ for every $j=1,\dotsc,s$ (see \ref{dip}).
  
Indeed by \ref{B} it is enough to show that $\eta_j\neq[C]$ for  every $j=1,\dotsc,s$.
Let $E_h$ and $\eta_{j_0}$ be as in \ref{colazione}. By 
   Prop.~\ref{yoga}
  we have $E_{h}\cdot \eta_j=0$ for every $j\in\{1,\dotsc,s\}$, $j\neq j_0$, 
moreover $E_h\cdot\eta_{j_0}=-1$.
 On the other hand
   $E_{h}\cdot C>0$ (see \ref{VQR}), 
   so that $\eta_j\neq [C]$ for every $j=1,\dotsc,s$.
\end{prg}
 \emph{We conclude that  for every exceptional plane $L\subset B$ there exists some $i\in\{1,\dotsc,r\}$ such that $C_B\equiv C_L+C_{E_i}$.}
\begin{prg}\label{fibr2}
  Suppose that there exists  a minimal movable face  $\tau$ of $\Eff(Y)$ with
  $\dim\tau\leq 3$. 

Then either $\rho_X\leq 11$, or $\dim\tau=3$ and $\tau$  does not contain classes of fixed prime divisor of type $(3,2)$.
\end{prg}
\begin{proof}
 We apply Prop.~\ref{usseaux} to $Y$ and $\tau$; 
 case $(vi)$ of the Proposition is excluded because $\tau$ is minimal, and
cases $(i)$ and $(ii)$ are excluded by $(a')$ and $(b')$ respectively. 

 In case $(iii)$ of Prop.~\ref{usseaux},  either $\dim\tau=2$, $\rho_Y\leq 10$, and $\rho_X\leq 11$, or $\dim\tau= 3$ and  $\tau$ does not contain classes of fixed prime divisors of type $(3,2)$. In case $(iv)$ we have $\rho_Y=10$ and $\rho_X=11$. 
In case $(v)$ we apply \ref{fibr} and get $\rho_X\leq 10$, because $\tau$ is not a facet.
\end{proof}
 \emph{Thus
  we can assume that every $2$-dimensional face of $\Eff(Y)$ is fixed, and that if $\tau$ is a 3-dimensional movable face of $\Eff(Y)$, then $\tau$ does not contain classes of fixed prime divisors of type $(3,2)$.}
\begin{prg}\label{milano}
  Let $L\subset B$ be an exceptional plane. There exists  a fixed prime divisor  $G\subset Y$ of type $(3,1)^{sm}$, adjacent to $B$, and such that $L\subset G$.
\end{prg}
\begin{proof}
 Consider 
 $Y\dasharrow\w{Y}\stackrel{\sigma_Z}{\to} Z$ the contraction associated to $B$ (see \S\ref{contraction}), so that $Z$ is a smooth Fano $4$-fold with $\rho_Z\geq 7$, and let $C_Z\subset Z$ be the center of the blow-up $\sigma_Z$. Let $\ell\subset \w{Y}$ be the exceptional curve corresponding to $L$, and $\Gamma\subset Z$ its image.

 By \ref{fibr2} and Lemma \ref{taberna} the cone $\Eff(Z)$ is generated by  classes of fixed prime divisors, thus there exists a fixed prime divisor $G_0\subset Z$ such that $G_0\cdot\Gamma>0$. By Lemma \ref{pfizer}, the transform $G\subset Y$
 of $G_0$ is a fixed prime divisor adjacent to $B$ and having the same type as $G_0$.

If $G_0\cap C_Z=\emptyset$, then $G_0$ cannot be of type $(3,2)$, otherwise as in \ref{previous} we get $G\cap B=\emptyset$, contradicting $(b')$.
Then  $G_0$ and $G$ are of type $(3,1)^{sm}$ by \ref{B2}. Moreover $\sigma_Z^{-1}(G_0)\cdot\ell=G_0\cdot\Gamma>0$, thus we conclude that $G\cdot C_L<0$ and
$L\subset G$.

\smallskip
 
Suppose now that $G_0\cap C_Z\neq\emptyset$. We show that there is another fixed prime divisor $P_0$ of $Z$ such that $P_0\cdot\Gamma>0$ and $P_0\cap C_Z=\emptyset$.

By \cite[Lemma 5.11(1)]{blowup}  $G_0$ must be a fixed prime divisor of type $(3,2)$; let $R$ be the associated extremal ray $R$ of type $(3,2)$ of $\NE(Z)$. Since $G_0\cdot\Gamma>0$, $[\Gamma]\not\in R$ and $\Gamma$ must intersect some curve of anticanonical degree one with class in $R$, and by Lemma \ref{vecchio} we conclude that $[C_Z]\in R$, $-K_Z\cdot C_Z=1$, and 
 $\Gamma\not\subset G_0$.

 Let us consider the contraction $f\colon Z\to W$ such that $\NE(f)=R$, so that $G_0=\Exc(f)$, and set $\Gamma':=f(\Gamma)$.

 There is a one-dimensional face $\eta$ of $\Eff(W)$ such that $\eta\cdot\Gamma'>0$. If $\eta$ is movable, then by
 Lemma \ref{taberna} $\eta=f_*\eta_Z$ where $\eta_Z$ is a $2$-dimensional movable face of $\Eff(Z)$ containing $[G_0]$. Again by Lemma \ref{taberna} we have $\eta_Z=(\sigma_Z)_*\eta_Y$ where $\eta_Y$ is a $3$-dimensional movable face of $\Eff(Y)$ containing $[G]$, contradicting \ref{fibr2}.

Thus $\eta$ is fixed, and there is a fixed prime divisor $P\subset W$ such that $P\cdot \Gamma'>0$, so that $f^*P\cdot\Gamma>0$ in $Z$.
 
If $P_0\subset Z$ is the transform of $P$, we have $f^*P=P_0+mG_0$ with $m=P_0\cdot C_Z$. On the other hand $P_0$ is fixed and adjacent to $G_0$ by Lemma \ref{taberna}$(c)$, hence $m=0$ by Lemma \ref{nigra}$(a)$, and $P_0\cdot\Gamma>0$. Finally $C_Z\not\subset P_0$ by \cite[Lemma 5.11(2)]{blowup}, thus $P_0\cap C_Z=\emptyset$.
\end{proof}
\begin{prg}\label{minimal}
For every fixed prime divisor $B\subset Y$ of type $(3,1)^{sm}$ let $n_B$ be the number of divisors among $E_1,\dotsc,E_r$ which are adjacent to $B$, so that $n_B>0$ by \ref{torino}. 
Let us choose $B_1$ with
$n_{B_1}$ minimal, namely such that $n_{B_1}\leq n_B$ for every fixed prime divisor $B\subset Y$ of type $(3,1)^{sm}$.
  \end{prg}
\begin{prg}
There exists $i\in\{1,\dotsc,r\}$ such that $B_1+E_i$ is movable and non-big.
\end{prg}
\begin{proof}
 By \ref{torino} there exists some $h\in\{1,\dotsc,r\}$ such that  $E_h$ is adjacent to $B_1$, hence $E_h\cdot C_{B_1}=B_1\cdot C_{E_h}=0$ by \ref{lite}. 
Then by \ref{colazione} there exists
   an exceptional plane $L_0\subset B_1$
  such that  $C_{B_1}\equiv C_{L_0}+C_{E_h}$.  

 We apply \ref{milano} to $B_1$ and $L_0$, and let $G\subset Y$ be a fixed prime divisor of type $(3,1)^{sm}$ adjacent to $B_1$ and such that $L_0\subset G$. We have $G\cdot C_{B_1}=B_1\cdot C_{G}=0$ by Lemma \ref{nigra}$(b)$.

By \ref{motorino}
 there exists $i\in\{1,\dotsc,r\}$
 such that $C_{G}\equiv C_{L_0}+C_{E_{i}}$; we have $i\neq h$, because $C_{B_1}\not\equiv C_{G}$. Recall that $E_{i}\cdot C_{E_{h}}=E_{h}\cdot C_{E_{i}}=0$ by \ref{VQR}.
 We get:
 \stepcounter{thm}
 \begin{subequations}
 \begin{align}
   C_{B_1}+C_{E_{i}}&\equiv  C_{L_0}+C_{E_{h}}+C_{E_{i}}\equiv C_{G}+C_{E_{h}}\label{somma}\\
B_1\cdot C_{E_{i}}&=B_1\cdot(C_{G}+C_{E_{h}}-C_{B_1})=1\label{primo}\\\notag
G\cdot C_{E_h}&=G\cdot (C_{B_1}+C_{E_{i}}-C_{G})=G\cdot C_{E_{i}}+1\geq 1\\
E_{i}\cdot C_{B_1}&=E_{i}\cdot(C_{G}+C_{E_{h}}-C_{E_i})=  E_{i}\cdot C_{G}+1\geq 1
 \label{secondo}\end{align}
 \end{subequations}
so that $E_{h}$ is not adjacent to $G$ and $E_{i}$ is not adjacent to $B_1$
by \ref{lite}.

 Let $j\in\{1,\dotsc,r\}\smallsetminus\{i,h\}$. Then $E_j\cdot C_{E_{i}}=E_{j}\cdot C_{E_{h}}=0$ by \ref{VQR}, so that
 $E_j\cdot C_{B_1}=E_j\cdot C_{G}$ by \eqref{somma}, therefore $E_j$ is adjacent to $B_1$ if and only if it is adjacent to $G$, again by \ref{lite}.

 We conclude that either $E_{i}$ is not adjacent to $G$ and
$n_{G}=n_{B_1}-1$, or  $E_{i}$ is adjacent to $G$ and
$n_{G}=n_{B_1}$ (see \ref{minimal}). By the minimality of $n_{B_1}$, we conclude that $E_{i}$ must be adjacent to $G$, so that $E_{i}\cdot C_{G}=0$ by \ref{lite}, and finally we get  $B_1\cdot C_{E_i}=E_{i}\cdot C_{B_1}=1$ by \eqref{primo} and \eqref{secondo}.

By Lemma \ref{2faces}$(c)$, this implies that $B_1+E_{i}$ is movable and non-big.
\end{proof}
\begin{prg}\label{pioggia}
  Assume for simplicity that $i=1$, and
  let $f\colon Y\dasharrow Z$ be the rational contraction of fiber type defined by $m_1(B_1+E_{1})$ for $m_1\in\mathbb{N}$, $m_1\gg 0$. By Lemma \ref{sushi}, $(a')$, and $(b')$, we have that $Z\cong\pr^2$ and $f$ is equidimensional.

Let 
   $\tau$ be the smallest face
  of $\Eff(Y)$ containing $f^*\Nef(\pr^2)$, so that $\tau$ is a movable face, 
 $f$ is a rational contraction associated to $\tau$ as in \ref{setup}, and $\tau\cap\Mov(Y)=f^*\Nef(\pr^2)$. 
By \ref{fibr} we can assume that $\tau$ is a facet 
 of $\Eff(Y)$ and that the general fiber of $f$ is $\pr^2$.

Let $\Gamma_1,\dotsc,\Gamma_m\subset\pr^2$ be the  irreducible curves 
 such that $f^*\Gamma_i$ is reducible; we can assume that $B_1+E_1=f^*\Gamma_1$.
By \cite[Lemma 5.2]{fibrations} we know that $f^*\Gamma_i$ has two irreducible components, both fixed divisors, at least one 
of which of type $(3,2)$. We also have $m=\rho_Y-2$
 by \cite[Cor.~2.16]{fibrations}.

Recall from \ref{VQR} that the fixed prime divisors of type $(3,2)$ of $Y$ are $E_1,\dotsc,E_r$, and possibly $E_0$. If $E_0$ is a component of $f^*\Gamma_i$ for some $i\in\{1,\dotsc,\rho_Y-2\}$, then for $k\neq i$ $f^*\Gamma_k$ must have as a component $E_j$ for
 some $j\in\{1,\dotsc,r\}$. Then $E_0$ and $E_j$ should be adjacent by \cite[Cor.~2.18]{fibrations}, contradicting \ref{ancora}.

Therefore up to renumbering we can assume that $E_i$ is a component of $f^*\Gamma_i$ for every $i=1,\dotsc,\rho_Y-2$ (in particular $r\geq\rho_Y-2$).

If $j\in\{1,\dotsc,r\}$, $j\neq i$, we have $E_j\cdot C_{E_i}=0$ by \ref{VQR}, so that $E_i$ and $E_j$ are adjacent by Lemma \ref{2faces}$(a)$;
since 
$f^*\Gamma_i$ is movable,
 the second component 
$B_i$ of $f^*\Gamma_i$
must be
 a fixed prime divisor of type $(3,1)^{sm}$.

Note that $\tau$ is generated by the classes of fixed prime divisors that do not dominate
 $\pr^2$ under $f$ (see \ref{setup}), so that:
$$\tau=\langle [E_{1}],\dotsc,[E_{\rho_Y-2}],[B_1],\dotsc,[B_{\rho_Y-2}]\rangle.$$
Moreover, by \cite[Cor.~2.18]{fibrations}, for every partition $\{1,\dotsc,\rho_Y-2\}=I\sqcup J$ the cone $\langle [E_i],[B_j]\rangle_{i\in I, j\in J}$ is a fixed face of $\tau$. In particular $B_i$ and $E_j$ are adjacent for every $i,j\in\{1,\dotsc,\rho_Y-2\}$, $i\neq j$, and $E_j\cdot C_{B_i}=B_i\cdot C_{E_j}=0$ by \ref{lite}.

We also have, for every $i=1,\dotsc,\rho_Y-2$, that  $\langle [E_i],[B_i]\rangle\cap\Mov(Y)=f^*\Nef(\pr^2)$, therefore
 $E_i\cdot C_{B_i}=B_i\cdot C_{E_i}=1$ by Lemma \ref{2faces}$(c)$. Hence for every $j\in\{1,\dotsc,\rho_Y-2\}$ we have $E_j\cdot(C_{E_i}+C_{B_i})=B_j\cdot(C_{E_i}+C_{B_i})=0$, so that $\tau$ lies on the hyperplane $(C_{E_i}+C_{B_i})^{\perp}$, and  finally
\stepcounter{thm}
\begin{equation}\label{tau}
\tau=(C_{E_i}+C_{B_i})^{\perp}\cap\Eff(Y).\end{equation}
\end{prg}
\begin{prg}\label{adjacent}
Every fixed prime divisor $P\subset Y$ of type $(3,1)^{sm}$
can be not adjacent to at most two divisors among $E_1,\dotsc,E_r$.

Indeed by \ref{pioggia} $B_1$ is adjacent to $E_2,\dotsc,E_{\rho_Y-2}$, so that $n_{B_1}\geq\rho_Y-3$ (see \ref{minimal}). By the minimality of $n_{B_1}$, we deduce that $n_P\geq\rho_Y-3$ for every fixed prime divisor $P\subset Y$ of type $(3,1)^{sm}$, namely $P$ is adjacent to at least $\rho_Y-3$ divisors among $E_1,\dotsc,E_r$. Since $r\leq\rho_Y-1$ (see \ref{VQR}), we conclude that $P$ can be not adjacent to at most two divisors among $E_1,\dotsc,E_r$.
\end{prg}  
\begin{prg}\label{eta}
We consider now the fixed face $\langle [B_1],\dotsc,[B_{\rho_Y-4}],[E_{\rho_Y-3}],[E_{\rho_Y-2}]\rangle$ of $\tau$, of dimension $\rho_Y-2$; there exists 
 a facet $\eta$ of $\Eff(Y)$ such that
$$\tau\cap\eta=\langle [B_1],\dotsc,[B_{\rho_Y-4}],[E_{\rho_Y-3}],[E_{\rho_Y-2}]\rangle.$$

Let $P\subset Y$ be a fixed prime divisor with $[P]\in\eta\smallsetminus 
\langle [B_1],\dotsc,[B_{\rho_Y-4}],[E_{\rho_Y-3}],[E_{\rho_Y-2}]\rangle$; in particular $[P]\not\in\tau$.

The possibilities for $P$ are $P=E_0$, $P=E_{\rho_Y-1}$ (and $r=\rho_Y-1$), or $P$ of type $(3,1)^{sm}$.
\end{prg}
\begin{prg}
Either $P=E_{\rho_Y-1}$ or $P$ is of type $(3,1)^{sm}$.

Indeed if $P=E_0$, then $\eta$ contains both $[E_0+E_{\rho_Y-3}]$ and $[E_0+E_{\rho_Y-2}]$ that are movable by \ref{ancora}. Let $\eta_0$ be the minimal face of $\eta$ containing $[2E_0+E_{\rho_Y-3}+E_{\rho_Y-2}]$, so that $\eta_0$ is movable
 and
$[E_0+E_{\rho_Y-3}],[E_0+E_{\rho_Y-2}]\in\eta_0\cap\Mov(Y)$. Let $g\colon Y\dasharrow W$ be a
rational contraction of fiber type associated to $\eta_0$ as in \ref{setup}, so that
$\rho_W=\dim(\eta_0\cap\Mov(Y))\geq 2$ and $[E_0+E_{\rho_Y-3}]\in g^*\Nu(W)$.
Then Lemma \ref{sushi}, $(a')$, and $(b')$ give a contradiction.
\end{prg}
\begin{prg}\label{zero}
Up to reordering $E_1,\dotsc,E_{\rho_Y-4}$, we have
 $P\cdot C_{E_i}=0$ for every $i=1,\dotsc,\rho_Y-6$.

This is clear by \ref{VQR} if
   $P=E_{\rho_Y-1}$. If $P$ is of type $(3,1)^{sm}$, then by \ref{adjacent} $P$
can be not adjacent to at most two divisors 
 among $E_1,\dotsc,E_{\rho_Y-4}$, so up to renumbering we can assume that $P$ is adjacent to $E_1,\dotsc,E_{\rho_Y-6}$, and $P\cdot C_{E_i}=0$ for every $i=1,\dotsc,\rho_Y-6$ by \ref{lite}.
\end{prg}
\begin{prg}
We have $[P+B_i]\in\eta\cap\Mov(Y)$ for every $i=1,\dotsc,\rho_Y-6$. 

Indeed let $i\in\{1,\dotsc,\rho_Y-6\}$. Clearly $[P+B_i]\in\eta$ because $[P],[B_i]\in\eta$. Moreover since $[P]\not\in\tau$ (see \ref{eta}), we must have 
$P\cdot (C_{E_i}+C_{B_i})>0$ by \eqref{tau}, and $P\cdot C_{E_i}=0$ by \ref{zero}, thus  $P\cdot C_{B_i}>0$.  We also have $B_i\cdot C_P>0$: this follows from \ref{lite} if $P=E_{\rho_Y-1}$, and from
Cor.~\ref{solfeggio}
 if $P$ is $(3,1)^{sm}$.
Finally
$[P+B_i]\in\Mov(Y)$ by Lemma \ref{2faces}$(b)$.
\end{prg}
\begin{prg}
  Let  $\eta_1$ be the minimal face of $\eta$ containing $[(\rho_Y-6)P+B_1+\cdots+B_{\rho_Y-6}]$, so that $\eta_1$ is movable,
contains $[P],[B_{1}],\dotsc,[B_{\rho_Y-6}]$, and
$[P+B_i]\in\eta_1\cap\Mov(Y)$ for every $i=1,\dotsc,\rho_Y-6$. Let $h\colon Y\dasharrow S$ be a
rational contraction of fiber type associated to $\eta_1$ as in \ref{setup}, so that
$\rho_S=\dim(\eta_1\cap\Mov(Y))\geq \rho_Y-6\geq 2$ and $\eta_1$ is the smallest face of $\Eff(Y)$ containing $h^*\Eff(S)$.

Using Prop.~\ref{usseaux}, $(a')$, and $(b')$, we see that we must be in case $(vi)$ of the Proposition, namely that $\dim S=2$ and $h$ is quasi-elementary.  Then  $h^*\Eff(S)$ is a face of $\Eff(Y)$ by \cite[Prop.~2.22]{eff}, thus
 $\eta_1=h^*\Eff(S)\cong\Eff(S)$, and  $\rho_S=\dim(\eta_1\cap\Mov(Y))=\dim\eta_1$. If $\rho_S\leq 3$ we get $\rho_Y\leq 9$ and $\rho_X\leq 10$, so we can assume that $\rho_S\geq 4$, and  $S$ is a smooth del Pezzo surface by Lemma \ref{ufficio}.
\end{prg}
\begin{prg}
Recall that $[B_1],\dotsc,[B_{\rho_Y-4}]\in\eta$ and $[B_1],\dotsc,[B_{\rho_Y-6}]\in\eta_1$. Up to exchanging $B_{\rho_Y-5}$ and $B_{\rho_Y-6}$ we can assume that
for some  $t\in\{\rho_Y-6,\rho_Y-5,\rho_Y-4\}$ we have
$[B_1],\dotsc,[B_{t}]\in\eta_1$ and that $[B_i]\not\in\eta_1$ for $i=t+1,\dotsc,\rho_Y-4$ (if $t<\rho_Y-4$). Note that $\langle[B_1],\dotsc,[B_{t}]\rangle$ is a fixed face of $\eta_1$ (see \ref{eta}), hence $\rho_S=\dim\eta_1\geq t+1$.
\end{prg}
\begin{prg}
  For every $i=1,\dotsc,t$ we have $B_i=h^*C_i$ where $C_i\subset S$ is an irreducible curve. Since $h^*(\langle[C_1],\dotsc,[C_{t}]\rangle)=\langle[B_1],\dotsc,[B_{t}]\rangle$ is a fixed face of $\Eff(Y)$, $\langle[C_1],\dotsc,[C_{t}]\rangle$ must be a fixed face of $\Eff(S)$,
  so that $C_1,\dotsc,C_t$ are pairwise disjoint $(-1)$-curves.

We show that there exists a $(-1)$-curve $C'\subset S$ different from $C_{t-1}$ and $C_{t}$, and disjoint from $C_1,\dotsc,C_{t-2}$. In particular $\langle[C'],[C_1],\dotsc,[C_{t-2}]\rangle$ is a fixed face of $\Eff(S)$.

  Indeed let $\alpha\colon S\to S'$ be the contraction of the $(-1)$-curves $C_1,\dotsc,C_{t-2}$, so that $S'$ is a smooth del Pezzo surface with $\rho_{S'}=\rho_S-(t-2)\geq 3$. Then $\alpha(C_{t-1})$ and $\alpha(C_{t})$ are  $(-1)$-curves in $S'$, and there is a $(-1)$-curve $C''\subset S'$ different from $\alpha(C_{t-1})$ and $\alpha(C_{t})$. Then $C'$ is the transform of $C''$ in $S$.
\end{prg}
\begin{prg}
Let us consider $P':=h^*C'$, so that $\langle[P'],[B_1],\dotsc,[B_{t-2}]\rangle=h^*(\langle[C'],[C_1],
\dotsc,[C_{t-2}]\rangle)$ is a fixed face of $\Eff(Y)$. Then
$P'$ is a fixed prime divisor in $Y$, with class in $\eta$, distinct from  $B_{t-1}$ and $B_{t}$, and adjacent to $B_1,\dotsc,B_{t-2}$.

If $P'$ is of type $(3,2)$, we apply Lemma \ref{sushi} and get a contradiction with $(a')$ and $(b')$.
Therefore $P'$ is of type $(3,1)^{sm}$, and $P'\cdot C_{B_i}=0$ for $i=1,\dotsc,t-2$ by Lemma \ref{nigra}$(b)$. Moreover $P'$ is different from $B_1,\dotsc,B_t$ by construction, from $B_{t+1},\dotsc,B_{\rho_Y-4}$ (if $t<\rho_Y-4$) because $[P']\in\eta_1$, and from $E_{\rho_Y-3},E_{\rho_Y-2}$ because they are of type $(3,2)$. Thus $[P']\not\in\langle [B_1],\dotsc,[B_{\rho_Y-4}],[E_{\rho_Y-3}],[E_{\rho_Y-2}]\rangle=\tau\cap\eta$, and finally $[P']\not\in\tau$.

 By \eqref{tau}  we have $0<P'\cdot(C_{E_i}+C_{B_i})=P'\cdot C_{E_i}$ for every $i=1,\dotsc,t-2$, so that by \ref{lite} $P'$ is not adjacent to $t-2$ divisors among $E_1,\dotsc,E_r$. Therefore \ref{adjacent} yields $t-2\leq 2$, $\rho_Y-6\leq t\leq 4$, $\rho_Y\leq 10$, and $\rho_X\leq 11$. This concludes the proof of Th.~\ref{finalmente}.\qedhere
\end{prg}
\end{proof}
\section{The case $(3,0)^{Q}$}\label{Q}
\noindent In this section we conclude the proof of Th.~\ref{main}, by considering the case of fixed prime divisors of type $(3,0)^Q$; in fact we prove the following more refined version.
\begin{thm}\label{main2}
Let $X$ be a smooth Fano $4$-fold having a small elementary contraction. Then $\rho_X\leq 12$, and if $\rho_X=12$, then $X$ has a rational contraction onto a $3$-fold.
\end{thm}
First we prove two preliminary lemmas which, given two adjacent fixed prime divisors of type $(3,0)^Q$, yield either $\rho_X\leq 12$, or the existence of some special exceptional planes. Then we deal with the final proof; let us give an outline.

Let $X$ be a smooth Fano $4$-fold with $\rho_X\geq 7$ and having a small elementary contraction. Using the previous results we can assume that
  every fixed prime divisor of $X$ is of type $(3,2)$ or $(3,0)^Q$ (with at least one $D$ of type $(3,0)^Q$), that every $2$-dimensional face of $\Eff(X)$ is fixed, and finally that we are in one of the following cases:
\begin{enumerate}[$(a)$]
\item every fixed prime divisor of $X$ is of type $(3,0)^Q$;
\item  $D$  has one adjacent fixed prime divisor of type $(3,2)$, and the other divisors adjacent to $D$ are of type $(3,0)^Q$.
\end{enumerate}
  These two cases will be treated in parallel; for simplicity let us consider now case $(a)$.
 We work with the fixed prime divisors $E_1,\dotsc,E_r$ which are adjacent to $D$. We also consider the classes $\gamma_1,\dotsc,\gamma_t\in\N(X)$ of all lines in the exceptional planes contained in $D$. We show that for $h=1,\dotsc,r$ the hyperplane $E_h^{\perp}\subset\N(X)$ must contain $\rho_X-2$ classes among  $\gamma_1,\dotsc,\gamma_t$, and for $j\neq h$, $E_j^{\perp}$ must contain a different subset of $\rho_X-2$ classes. Then we develop more conditions on the $\gamma_i$'s and the possible subsets of $\rho_X-2$ classes contained in the $E_h^{\perp}$'s, in order to get the statement by applying the preliminary lemmas.
\begin{lemma}\label{michi}
Let $X$ be a smooth Fano $4$-fold with $\rho_X\geq 8$ and $\delta_X\leq 1$, and let $D,E$ be two adjacent fixed prime divisors of type $(3,0)^Q$ in $X$. Let $L\subset E$ be an exceptional plane such that $L\cap D=\emptyset$. Then 
one of the following holds:
\begin{enumerate}[$(i)$]
\item
$\rho_X\leq 12$, and if $\rho_X=12$, then $X$ has a rational contraction onto a $3$-fold;
\item
there exists an exceptional plane ${M}\subset D$ such that $C_{M}+C_E\equiv C_D+C_L$ and $D\cdot C_{M}=-1$.
\end{enumerate}
\end{lemma}
\begin{proof}
 Let $V$ be 
 the family of lines in $X$ given by $E$ and $L$ as in Prop.~\ref{importante}, so that $C_E\equiv C_L+[V]$ and $B:=\Lo V$ is a divisor.
Since $L\cap D=\emptyset$ we have $D\cdot C_{L}=0$, moreover $D\cdot C_E=0$ by Lemma \ref{nigra}$(b)$, hence $D\cdot [V]=0$, and the general curve of the family $V$ is disjoint from $D$.

Let $X\stackrel{\ph}{\dasharrow}\w{X}\stackrel{\sigma}{\to} Y$ be the contraction associated to $D$ as in \S\ref{contraction}, so that $\Exc(\sigma)=\w{D}$ and $\sigma(\w{D})=p$. Then the birational map $X\dasharrow Y$ is an isomorphism on the general curve of the family $V$, and $V$ yields a family of lines $V_Y$ in $Y$, with locus a prime divisor $B_Y\subset Y$ which is the transform of $B\subset X$; 
we also consider the transform 
$\w{B}$ of $B$ in $\w{X}$.

Recall that for every exceptional plane $L'\subset D$ we have $D\cdot C_{L'}<0$, while $D\cdot [V]=0$, hence $[V]\neq[C_{L'}]$ and $L'\not\subset B$ by Lemma \ref{excplane}$(b)$. Therefore neither $B\cap D$ nor $\w{B}\cap\w{D}$, if non-empty, can be contained in the indeterminacy locus of $\ph$ or $\ph^{-1}$, and we conclude that $B\cap D\neq \emptyset$ if and only if $\w{B}\cap\w{D}\neq\emptyset$, if and only if $p\in B_Y$.

Suppose that $B\cap D=\emptyset$. If $B$ is not nef, then $B$ is fixed of type $(3,2)$ by Lemma \ref{chitarre},  and $\N(B,X)\subseteq D^{\perp}\subsetneq\N(X)$ (see Rem.~\ref{easy}), so we get 
  $(i)$ by Th.~\ref{natale}. If $B$ is nef, then $B_Y$ is nef too (because $B$ is contained in the open subset where the birational map $X\dasharrow Y$ is an isomorphism), and we get again $(i)$ by Prop.~\ref{test}.

Finally suppose that $B\cap D\neq\emptyset$, so that  $p\in B_Y$. Let $\Gamma$ be a general curve of the family $V_Y$, and $\Gamma_0$ a curve of the family containing $p$. By Rem.~\ref{vecchissimo} we have $\Gamma_0=\sigma(\ell)$, where $\ell\subset\w{X}$ is an exceptional curve such that $\w{D}\cdot \ell=1$; moreover $\w{\Gamma}\equiv\ell+C_{\w{D}}$, where $\w{\Gamma}\subset\w{X}$ is the transform of $\Gamma$. Now if $M\subset D\subset X$ is the exceptional plane corresponding to $\ell$, we have $D\cdot C_{M}=-\w{D}\cdot\ell=-1$, and $[V] \equiv\lambda C_M+C_{{D}}$ where $\lambda\in\R$; intersecting with $-K_X$ we get
$C_D\equiv [V]+C_{M}$. This yields $C_{M}+C_E\equiv C_D+C_L$, so we get $(ii)$.
\end{proof}
For simplicity we set $\rho:=\rho_X$ for the rest of the section.
\begin{lemma}\label{covid}
  Let $X$ be a smooth Fano $4$-fold with $\rho\geq 7$ and $\delta_X\leq 1$, and
  let $D,E$ be two adjacent fixed prime divisors of type $(3,0)^Q$. Suppose that 
one of the following holds:
\begin{enumerate}[$(a)$]
\item every fixed prime divisor of $X$ is of type $(3,0)^Q$;
\item there is a fixed prime divisor $F$, of type $(3,2)$, such that
$\langle [D],[E],[F]\rangle$ is a fixed face of $\Eff(X)$.
\end{enumerate}
Then one of the following holds:
\begin{enumerate}[$(i)$]
\item
 $\rho\leq 12$, and if $\rho=12$, then $X$ has a rational contraction onto a $3$-fold;
\item
 there exist $L_1,\dotsc,L_{\rho-2}\subset {E}$ exceptional planes, disjoint from $D$, such that\\
$[C_E],[C_{L_1}],\dotsc,[C_{L_{\rho-2}}]$ is a basis of $D^{\perp}$.
In case $(b)$, we can moreover assume that  $C_E\equiv C_F+C_{L_1}$ and 
$[C_E],[C_F],[C_{L_2}],\dotsc,[C_{L_{\rho-2}}]$ is a basis of $D^{\perp}$.
\end{enumerate}
\end{lemma}
\begin{proof}
We note that $X$ is not a product of surfaces (see Ex.~\ref{2022} or Rem.~\ref{delta}).
 Let $X\stackrel{\ph}{\dasharrow} \w{X}\stackrel{\sigma}{\to} Y$ be the contraction associated to $D$ as in \S\ref{contraction}, and $E_{\w{X}}\subset \w{X}$, $E_Y\subset Y$ the transforms of $E$; we still denote by $C_E\subset E_{\w{X}}$ the transform of $C_E$. By Lemma \ref{nigra}$(b)$ we know that $D\cdot C_E=E\cdot C_D=0$ and $D\cap E$ is either empty, or a union of exceptional planes; on the other hand $\w{D}\cap E_{\w{X}}=\emptyset$ in $\w{X}$.

If $\N(E_Y,Y)\subsetneq\N(Y)$, we apply Lemma \ref{pizza}. Either $X$ has a rational contraction onto a $3$-fold and we get $(i)$ by Th.~\ref{3fold}, or there is a fixed prime divisor $G\subset X$, of type $(3,2)$, adjacent to $D$, and such that $E\cdot C_G>0$. In particular we are in case $(b)$, and $G\neq F$ because $E\cdot C_F=0$ by Lemma \ref{nigra}$(a)$. Then $D$ is adjacent to two fixed prime divisors of type $(3,2)$, and 
 we get again $(i)$ by Lemma \ref{eccellenza}. 

Otherwise $\N(E_Y,Y)=\N(Y)$, and since $\N(E_Y,Y)=\sigma_*(\N(E_{\w{X}},\w{X}))$, we get $\dim\N(E_{\w{X}},\w{X})\geq\rho-1$. Since $\w{D}\cap  E_{\w{X}}=\emptyset$, we have $\N(E_{\w{X}},\w{X})=\w{D}^{\perp}$ (see Rem.~\ref{easy}); note that $[C_{\w{D}}]\not\in\N(E_{\w{X}},\w{X})$ because $\w{D}\cdot C_{\w{D}}=-1$.

There is a SQM $\w{X}\dasharrow \wi{X}$, obtained by considering the 
$E_{\w{X}}$-negative small extremal rays of $\NE(\w{X})$, whose indeterminacy locus is the union of the exceptional planes contained in $E_{\w{X}}$, and that has a factorization as in Lemma \ref{basic1}. Moreover the transform $E_{\wi{X}}\subset\wi{X}$ is the exceptional locus of an elementary divisorial contraction of type $(3,0)^Q$. Then, after \cite[Rem.~3.13]{eff}, as in Rem.~\ref{generators} we see that
 $E_{\w{X}}$ contain at least $\dim\N(E_{\w{X}},\w{X})-\dim\N(E_{\wi{X}},\wi{X})=\rho-2$ exceptional planes
 $L_1,\dotsc,L_{\rho-2}$ 
 such that $[C_{E}],[C_{L_1}],\dotsc,[C_{L_{\rho-2}}]$ is a basis of
$\N(E_{\w{X}},\w{X})$, so that $[C_{\w{D}}],[C_E],[C_{L_1}],\dotsc,[C_{L_{\rho-2}}]$ is a basis of $\N(\w{X})$.
 Note that each $L_i$ is disjoint from $\w{D}$ and also from every exceptional curve of $\w{X}$, by Lemma \ref{basic1}$(b)$, hence $L_i$ is contained in the open subset where $\ph\colon X\dasharrow \w{X}$ is an isomorphism,
and its transform in $X$ is still an exceptional plane $L_i\subset E$ such that $L_i\cap D=\emptyset$.
Moreover  $[C_E],[C_{L_1}],\dotsc,[C_{L_{\rho-2}}]$ are linearly independent in $\N(X)$, and they belong to $D^{\perp}$.

In case $(b)$, if $F\cap E=\emptyset$, then $\N(F,X)\subseteq E^{\perp}\subsetneq\N(X)$ (see Rem.~\ref{easy}) and we get again $(i)$ by Th.~\ref{natale}. Suppose that $F\cap E\neq\emptyset$; we have $E\cdot C_F=0$ by Lemma \ref{nigra}$(a)$, and by Lemma \ref{zoom} there exists an exceptional plane $L\subset E$ such that $C_E\equiv C_F+C_{L}$. Moreover we have $D\cdot C_F=0$ again by Lemma \ref{nigra}$(a)$, thus $0=D\cdot C_E= D\cdot (C_F+C_{L})=D\cdot C_L$. If $L\subset D$ we have $D\cdot C_L<0$ (see \ref{contraction}), therefore $L\cap D=\emptyset$ and $L$ is contained in the open subset where  $\ph\colon X\dasharrow\w{X}$ is an isomorphism.
Then we just choose $L_1=L$ in the previous construction, and get $(ii)$. 
\end{proof}
\begin{proof}[Proof of Theorem \ref{main2}]
\begin{prg}\label{rettore}
We assume that $\rho\geq 8$. 
Since $X$ has a small elementary contraction, it is not a product of surfaces, and we can assume that $\delta_X\leq 1$ by
 Theorems \ref{codim} and \ref{delta2}.
We can also assume that $X$ has no fixed prime divisor of type $(3,0)^{sm}$ or $(3,1)^{sm}$, by Theorems \ref{30}
and \ref{finalmente}, so that every fixed prime divisor of $X$ is of type $(3,0)^Q$ or $(3,2)$.

By Prop.~\ref{usseaux} we can assume that every $2$-dimensional face $\tau$ of $\Eff(X)$ is fixed, and that if $\tau$ is a $3$-dimensional non-fixed face of $\Eff(X)$,  then $\tau$ is generated by classes of fixed prime divisors of type $(3,2)$. In particular $\Eff(X)$ is generated by classes of fixed prime divisors.

By assumption $\NE(X)$ has a small extremal ray $R$, and there exists some fixed prime divisor having negative intersection with $R$; moreover it cannot be of type $(3,2)$ (see Rem.~\ref{stress}), thus
there exists at least one fixed prime divisor of type $(3,0)^Q$.
 By Lemma \ref{eccellenza} we can assume that we are in one of the following cases:
\begin{enumerate}[$(a)$]
\item every fixed prime divisor of $X$ is of type $(3,0)^Q$;
\item there is a fixed prime divisor $D$ of type $(3,0)^Q$ with one adjacent  fixed prime divisor 
$F$ of type $(3,2)$, and the other fixed prime divisors adjacent to $D$ are of type $(3,0)^Q$.
\end{enumerate}
\end{prg}
\begin{prg}\label{CCS}
Let $\tau$ be a  $3$-dimensional face of $\Eff(X)$
containing the class of some fixed prime divisor of type $(3,0)^Q$. Then $\tau$ is fixed by \ref{rettore},  so it is simplicial (see Lemma \ref{simplicial}), and
 by  Lemma \ref{eccellenza} we can assume that $\tau$ contains at most one class of a fixed prime divisor of type $(3,2)$.
\end{prg}
\begin{prg}\label{gelato}
  In case $(b)$, we can assume that every $4$-dimensional face $\eta$ of  $\Eff(X)$ containing $[D]$ and $[F]$ is fixed and simplicial.
\end{prg}
\begin{proof}
Let $\tau$ be a facet of $\eta$ containing $[D]$. Then by \ref{CCS}
$\tau$ is fixed and contains  at most one class of a fixed prime divisor of type $(3,2)$. Thus every facet $\tau'$ of $\eta$ such that $\dim(\tau\cap\tau')=2$ still contains the class of a fixed prime divisor of type $(3,0)^Q$. Proceeding in this way, we conclude that every facet of $\eta$ is fixed.

Now if $\eta$ is not fixed, it is a minimal movable face, and applying Prop.~\ref{usseaux} we get the statement. If instead $\eta$ is fixed, it is simplicial too (see Lemma \ref{simplicial}).
\end{proof}  
\begin{prg}\label{letto}
  Let us fix the following notation. In case $(a)$, we fix a $D$ of type $(3,0)^Q$, and
  $E_1,\dotsc,E_r$ are the fixed prime divisors adjacent to $D$. In case $(b)$, $E_1,\dotsc,E_r$ are the fixed prime divisors such that $\langle [D],[E_i],[F]\rangle$ is a fixed face of $\Eff(X)$. 

In both cases $E_1,\dotsc,E_r$ are all of type $(3,0)^Q$ (in case $(b)$ we use \ref{CCS}). Note that $E_i\cdot C_D=D\cdot C_{E_i}=0$ for all $i$ by Lemma \ref{nigra}$(b)$. Moreover, by \ref{rettore}, in case $(a)$ 
$\langle [D],[E_i]\rangle$ are all the $2$-dimensional faces of $\Eff(X)$ containing $[D]$, and in case $(b)$  $\langle [D],[E_i],[F]\rangle$ are all the $3$-dimensional faces of $\Eff(X)$ containing $\langle [D],[F]\rangle$.

Suppose that we are in  case $(b)$; similarly as above we have $D\cdot C_F=E_i\cdot C_F=0$ for all $i$. Moreover if $F$ is disjoint from one of the divisors $D,E_1,\dotsc,E_r$, then $\N(F,X)\subsetneq\N(X)$ (see Rem.~\ref{easy}) and we conclude by 
Th.~\ref{natale}. Thus we can assume that $F$ intersects all divisors $D,E_1,\dotsc,E_r$, and by Cor~\ref{pasquetta} we deduce that 
$F\cdot C_D=F\cdot C_{E_i}=1$ for all $i$. Summing up we have:
\stepcounter{thm}
\begin{equation}\label{divano}
D\cdot C_{E_i}=E_i\cdot C_D=D\cdot C_F=E_i\cdot C_F=0\ \text{ and }\ F\cdot C_D=F\cdot C_{E_i}=1\ \text{ for all }i.
\end{equation}
\end{prg}
\begin{prg}\label{treno}
We have $r\geq\rho$ in case $(a)$, and $r\geq\rho-1$ in case $(b)$.
\end{prg}
\begin{proof}
Note that $r\geq\rho-1$ and $r\geq\rho-2$ respectively for dimensional reasons, because  $\Eff(X)$ is a cone of dimension $\rho$, thus $[D]$ must 
be contained in at least 
 $\rho-1$ faces of dimension $2$, and similarly in case $(b)$.

In case $(a)$, assume that  $r=\rho-1$. Since every $3$-dimensional face of $\Eff(X)$ is fixed and simplicial by \ref{rettore}, we conclude that for every $i\neq j$ the cone $\langle[D],[E_i],[E_j]\rangle$ is a face of $\Eff(X)$, in particular $E_i$ and $E_j$ are adjacent, hence $E_i\cdot C_{E_j}=0$ by Lemma \ref{nigra}$(b)$. Now the classes of $D,E_1,\dotsc,E_{\rho-1}$ form a basis of $\Nu(X)$, so that we can write
$$-K_X\equiv a_0D+\sum_{i=1}^{\rho-1}a_iE_i$$
with $a_i\in \R$. Using \eqref{divano}, intersecting with $C_D$ we get $a_0=-2$, and intersecting with $C_{E_j}$ we get $a_j=-2$ for $j=1,\dotsc,\rho-1$, which gives a contradiction.

Case $(b)$ is similar. Suppose that $r=\rho-2$, and recall that by \ref{gelato} every $4$-dimensional face of $\Eff(X)$ containing $[D]$ and $[F]$ is fixed and simplicial. Then for every $i\neq j$ the cone $\langle[D],[F],[E_i],[E_j]\rangle$ is a face of $\Eff(X)$,  in particular $E_i$ and $E_j$ are adjacent, hence $E_i\cdot C_{E_j}=0$. Now the classes of $D,F,E_1,\dotsc,E_{\rho-2}$ form a basis of $\Nu(X)$, so that we can write
$$-K_X\equiv a_0D+\sum_{i=1}^{\rho-2}a_iE_i+bF$$
with $a_i,b\in \R$. Using \eqref{divano}, intersecting with $C_F$ we get $b=-1$, intersecting with 
$C_D$ we get $a_0=b-2=-3$ and intersecting with $C_{E_j}$ we get $a_j=b-2=-3$, $j=1,\dotsc,\rho-2$, again a contradiction. 
\end{proof}
\begin{prg}\label{pasqua}
 Let us consider all the exceptional planes contained in $D$, and let $\gamma_1,\dotsc,\gamma_t\in\N(X)$ be the distinct classes of all lines in these exceptional planes. We have $D\cdot\gamma_i<0$ for every $i$ (see \S\ref{contraction}); let us reorder the $\gamma_i$'s in such a way that:  
\stepcounter{thm}
\begin{equation}\label{compleanno}
D\cdot\gamma_i=-1\ \text{ for }i=1,\dotsc,s\ \text{ and }\ 
D\cdot\gamma_i\leq -2\ \text{ for }i=s+1,\dotsc,t.
\end{equation} 
In case $(b)$, we can assume that $C_D\equiv C_F+\gamma_1$. Indeed 
 we have $D\cap F\neq\emptyset$ and $D\cdot C_F=0$ (see \ref{letto}), and by Lemma \ref{zoom}  there exists $i\in\{1,\dotsc,s\}$ such that $C_D\equiv C_F+\gamma_i$; up to renumbering we can assume that $i=1$.
\end{prg}
\begin{prg}\label{freccia}
Let us consider $E_h$ with $h\in\{1,\dotsc,r\}$; we apply Lemma \ref{covid} to $D$ and $E_h$ (and $F$ in case $(b)$).
Either we get the statement, or $E_h^{\perp}$
must be generated  $[C_D]$ and by  $\rho-2$ classes among  $\gamma_1,\dotsc,\gamma_t$.

If we choose some $E_j$ different from $E_h$, then the classes $[E_h]$ and $[E_j]$ cannot be multiples, hence
the hyperplanes $E_h^{\perp}$ and $E_j^{\perp}$ are different. Thus  $E_j^{\perp}$ is generated  by $[C_D]$ and by a different choice of $\rho-2$ among the classes $\gamma_1,\dotsc,\gamma_t$.

In case $(a)$, using \ref{treno}, we get: $$\rho\leq r\leq\binom{t}{\rho-2},$$
which implies that $t\geq\rho$.

Similarly,
 in case $(b)$, since $C_D\equiv C_F+\gamma_1$, by Lemma \ref{covid} we have also that $E_h^{\perp}$
 must be generated  $[C_D]$, $[C_F]$, and by $\rho-3$ classes among $\gamma_2,\dotsc,\gamma_t$.
Then we get
 $$\rho-1\leq r\leq\binom{t-1}{\rho-3},$$
which implies again that $t\geq\rho$.
\end{prg}
\begin{prg}
Now we apply Lemma \ref{covid} to $E_1$ and $D$ (with the roles interchanged with respect to \ref{freccia}), and $F$ in case $(b)$. Either we get the statement, or there exist exceptional planes $L_1,\dotsc,L_{\rho-2}\subset E_1$
such that $L_j\cap D=\emptyset$ and the classes $[C_{L_1}],\dotsc,[C_{L_{\rho-2}}]$ are all distinct. 
\end{prg}
\begin{prg}\label{gomme}
 We apply Lemma \ref{michi} to $D$, $E_1$, and $L_j\subset E_1$, with $j\in\{1,\dotsc,\rho-2\}$. Either we get the statement, or there exists an exceptional plane $M_j\subset D$ such that $C_{M_j}\equiv C_D+C_{L_j}-C_{E_1}$ and $D\cdot C_{M_j}=-1$.

 In particular the classes $[C_{M_j}]$, for $j=1,\dotsc,\rho-2$, are all distinct, and must appear among $\gamma_1,\dotsc,\gamma_s$ (see \eqref{compleanno}); we deduce that $s\geq \rho-2$.
\end{prg}
\begin{prg}\label{AV}
Fix $i\in\{1,\dotsc,s\}$, so that $D\cdot\gamma_i=-1$. Let $P_i\subset D$ be an exceptional plane whose lines have class $\gamma_i$, and let $V_i$ be the family of lines in $X$ given by $D$ and $P_i$ as in Prop.~\ref{importante}, with locus a prime divisor $B_i\neq D$, so that  $C_D\equiv\gamma_i+[V_i]$, $B_i\cdot\gamma_i>0$, and $P_i\not\subset B_i$. 
Moreover by Lemma \ref{zoom2} $B_i$ cannot meet any exceptional plane $L\subset D$ such that $C_L\not\equiv\gamma_i$, so that $B_i\cdot\gamma_j=0$ for every $j\in\{1,\dotsc,t\}$, $j\neq i$. 
In particular $B_i\cap D$ is not a  union of exceptional planes, and $B_i\cdot C_D>0$ by Rem.~\ref{intersection}.
 Summing up we have:
 $$C_D\equiv\gamma_i+[V_i],\  B_i\cdot C_D>0,\ B_i\cdot\gamma_i>0,\ \text{and}
\
B_i\cdot\gamma_j=0\text{ for every }j\in\{1,\dotsc,t\},\ j\neq i.$$ 
\end{prg}
\begin{prg}\label{assemblee}
We deduce that both $\gamma_1,\dotsc,\gamma_s\in\N(X)$ and $[B_1],\dotsc,[B_s]\in\Nu(X)$ are linearly independent, and that $$\gamma_{s+1},\dotsc,\gamma_t\in B_1^{\perp}\cap\cdots\cap B_s^{\perp}.$$
Moreover note that $(-K_X+D)\cdot\gamma_i=0$ for every $i=1,\dotsc,s$, and since $-K_X+D\not\equiv 0$, we have $s\leq\rho-1$,   namely $s\in\{\rho-2,\rho-1\}$ (see \ref{gomme}).
\end{prg}
\begin{prg}\label{sum}
We will need the following estimation. For every $i=1,\dotsc,s$ we have $B_i\cdot [V_i]\geq -1$ (see Lemma \ref{chitarre}) and $0<B_i\cdot\gamma_i=B_i\cdot C_D-B_i\cdot [V_i]\leq B_i\cdot C_D+1$, thus
$$\frac{B_i\cdot C_D}{B_i\cdot\gamma_i}\geq \frac{B_i\cdot C_D}{B_i\cdot
C_D+1}\geq\frac{1}{2}\quad\text{ and }\quad
M:=\sum_{i=1}^s\frac{B_i\cdot C_D}{B_i\cdot\gamma_i}\geq\frac{1}{2}s.$$
\end{prg}
\begin{prg}
We show that
$s=\rho-2$.

Otherwise, we have $s=\rho-1$ by \ref{assemblee}.
Note that $(-K_X+D)^{\perp}$ is generated by $\gamma_1,\dotsc,\gamma_{\rho-1}$ (see \ref{assemblee}), and $(-K_X+D)\cdot C_D=1$, so that $\gamma_1,\dotsc,\gamma_{\rho-1}, [C_D]$ is a basis of $\N(X)$.
Recall that $t\geq \rho$ by \ref{freccia}, and write
$$\gamma_{\rho}=\sum_{i=1}^{\rho-1}a_i\gamma_i+a_0[C_D]$$
with $a_i\in\R$. By intersecting with $B_j$ for $j\in\{1,\dotsc,\rho-1\}$
we get  (see \ref{AV})
$$a_j=-a_0\frac{B_j\cdot C_D}{B_j\cdot\gamma_j}\quad\text{and}\quad 
\gamma_{\rho}=a_0\left(-\sum_{i=1}^{\rho-1}\frac{B_i\cdot C_D}{B_i\cdot\gamma_i}
\gamma_i+[C_D]\right).$$
Intersecting with $-K_X$ we get $1=a_0(-M+2)$, where  as in \ref{sum} we have $M\geq s/2=(\rho-1)/2\geq 7/2$. Hence 
$$\gamma_{\rho}=-\frac{1}{M-2}\left(-\sum_{i=1}^{\rho-1}\frac{B_i\cdot C_D}{B_i\cdot\gamma_i}
\gamma_i+[C_D]\right),$$
and finally intersecting with $D$ we get (see \eqref{compleanno}):
$$-2\geq D\cdot \gamma_{\rho}=-\frac{M-1}{M-2},$$
 which yields $M\leq 3$, a contradiction. Therefore $s=\rho-2$.
\end{prg}
\begin{prg}
Let us consider the plane 
$$\pi:=B_1^{\perp}\cap\cdots\cap B_{\rho-2}^{\perp}\subset\N(X).$$
We have $\gamma_{\rho-1},\dotsc,\gamma_t\in\pi$ (see \ref{assemblee}), and since $-K_X\cdot\gamma_i=1$ for every $i$, these classes cannot be proportional.
Recall that $t\geq\rho$ by \ref{freccia}, so there are at least two such classes.

Let us consider the $2$-dimensional cone $\langle \gamma_{\rho-1},\dotsc,\gamma_t\rangle\subset\pi$; up to renumbering we can assume that $\gamma_{\rho-1}$ and $\gamma_{\rho}$ generate the cone, that $D\cdot \gamma_{\rho-1}\leq D\cdot\gamma_{\rho}$, 
and that if $t>\rho$ the remaining classes $\gamma_{\rho+1},\dotsc,\gamma_t$ belong to the interior of the cone.
\end{prg}
\begin{prg}\label{-3}
Suppose that
$D\cdot \gamma_{\rho-1}=D\cdot\gamma_{\rho}$, and set $m:=-D\cdot \gamma_{\rho-1}\geq 2$ (see \eqref{compleanno}). Then $(-mK_X+D)\cdot\gamma_{\rho-1}=(-mK_X+D)\cdot\gamma_{\rho}=0$,
hence $(-mK_X+D)^{\perp}\supset\pi$
and
$$-mK_X+D\equiv\sum_{i=1}^{\rho-2}\lambda_iB_i$$
with $\lambda_i\in\R$. Intersecting with $\gamma_j$, $j\in\{1,\dotsc,\rho-2\}$, we get $m-1=\lambda_jB_j\cdot\gamma_j$ and 
$$-mK_X+D\equiv\bigl(m-1\bigr)\sum_{i=1}^{\rho-2}\frac{1}{B_i\cdot\gamma_i}B_i.$$
Then intersecting with $C_D$ we get $2m-1=(m-1)M$ where $M\geq s/2=(\rho-2)/2$ by \ref{sum}. Since $m\geq 2$, we get $M\leq 3$ and $\rho= 8$.
\end{prg}
\emph{We assume from now on that $D\cdot \gamma_{\rho-1}< D\cdot\gamma_{\rho}\leq -2$, so that $D\cdot\gamma_{\rho-1}\leq -3$.}
\begin{prg}
Let $h\in\{1,\dotsc,r\}$. Since $E_h$ is adjacent to $D$,
$E_h\cap D$ is either empty or a disjoint union of exceptional planes, by Lemma \ref{nigra}$(b)$. Thus for every  $i\in\{1,\dotsc,t\}$, if $P_i\subset D$
is an exceptional plane whose lines have class $\gamma_i$, we have either $P_i\cap E_h=\emptyset$ and
$E_h\cdot\gamma_i= 0$,
or $P_i\subset E_h$ and
$E_h\cdot\gamma_i< 0$ (see \S\ref{contraction}); in any case
 we have:
  $$E_h\cdot\gamma_i\leq 0.$$ 
\end{prg}
\begin{prg}
We show that  $E_h^{\perp}$ does not contain the plane $\pi$.
Suppose otherwise: then
there exist $\lambda_i\in\R$ such that $E_h\equiv\sum_{i=1}^{\rho-2}\lambda_iB_i$, and for every $i\in\{1,\dotsc,\rho-2\}$ we get
$$0\geq E_h\cdot\gamma_i=\lambda_iB_i\cdot\gamma_i$$
thus $\lambda_i\leq 0$; this is impossible because $E_h$ is effective and non-zero.

Therefore $E_h^{\perp}$ can contain at most one of the classes $\gamma_{\rho-1},\dotsc,\gamma_t\in\pi$. 
On the other hand, since $E_h\cdot \gamma_i\leq 0$ for every $i$,  $E_h^{\perp}$ must cut the cone $\langle \gamma_{\rho-1},\gamma_{\rho}\rangle$ along a face, and we conclude that $E_h^{\perp}\cap\{\gamma_{\rho-1},\dotsc,\gamma_t\}$ is either empty, or $\gamma_{\rho-1}$, or $\gamma_{\rho}$.
\end{prg}
\begin{prg}
 Recall that by \ref{freccia} $E_h^{\perp}$ must contain $\rho-2$ classes among $\gamma_1,\dotsc,\gamma_t$, and in case $(b)$ it must contain  $\rho-3$ classes among $\gamma_2\dotsc,\gamma_t$.
Thus 
 in case $(a)$ 
$E_h^{\perp}$ must contain either $\gamma_1,\dotsc,\gamma_{\rho-2}$, or  $\gamma_1,\dotsc,\check{\gamma}_i,\dotsc,\gamma_{\rho-2},\gamma_{\rho-1}$,
or $\gamma_1,\dotsc,\check{\gamma}_i,\dotsc,\gamma_{\rho-2},\gamma_{\rho}$, with $i\in\{1,\dotsc,\rho-2\}$; similarly in 
case $(b)$ 
$E_h^{\perp}$ must contain either $\gamma_2,\dotsc,\gamma_{\rho-2}$, or  $\gamma_2,\dotsc,\check{\gamma}_i,\dotsc,\gamma_{\rho-2},\gamma_{\rho-1}$,
or $\gamma_2,\dotsc,\check{\gamma}_i,\dotsc,\gamma_{\rho-2},\gamma_{\rho}$, with $i\in\{2,\dotsc,\rho-2\}$.

Again by \ref{freccia} the set of $\gamma_i$'s contained in $E_h^{\perp}$ must vary when $h$ varies, and by \ref{treno} 
we have $r\geq\rho$ in case $(a)$ and $r\geq\rho-1$ in case $(b)$. Therefore
there exists 
$h_0\in\{1,\dotsc,r\}$ such that  $E_{h_0}\cdot\gamma_{\rho-1}=0$.
\end{prg}
\begin{prg}
Consider an exceptional plane $P_{\rho-1}\subset D$ whose lines have class $\gamma_{\rho-1}$. 
We have $E_{h_0}\cdot\gamma_{\rho-1}=0$, therefore $P_{\rho-1}\not\subset E_{h_0}$
(otherwise $E_{h_0}\cdot\gamma_{\rho-1}<0$, see \S\ref{contraction}) and hence 
$P_{\rho-1}\cap E_{h_0}=\emptyset$.

We apply Lemma \ref{michi} to $D$, $E_{h_0}$, and $P_{\rho-1}\subset D$. Either we get the statement, or
 there exists an exceptional plane $N\subset E_{h_0}$ such that $C_D+C_{N}\equiv C_{E_{h_0}}+\gamma_{\rho-1}$ (note that $[C_N]\neq\gamma_{\rho-1}$).
We show that this last case leads to a contradiction.
\end{prg}
\begin{prg}
We have $D\cdot C_{N}=D\cdot(C_{E_{h_0}}+\gamma_{\rho-1}-C_D)=D\cdot \gamma_{\rho-1}+1\leq -2$ by \eqref{divano} and \ref{-3}, so that $N\subset D$ and
 $[C_N]=\gamma_j$ for some $j \in\{\rho,\dotsc,t\}$. Note that all the classes $\gamma_{\rho-1},\dotsc,\gamma_{t}$ have intersection $1$ with $-K_X$, so they  belong to the segment from $\gamma_{\rho-1}$ to $\gamma_{\rho}$ in the plane $\pi$, and there exists some $\lambda\in (0,1]$ such that
$$[C_N]=(1-\lambda)\gamma_{\rho-1}+\lambda\gamma_{\rho}.$$
Intersecting with $D$ we get
\begin{gather*}
D\cdot \gamma_{\rho-1}+1=D\cdot C_{N}=(1-\lambda) D\cdot\gamma_{\rho-1}+\lambda D\cdot\gamma_{\rho}=D\cdot \gamma_{\rho-1}+\lambda(D\cdot\gamma_{\rho}-
D\cdot\gamma_{\rho-1})\\
\text{and}\quad\lambda=\frac{1}{D\cdot \gamma_{\rho}-D\cdot\gamma_{\rho-1}}
\end{gather*}
so that $[C_N]$ is uniquely determined and 
 $[C_ {E_{h_0}}]=[C_D]+[C_N]-\gamma_{\rho-1}$ is uniquely determined too. This means that $E_{h_0}$ is the unique divisor in $\{E_1,\dotsc,E_r\}$ such that  $E_{h_0}\cdot\gamma_{\rho-1}=0$.
\end{prg}
\begin{prg}
In case $(a)$, we deduce that $r=\rho$, and up to renumbering we can assume that $h_0=\rho-1$ and:
\begin{enumerate}[$\bullet$]
\item
$E_i^{\perp}$ contains $\gamma_1,\dotsc,\check{\gamma}_i,\dotsc,\gamma_{\rho-2},\gamma_{\rho}$, for $i=1,\dotsc,\rho-2$;
\item $E_{\rho-1}^{\perp}$ contains  
$\gamma_2,\dotsc,\gamma_{\rho-2},\gamma_{\rho-1}$;
\item
$E_{\rho}^{\perp}$ contains $\gamma_1,\dotsc,\gamma_{\rho-2}$.
\end{enumerate}
Then $E_1^{\perp}$, $E_{\rho-1}^{\perp}$ and $E_{\rho}^{\perp}$ all contain the $\rho-2$ classes $\gamma_2,\dotsc,\gamma_{\rho-2},[C_D]$, which are linearly independent, because  $\gamma_2,\dotsc,\gamma_{\rho-2}$ are linearly independent and belong to $(-K_X+D)^{\perp}$, while 
 $(-K_X+D)\cdot C_D=1$. We deduce that the classes $[E_1]$, $[E_{\rho-1}]$ and $[E_{\rho}]$ must be linearly dependent, but this is impossible, because they generate one-dimensional faces of $\Eff(X)$.
\end{prg}
\begin{prg}
Case $(b)$ is similar:
we deduce that 
 $r=\rho-1$
and we find three distinct divisors $E_i$'s that have zero intersection with $C_D$, $C_F$, and $\rho-4$ among the $\gamma_j$'s; we conclude that these three divisors must have linearly dependent classes, which gives again a contradiction. This concludes the proof of Th.~\ref{main2}.
\qedhere
\end{prg}
\end{proof}
%\bibliographystyle{amsalpha}
%\bibliography{BiblioBreve}

\begin{thebibliography}{BCHM10}

\bibitem[ACO04]{occhettaGM}
M.~Andreatta, E.~Chierici, and G.~Occhetta, \emph{Generalized {M}ukai
  conjecture for special {F}ano varieties}, Cent.\ Eur.\ J.\ Math.\ \textbf{2}
  (2004), 272--293.

\bibitem[AW97]{AWaview}
M.~Andreatta and J.A. Wi{\'s}niewski, \emph{A view on contractions of higher
  dimensional varieties}, Algebraic Geometry - Santa Cruz 1995,
  Proc.~Symp.~Pure Math., vol.~62, 1997, pp.~153--183.

\bibitem[Bat99]{bat2}
V.V. Batyrev, \emph{On the classification of toric {F}ano 4-folds}, J.\ Math.\
  Sci.\ (New York) \textbf{94} (1999), 1021--1050.

\bibitem[BCHM10]{BCHM}
C.~Birkar, P.~Cascini, C.D. Hacon, and J.~McKernan, \emph{Existence of minimal
  models for varieties of log general type}, J.\ Amer.\ Math.\ Soc.\
  \textbf{23} (2010), 405--468.

\bibitem[Cas08]{fanos}
C.~Casagrande, \emph{Quasi-elementary contractions of {F}ano manifolds},
  Compos.\ Math.\ \textbf{144} (2008), 1429--1460.

\bibitem[Cas12]{codim}
\bysame, \emph{On the {P}icard number of divisors in {F}ano manifolds},
  Ann.~Sci.~{\'E}c.~Norm.~Sup{\'e}r.\ \textbf{45} (2012), 363--403.

\bibitem[Cas13a]{eff}
\bysame, \emph{\noop{aaa}{O}n the birational geometry of {F}ano 4-folds},
  Math.~Ann.\ \textbf{355} (2013), 585--628.

\bibitem[Cas13b]{cdue}
\bysame, \emph{\noop{zzz}{N}umerical invariants of {F}ano 4-folds},
  Math.~Nachr.\ \textbf{286} (2013), 1107--1113.

\bibitem[Cas17]{blowup}
\bysame, \emph{{F}ano 4-folds, flips, and blow-ups of points}, J.~Algebra
  \textbf{483} (2017), 362--414.

\bibitem[Cas20]{fibrations}
\bysame, \emph{Fano 4-folds with rational fibrations}, Algebra Number Theory
  \textbf{14} (2020), 787--813.

\bibitem[CCF19]{vb}
C.~Casagrande, G.~Codogni, and A.~Fanelli, \emph{The blow-up of {$\pr^4$} at
  {$8$} points and its {F}ano model, via vector bundles on a del {P}ezzo
  surface}, Rev.\ M{\'a}t.\ Complut.\ \textbf{32} (2019), 475--529.

\bibitem[CR22]{delta3_4folds}
C.~Casagrande and E.A. Romano, \emph{Classification of {F}ano 4-folds with
  {L}efschetz defect 3 and {P}icard number 5}, J.\ Pure Appl.\ Algebra
  \textbf{226} (2022), no.~3, 13 pp.

\bibitem[Deb01]{debarreUT}
O.~Debarre, \emph{Higher-dimensional algebraic geometry}, Universitext,
  Springer-Verlag, 2001.

\bibitem[{Del}14]{gloria}
G.~{Della Noce}, \emph{\noop{aaa}{O}n the {P}icard number of singular {F}ano
  varieties}, Int.\ Math.\ Res.\ Not.\ \textbf{2014} (2014), 955--990.

\bibitem[Ewa96]{ewald2}
G.~Ewald, \emph{Combinatorial convexity and algebraic geometry}, Graduate Texts
  in Mathematics, vol. 168, Springer-Verlag, 1996.

\bibitem[HK00]{hukeel}
Y.~Hu and S.~Keel, \emph{Mori dream spaces and {GIT}}, Michigan Math.\ J.\
  \textbf{48} (2000), 331--348.

\bibitem[IP99]{fanoEMS}
V.A. Iskovskikh and Yu.G. Prokhorov, \emph{Algebraic geometry {V} - {F}ano
  varieties}, Encyclopaedia Math.\ Sci.\, vol.~47, Springer-Verlag, 1999.

\bibitem[Kaw89]{kawsmall}
Y.~Kawamata, \emph{Small contractions of four dimensional algebraic manifolds},
  Math.\ Ann.\ \textbf{284} (1989), 595--600.

\bibitem[KM98]{kollarmori}
J.~Koll{\'a}r and S.~Mori, \emph{Birational geometry of algebraic varieties},
  Cambridge Tracts in Mathematics, vol. 134, Cambridge University Press, 1998.

\bibitem[Kol96]{kollar}
J.~Koll{\'a}r, \emph{Rational curves on algebraic varieties}, Ergebnisse der
  Mathematik und ihrer Grenzgebiete, vol.~32, Springer-Verlag, 1996.

\bibitem[KPS18]{KPS}
A.G. Kuznetsov, Yu.G. Prokhorov, and C.A. Shramov, \emph{Hilbert schemes of
  lines and conics and automorphism groups of {F}ano threefolds},
  Jpn.~J.~Math.\ \textbf{13} (2018), 109--185.

\bibitem[MM86]{morimukai2}
S.~Mori and S.~Mukai, \emph{Classification of {F}ano {$3$}-folds with $b_2\geq
  2$, {I}}, Algebraic and Topological Theories -- to the memory of Dr.\
  Takehiko Miyata (Kinosaki, 1984), Kinokuniya, Tokyo, 1986, pp.~496--545.

\bibitem[Oka16]{okawa_MCD}
S.~Okawa, \emph{On images of {M}ori dream spaces}, Math.\ Ann.\ \textbf{364}
  (2016), 1315--1342.

\bibitem[Sec21]{saverio}
S.A. Secci, \emph{Fano 4-folds having a prime divisor of {P}icard number 1},
  preprint arXiv:2103.16140, 2021.

\bibitem[Wi{\'s}91]{wisn}
J.A. Wi{\'s}niewski, \emph{On contractions of extremal rays of {F}ano
  manifolds}, J.~Reine Angew.~Math.\ \textbf{417} (1991), 141--157.

\end{thebibliography}
\providecommand{\noop}[1]{}
\providecommand{\bysame}{\leavevmode\hbox to3em{\hrulefill}\thinspace}
\providecommand{\MR}{\relax\ifhmode\unskip\space\fi MR }
\providecommand{\MRhref}[2]{%
  \href{http://www.ams.org/mathscinet-getitem?mr=#1}{#2}
}
\providecommand{\href}[2]{#2}

\end{document}